\documentclass[11pt, oneside]{article}   	
\newcommand{\Title}{Rigidity for Markovian Maximal Couplings of Elliptic Diffusions}

\usepackage[totalwidth=480pt,totalheight=680pt]{geometry}                		
\usepackage[british]{babel}

\usepackage[usenames,dvipsnames]{color}
\usepackage[colorlinks,citecolor=BlueViolet,linkcolor=Brown]{hyperref}

\usepackage{graphicx}				
\usepackage{natbib}

\usepackage{amssymb}
\usepackage{amsmath}
\usepackage{amsthm}

\newcommand{\x}{\mathbf{x}}
\newcommand{\y}{\mathbf{y}}
\newcommand{\z}{\mathbf{z}}

\newcommand{\Expect}[1]{\operatorname{\mathbb{E}}\left[#1\right]}
\newcommand{\Identity}{\operatorname{\mathbb{I}}}
\newcommand{\Law}[1]{\operatorname{\mathcal{L}}\left(#1\right)}
\newcommand{{\LPC}}{\textbf{LPC}}
\newcommand{\Prob}[1]{\operatorname{\mathbb{P}}\left[#1\right]}
\newcommand{\Mirrors}{\mathcal{H}}
\newcommand{\Reals}{\mathbb{R}}

\newcommand{\ball}{\operatorname{\mathcal{B}}}
\renewcommand{\d}{\operatorname{d}}
\newcommand{\dist}{\operatorname{dist}}
\newcommand{\eps}{\varepsilon}
\newcommand{\generator}{L}

\newtheorem{thm}{Theorem}
\newtheorem{lem}[thm]{Lemma}
\newtheorem{defn}[thm]{Definition}
\newtheorem{rem}[thm]{Remark}
\newtheorem{cor}[thm]{Corollary}

\title{\Title}
\author{Sayan Banerjee and Wilfrid S. Kendall}
\date{\today}							

\begin{document}

\maketitle
\footnotetext[0]{This work was supported by EPSRC Research Grant EP/K013939.}

\begin{abstract}\small
Maximal couplings are (probabilistic) couplings of Markov processes such that the tail probabilities of the coupling time attain the total variation lower bound (Aldous bound) uniformly for all time. 
Markovian (or immersion) couplings are couplings defined by strategies where neither process is allowed to look into the future of the other before making the next transition. 
Markovian couplings are typically easier to construct and analyze than general couplings, and play an important r\^ole in many branches of probability and analysis. 
\citeauthor{HsuSturm-2013}, in a preprint circulating in 2007, but later published in \citeyear{HsuSturm-2013},
proved that the \emph{reflection-coupling} of Brownian motion is the \emph{unique Markovian maximal coupling (MMC)} of Brownian motions starting from two different points. 
Later, \citet{Kuwada-2009} proved that the existence of a MMC for Brownian motions on a Riemannian manifold enforces existence of a reflection structure on the manifold.

In this work, we investigate suitably regular elliptic diffusions on  
manifolds,
and show how consideration of the diffusion geometry (including dimension of the isometry group and flows of isometries) is fundamental in classification of the space and the generator of the diffusion for which an MMC exists,
especially when the MMC also holds under local perturbations of the starting points for the coupled diffusions. 
We also describe such diffusions in terms of \emph{Killing vectorfields} (generators of isometry groups) and \emph{dilation vectorfields} (generators of scaling symmetry groups). 
This permits a complete characterization of those possible manifolds and their diffusions for which there exists a MMC under local perturbations of the starting points of the coupled diffusions.
For example, in the time-homogeneous case it is shown that the only possible manifolds that may arise are Euclidean space, hyperbolic space and the hypersphere.
Moreover the permissible drifts can then derive only from rotation isometries of these spaces (and dilations, in the Euclidean case).
In this sense, a geometric rigidity phenomenon holds good.
\end{abstract}

\\
\emph{Key words and phrases:} 
\\
\textsc{characteristic operator;
coupling;
diffusion-geodesic completeness;
diffusion geometry;
diffusion matrix;
elliptic diffusion;
global isometry group;
homogeneous space;
immersion coupling;
infinitesimal generator;
involutive isometry;
Killing vectorfield;
Laplace-Beltrami operator;
large deviations;
local perturbation condition;
Markovian coupling;
maximal coupling;
maximally symmetric space;
orthonormal frame bundle;
Riemannian manifold;
reflection coupling;
rigidity;
stochastic completeness;
stochastic differential equation;
stochastic parallel transport;
Strato\-novich differential;
strong maximum principle;
Topogonov comparison theorem;
totally geodesic submanifold.}

\section{Introduction}\label{sec:rigidity}

Let $(\Omega_1, \mathcal{F}_1,\mu_1)$ and $(\Omega_2, \mathcal{F}_2,\mu_2)$ be two probability spaces. 
A (probabilistic) \emph{coupling} of  $\mu_1$ and $\mu_2$ is a measure $\mu$ on the product measurable space $(\Omega_1 \times \Omega_2, \mathcal{F}_1 \times \mathcal{F}_2)$ with marginals $\mu_1$ and $\mu_2$. 
This paper considers the question of coupling of (the laws of) two realizations $X$ and $Y$ of a Markov process on some state space $S$. 
We distinguish two important classes.
The first class (thematic for the foundational theory of probabilistic coupling) consists of couplings where, with positive probability, 
$X$ and $Y$ can stick together and move as a single process after some random time
\[
 \tau\quad=\quad
 \inf\{s>0: X_t=Y_t \text{ for all } t>s\}\,; 
\]
here \(\tau\) is called the \emph{coupling time}.  
The other class consists of couplings (\emph{Shy Couplings}) where the two processes $X$ and $Y$ remain separated by at least a fixed positive distance \(\eps\) for all time. 
Recent investigations of the second class of couplings can be found in \citet{BenjaminiBurdzyChen-2007} and \citet{BramsonBurdzyKendall-2013,BramsonBurdzyKendall-2014}; 
in this article, we concentrate on the first class. 

Probabilistic coupling is a central technique of modern probability theory \citep{Lindvall-1992,Thorisson-2000}. 
Attention naturally focusses on a fundamental question: how fast can we make $X$ and $Y$ meet? 
This has direct relevance, for example to the study of probabilistic algorithms and to gradient estimates for harmonic functions, and is also very valuable in eliciting the range of possibilities for coupling constructions.
Mathematically, this amounts to constructing couplings where $\Prob{\tau>t}$ is minimised for all time $t$. 
The Aldous inequality states that, for any $t>0$,
\begin{equation}\label{eq:Aldous}
\Prob{\tau>t} \quad\geq\quad \|\mu_{1,t}-\mu_{2,t}\|_{TV}\,, 
\end{equation}
where $\mu_{1,t}$ and $\mu_{2,t}$ are the distributions of $X_t$ and $Y_t$ respectively, while 
\[
\|\nu\|_{TV}\quad=\quad\sup\{|\nu(A)|\;:\;\text{measurable }A\} 
\]
denotes the total variation norm on signed measures \(\nu\). 
Thus a maximally efficient possible coupling (a \emph{Maximal Coupling}) would attain equality in the Aldous inequality \eqref{eq:Aldous} for all times $t>0$,
thus solving a multi-objective optimization problem. 
The remarkable construction of \citet{Griffeath-1974}, later simplified in a most elegant way by \citet{Pitman-1976}, shows that maximal couplings always exist for discrete Markov chains.
\citet{Goldstein-1979} generalized the construction to the case of non-Markovian processes; \citet{SverchkovSmirnov-1990} generalized it to continuous-time c\`adl\`ag processes. 
Here is a summary of the \citeauthor{Pitman-1976} approach, which is a model for the construction below (in Subsection \ref{sec:genprop}) of maximal couplings of diffusions.
A deterministic time-varying interface is constructed using the transition probabilities of the diffusions which are to be coupled.
The distribution of the coupling time is elicited using the deficits of the transition probability masses integrated on each side of the interface 
(at any particular time, these deficits are equal and correspond to the probability of one, equivalently both, of the coupled processes hitting the interface at this time).
Now, the coupling time is sampled from this distribution, and the coupling location corresponds to a point on the interface at this time.
Finally, the coupling is realized by constructing a single process forward in time and time-reversed time-inhomogeneous diffusions connecting starting locations to the location and moment of coupling,
conditioning to avoid hitting the interface prematurely.

The major drawback of all these constructions is they are typically very implicit; in most cases, it is extremely hard, if not impossible, to make detailed calculations for such couplings. This is a strong motivation for considering Markovian couplings, which we now describe.

 Let $X$ and $Y$ be Markov processes starting from $x_0$ and $y_0$ respectively. Let $\mathcal{F}_s=\sigma\{(X_{s'},Y_{s'}): s' \le s\}$ denote the joint filtration generated by \(X\) and \(Y\) together up to time $s$. A coupling of $X$ and $Y$ is called \emph{Markovian} if the joint process 
\[
 \{(X_{t+s},Y_{t+s}) : t \ge 0\} \text{ conditioned on } \mathcal{F}_s
 \]
is again a coupling of the laws of $X$ and $Y$, but now starting from $(X_s,Y_s)$. 
(An alternative martingale-based characterization makes a succinct connection to the theory of immersions of filtrations. 
For this reason Markovian couplings are also called \emph{immersion couplings}: \citealp{Kendall-2013a})
 
A natural and immediate question is, when can a maximal coupling of two diffusions be Markovian? The standard (and elegant) example in the literature is the \emph{reflection-coupling} of Euclidean Brownian motions starting from two different points: 
the second Brownian path is obtained from the first by reflecting the first path on the hyperplane bisecting the line joining the starting points 
until the first path (equivalently, the second, reflected, path) hits this hyperplane. 
Both paths then evolve together (``synchronously'') as a single Brownian path. 
Straightforward calculations, based on the reflection principle, show that this construction is in fact a Markovian maximal coupling (MMC). 
Furthermore, \citet{HsuSturm-2013} proved that this is the \emph{unique} such coupling for Euclidean Brownian motion. 
A few other examples are discussed in the literature:
Ornstein Uhlenbeck processes \citep{Connor-2007a}, also Brownian motion on manifolds which possess certain reflection symmetries. 
The reflection coupling idea manifests itself throughout the area of probabilistic coupling: 
for example it has a natural generalization to Brownian motion on Riemannian manifolds \citep{Kendall-1986b,Cranston-1991}, involving stochastic parallel transport and development, and not requiring any symmetries of the manifold.
However it seems unlikely that such generalizations will normally provide maximal couplings.
\citet{Kuwada-2009} investigated this question for Brownian motion on manifolds (and their generalisations to metric spaces). 
Under suitable mild regularity assumptions he showed that a reflection symmetry of the space is \emph{necessary} for the existence of a Markovian maximal coupling of two Brownian motions started from a specified pair of points.
Working under some further assumptions, he proved that the fixed point set of the symmetry (the ``mirror'', characterizing this isometry) does not change with time;
the maximal coupling is given simply by reflecting one process onto the other using the reflection symmetry defined by this mirror.

The aim of this paper is to develop the results of \citeauthor{Kuwada-2009} to the case of general regular elliptic diffusions with smooth coefficients.
It will be shown that Markovian maximal couplings are rare, in the sense that a stable local existence result enforces extreme global symmetry on the manifold: a kind of rigidity result.
Section \ref{sec:Euclidean} considers implications of existence of Markovian maximal couplings for \(d\)-dimensional Euclidean diffusions 
(``Euclidean'' here meaning that the diffusion matrix is the identity matrix),
under rather general regularity assumptions on the (possibly time-inhomogeneous) drift.
Extending Kuwada's argument, the existence of an MMC implies there is a mirror symmetry between the coupled processes at any given time.
However the influence of the non-zero drift now means that the mirror can vary deterministically with time, 
making the coupled dynamics considerably more complicated. 
We study the evolution of the mirror in time using stochastic calculus and we obtain a functional equation that the drift must satisfy for a Markovian maximal coupling to exist. 
This equation can be used to characterise all time-inhomogeneous diffusions which admit such couplings.

In the time-homogeneous case the characterization can be refined under the additional hypothesis that there is also a Markovian maximal coupling under local perturbation of the starting points,
which is to say, Markovian maximal couplings exist locally in a stable sense:
\begin{defn}[Local Perturbation Condition ({\LPC})]\label{def:LPC}
 There is $r>0$, and initial points \(\x_0\) and \(\y_0\),
 such that there exists a Markovian maximal coupling of the diffusion processes $X$ and $Y$ starting from $\x$ and $\y$ for every $\x \in \ball(\x_0,r)$ and $\y \in \ball(\y_0,r)$, 
 where \(\ball(\x_0,r)\) is the open metric ball centred at \(\x_0\) and of radius \(r\). 
\end{defn}
We will show that, for any dimension \(d\geq1\), {\LPC} holds for a suitably regular Euclidean diffusion with time-homogenous drift
if and only if the drift takes the form $\mathbf{b}(\x)=\lambda\x + T\x + \mathbf{c}$, where $\lambda$ is a scalar, $T$ is a skew-symmetric matrix and $\mathbf{c}$ is a fixed vector. 
This implies that \emph{Brownian motion with constant drift} and \emph{Ornstein-Uhlenbeck process} are the only \emph{one-dimensional} examples of time-homogeneous diffusions 
for which there are successful Markovian maximal couplings from arbitrary pairs of starting points. 
In higher dimensions, for regular Euclidean diffusions under {\LPC}, essentially the same is true except that the drift may also include a rotational component.
In one dimension, even without {\LPC}, it turns out that a Markovian maximal coupling exists between two copies of a regular diffusion started from $x_0$ and $y_0$ 
if and only if the drift is either affine or an odd function around the midpoint of the starting points. 

Section \ref{sec:manifold} considers Markovian maximal couplings of Brownian motion with time-homogeneous drift on a complete Riemannian manifold $M$ under {\LPC}. 
This is the natural generalization of the context of Section \ref{sec:Euclidean}, 
since a regular elliptic diffusion on Euclidean space furnishes the space with a Riemannian metric by means of inverting the diffusion matrix, 
and then the diffusion is converted into a Brownian motion with drift on the resulting Riemannian manifold, so that the Riemannian geometry serves to classify a variety of diffusions
(compare the rather similar r\^ole of Fisher information in theoretical statistics). 
We assume that the elliptic diffusion is stochastically complete, and also diffusion-geodesically complete, in the sense that the diffusion Riemannian geometry is geodesically complete.
Strikingly, {\LPC} then produces a geometric rigidity phenomenon, 
namely a complete classification of the space $M$ as one of the three model spaces $\Reals^d$ (Euclidean space), $\mathbb{S}^d$ (Sphere) and $\mathbb{H}^d$ (Hyperbolic space) 
depending upon the sign of the (necessarily constant) curvature $K$ (see Theorem \ref{thm:Mclass} in Section \ref{sec:manifold}). 
The Euclidean case is fully covered in Section \ref{sec:Euclidean}, 
and delivers the necessary ideas and techniques which we generalise to the manifold setup in Section \ref{sec:manifold} to study Markovian maximal couplings on the other two spaces. 
It turns out that the only drifts which can yield Markovian maximal couplings are given by the \emph{Killing vectorfields}, 
defined as infinitesimal generators for the rigid motion group (namely, generators of one-parameter subgroups of isometries).

In this paper we confine our considerations to the case of elliptic diffusions, where there is a strong connection to Riemannian geometry,
and path-continuity permits the formation of interfaces of co-dimension \(1\) separating pairs of initial points. 
Possible extensions to hypoelliptic diffusions or to general Markov chains are potentially of great interest, but we leave these questions as topics for future work.

\subsection{Markovian maximal couplings: general properties}\label{sec:genprop}
We complete this introduction by defining some general notation and by describing some basic general properties of Markovian maximal couplings for general Markov processes on a metric space $(M,\dist)$.
\citet{Kuwada-2009} derived results similar to Lemmas \ref{lem:disjsupp} and \ref{lem:flow} below.
For the sake of clearer exposition, and as we are primarily interested in diffusion processes, we will state the results for continuous-time Markov processes. 
Denote the Markov process under consideration by $X$.

We assume that the metric space supports a positive Borel measure $m$ with $0<m(B)<\infty$ for any metric ball $B$ of finite radius. 
Consequently, the closed support of $m$ is the whole of $M$. 
We further assume that for any $t>s \ge0 $, the conditional distribution law $\Law{X_t \mid X_s=x}$ is absolutely continuous with respect to $m$ and has a probability kernel density 
given by $p(s,\x;t,\z)$ for \(\x\), \(\z \in M\) and \(0\leq s<t\). 

Let $\mu$ denote the law of a Markovian maximal coupling $(X,Y)$ of two copies of our Markov process started from $(\x_0, \y_0)$, 
which can be thought of as a measure on the coupled path-space $C[0,\infty)^2$, and let 
\[
\tau \quad=\quad 
\inf\{s>0: X_t=Y_t \text{ for all } t>s\}
\]
denote the coupling time of $X$ and $Y$. 

Motivated by \citeauthor{Pitman-1976}'s construction for finite Markov chains, we write 
\[
\alpha(s,\x,\y,t,\z)\quad=\quad p(s,\x; t,\z)-p(s,\y;t,\z)\,, 
\]
and set $\alpha^+(s,\x,\y,t,\z)=\max(\alpha(s,\x,\y,t,\z),0)$ and $\alpha^-(s,\x,\y,t,\z)=\max(-\alpha(s,\x,\y,t,\z),0)$.
If $s=0$ (and thus $\x=\x_0$ and $\y=\y_0$), then we abbreviate $\alpha(t,\z)$ for $\alpha(s,\x_0,\y_0,t,\z)$ and similarly for other quantities.

We will be dealing with Markov processes which are possibly time-inhomogeneous, 
so we say a Markov process starts from $(t,\x)$ if we are looking at the distribution law $\Law{\theta_tX \mid X_t=\x}$, 
where $\theta$ denotes the time-shift operator given by $(\theta_tX)_s=X_{t+s}$. 

Define the \emph{interface} between $p(0,\x_0;\cdot,\cdot)$ and $p(0,\y_0;\cdot,\cdot)$ at time $t$ to be the region where the corresponding heat kernels agree:
\begin{equation}\label{eq:interfaceset}
I(\x_0,\y_0,t)\quad=\quad
\{\z \in M \;:\; p(0,\x_0;t,\z)=p(0,\y_0;t,\z)\}\,.
\end{equation}
Also write
\begin{eqnarray}\label{eqnarray:int}
I^{-}(\x_0,\y_0,t)\quad&=\quad&\{\z \in M\;:\; p(0,\x_0;t,\z)>p(0,\y_0;t,\z)\},\nonumber\\
I^{+}(\x_0,\y_0,t)\quad&=\quad&\{\z \in M\;:\; p(0,\x_0;t,\z)< p(0,\y_0;t,\z)\}.
\end{eqnarray}
Finally, define the perpendicularly bisecting set (or ``hyperplane'') and the associated ``half-spaces'' (note that these are indeed a hyperplane and half-spaces in the Euclidean case):
\begin{eqnarray}\label{eqnarray:H}
H(\x,\y)\quad&=\quad&\{\z \in M\;:\; \dist(\x,\z)=\dist(\y,\z)\},\nonumber\\ 
H^-(\x, \y)\quad&=\quad&\{\z \in M\;:\; \dist(\x,\z)<\dist(\y,\z)\} \,,\nonumber\\
H^+(\x,\y)\quad&=\quad&\{\z \in M\;:\; \dist(\x,\z)>\dist(\y,\z)\} \,.
\end{eqnarray} 

\begin{lem} \label{lem:disjsupp}
Any joint maximal coupling law can be related to differences of the transition probability kernel densities as follows: for any Borel subset \(A\) of $M$, and \(s>0\),
\begin{align*}
\mu(X_s \in A, \tau>s) \quad&=\quad \int_A \alpha^+(s,\x)m(\d\x)\,,\\
\mu(Y_s \in A, \tau>s) \quad&=\quad \int_A \alpha^-(s,\x)m(\d\x)\,.
\end{align*}
\end{lem}
\begin{proof}
It is immediate that \(\mu(X_s\in A, \tau \le s)\leq \mu(X_s\in A)\). If \(p(0,\x_0;s,\cdot)  \leq p(0,\y_0;s,\cdot)\) on \(A\) then
\begin{align*}\label{align:disupp0} 
\mu(X_s=Y_s\in A, \tau \le s) 
\quad&=\quad
\mu(X_s\in A, \tau \le s)
\quad \le \quad\mu(X_s\in A) \nonumber\\
\quad & =\quad \int_A p(0,\x_0;s,\x)  m(\d\x)\quad = \quad \int_A p(0,\x_0;s,\x)  \wedge p(0,\y_0;s,\x)m(\d\x)\,.
\end{align*}
Interchanging the r\^oles of \(X\) and \(Y\), a corresponding argument applies if \(p(0,\x_0;s,\cdot)  \geq p(0,\y_0;s,\cdot)\) on \(A\). 
Hence additivity shows that for all \(A\) the coupling must satisfy
\begin{equation}\label{align:disupp} 
\mu(X_s=Y_s\in A, \tau \le s) \quad\leq\quad \int_A p(0,\x_0;s,\x)  \wedge p(0,\y_0;s,\x)m(\d\x)\,.
\end{equation}
Finally,  Aldous' inequality \eqref{eq:Aldous} is by definition an equality for a maximal coupling, so
 \begin{equation}\label{eq:Aldous-equality}
 \mu(\tau \le s)\quad=\quad \int_{\Reals^d} p(0,\x_0;s,\x)  \wedge p(0,\y_0;s,\x)m(\d\x).
 \end{equation}
It follows that the inequality (\ref{align:disupp}) must in fact be an equality. This proves the lemma.
\end{proof}

Only maximality was required for Lemma \ref{lem:disjsupp}. 
If in addition $\mu$ is Markovian, then the conditional law $\Law{\theta_sX, \theta_sY \mid \mathcal{F}_s}$ describes a Markovian coupling of two copies of our Markov process starting from $((s,X_s), (s,Y_s))$. 
Such a coupling therefore satisfies the following \emph{flow property}:
 \begin{lem}\label{lem:flow}
If $\mu$ is a Markovian maximal coupling and $\mu_s=\Law{X_s, Y_s}$ then, for $\mu_s$-almost every $(\x,\y)$ with $\x \neq \y$
the conditional law $\Law{\theta_sX, \theta_sY \mid X_s=\x, Y_s=\y}$ gives a Markovian maximal coupling of $(X,Y)$ starting from $((s,\x),(s,\y))$.
 \end{lem}
 \begin{proof}
 This follows immediately from the maximality of $\mu$ and the fact that $\mu$ is Markovian.
 \end{proof}

We now introduce notation to describe the set of pairs of initial points in the closed support of \(\mu_s\)
for which the forward processes $(\theta_s X, \theta_s Y)$ do indeed generate a maximal coupling:
\begin{align*}
\mathcal{M}(\mu_s)\quad&=\quad\{(\x,\y) \in \operatorname{Support}(\mu_s)\;:\; \x \neq \y
\text{ and }\Law{\theta_sX, \theta_sY \mid X_s=\x, Y_s=\y} \text{ yields } \\
&\qquad\qquad
\text{ a maximal coupling of }(X,Y) \text{ starting from }((s,\x),(s,\y))\}. 
\end{align*}


We conclude this introduction by noting an elementary observation about couplings of Markov processes.
\begin{lem}\label{lem:mapsto}
For each $t \ge 0$, let $F_t:(\Omega_1,\mathcal{F}_1) \rightarrow (\Omega_2,\mathcal{F}_2)$ be a bijective mapping between two measurable spaces such that $F_t, F_t^{-1}$ are measurable.
Then, for any Markov process $\{X_t: t \ge 0\}$ on $\Omega_1$, $\{F_t(X_t): t \ge 0\}$ defines a Markov process on $\Omega_2$. 
Furthermore $\{(X_t,Y_t): t \ge 0\}$ is a (Markovian) maximal coupling of Markov processes on $\Omega_1$ if and only if $\{(F_t(X_t),F_t(Y_t)): t \ge 0\}$ is a (Markovian) maximal coupling on $\Omega_2$.
\end{lem}
\begin{proof}
The first assertion is a direct consequence of the general definition of conditional expectation. The second assertion follows from the definition of maximality.
\end{proof}


\section{Markovian Maximal Couplings on Euclidean spaces}\label{sec:Euclidean}
We consider 
diffusions on Euclidean space \(\Reals^d\) with infinitesimal generator
\begin{equation}\label{eq:generator}
\generator\quad=\quad\frac{1}{2}\sum_{i=1}^d\partial_i^2+\sum_{i=1}^d b_i(t,\x)\partial_i\,,
\end{equation}
where $\displaystyle{\partial_i = \frac{\partial}{\partial x_i}}$.
In the following, $X$ will be used to denote a diffusion with the above generator.
We will refer below to such a diffusion as a \emph{Euclidean diffusion}, 
because diffusions with general diffusion coefficients are covered in Section \ref{sec:manifold} as instances of `Brownian motion plus drift on a manifold'. 
We make the following very general regularity assumptions (not necessary for all of our results, but imposed globally to streamline the exposition):
\begin{itemize}
\item[(A1)] The drift vectorfield $\mathbf{b}: [0,\infty) \times \Reals^d\rightarrow \Reals$ is continuously differentiable in the second (space) variable, 
moreover \(\mathbf{b}\) and all its first-order spatial partial derivatives \(\partial_i\mathbf{b}\) are bounded on compact subsets of $[0,\infty) \times \Reals^d$.
\item[(A2)] For every $t>s\ge0$, and $\x,\z \in \Reals^d$, the conditional distribution law $\Law{X_t \mid X_s=x}$ is 
the law of a diffusion with transition probability density kernel $p(s,\x;t,\z)$ (density with respect to Lebesgue measure), which is jointly continuous in all its arguments.
Moreover, $p(s,\cdot;\cdot, \cdot)$ is positive everywhere when \(s>0\). 
Finally, the density $p(s,\x;\cdot,\cdot): \Reals^+ \times\Reals^d\rightarrow \Reals$ is continuously differentiable in the time variable (first unspecified variable)
and twice continuously differentiable in the space variable (second unspecified variable).


\end{itemize}

\begin{rem}
Note that Assumption (A2) implies that the diffusion does not explode in finite time (otherwise $p(s,\x;t,\cdot)$ would determine a sub-probability density).
A sufficient condition for non-explosion is to require that $\mathbf{b}$ is locally Lipschitz in the space variable $\x$ (which follows from Assumption (A1)) 
and moreover that there exists a constant $C$ such that $|b(t,\x)| \le C(1+|t| +|\x|)$ for all $(t,\x) \in [0,\infty) \times \Reals^d$
\citep[Proposition 1.1.11]{Hsu-2002}. Furthermore, the fact that $\mathbf{b}$ is locally Lipschitz in $\x$ implies the existence of a unique strong solution to the SDE corresponding to (\ref{eq:generator}) for any given driving Brownian motion $B$ \citep[Theorem 1.1.8]{Hsu-2002}.
\end{rem}

We will sometimes say $\mathbf{b}$ satisfies Assumptions (A1) and (A2) if $\mathbf{b}$ satisfies (A1) and the corresponding diffusion (whose law is unique by the above remark) has transition probability densities satisfying (A2).\\\\
Recall that we say a diffusion starts from $(t,\x)$ if we are looking at the law 
$\Law{\theta_tX \mid X_t=\x}$, 
where $\theta$ denotes the time-shift operator given by $(\theta_tX)_s=X_{t+s}$. The resulting process is a diffusion with the identity diffusion matrix 
but using time-shifted drift $b(t+\cdot,\cdot)$ and starting from $x$ at time $0$.

Let $X$ and $Y$ be two copies of this diffusion starting from $\x_0$ and $\y_0$ respectively.
Recall
\begin{align*}
\mathcal{M}(\mu_s)\quad&=\quad\{(\x,\y) \in \operatorname{Support}(\mu_s)\;:\; \x \neq \y
\text{ and }\Law{\theta_sX, \theta_sY \mid X_s=\x, Y_s=\y} \text{ yields } \\
&\qquad\qquad
\text{ a maximal coupling of }(X,Y) \text{ starting from }((s,\x),(s,\y))\}. 
\end{align*}

\begin{rem}
The function $(s,\x) \mapsto p(0,\x_0;t-s,\x)$ satisfies a backward parabolic equation. 
Therefore uniqueness theory for such equations yields that there does not exist any $s>0$ such that $p(0,\x_0; s,\z)=p(0,\y_0; s,\z) \text{ for all } \z \in \Reals^d$. 
This, along with (\ref{eq:Aldous-equality}), implies that, for every $s>0$, $\mu(\tau>s)>0$ and thus $\mu(\mathcal{M}(\mu_s))>0$. 
In particular, $\mathcal{M}(\mu_s)$ is non-empty for each $s>0$.
\end{rem}

\subsection{Coupling and the interface}\label{subsection:interface}

Here, we show that the existence of a Markovian maximal coupling for $X$ and $Y$ implies that for each time $t$, the interface $I(\x_0,\y_0,t)$ will be a hyperplane bisecting the straight line joining $X_t$ and $Y_t$.

We begin with some preparatory lemmas.
Note that Brownian motion has fluctuations which are of order $O(\sqrt{t})$ while fluctuations resulting from the drift are of order $O(t)$. 
Thus, \emph{on small time scales}, the Brownian behaviour should dominate. 
The following lemma substantiates this intuition.

\begin{lem}\label{lem:LD}
Let $X$ be a diffusion given by 
\[
X_t\quad=\quad B_t+ \int_0^t \mathbf{b}(s,X_s)\d s\,,
\]
with $X_0=\x_0$ (so \(B_0=\x_0\)), and suppose the drift $\mathbf{b}$ satisfies Assumption (A1). Denote by $\mathbb{P}$ the underlying measure. Then, for any $\z \in \Reals^d$ and any $\delta>0$,
\begin{equation}\label{eq:LD}
\lim_{t \downarrow 0} \; t \log \frac{\Prob{X_t \in \ball(\z,\delta)}}{\Prob{B_t \in \ball(\z,\delta)}}\quad=\quad0\,.
\end{equation}
\end{lem}
\begin{proof}
Let $I=\sup\{|\y-\x_0|: \y \in \ball(\z,\delta)\}$ and choose $N>d\times I+1$. 
By continuity of $\mathbf{b}$, there is a finite $M$ for which $|\mathbf{b}(t,\y)| \le M$ for all $(t,\y) \in [0,1] \times \ball(\x_0, N)$.

Let $\tau_N= \inf\{t>0: X_t \not\in \ball(\x_0, N)\}$. 
Then, we can write
\begin{equation}
\Prob{X_t \in \ball(\z,\delta)}\quad=\quad
\Prob{X_t \in \ball(\z,\delta), \tau_N > t}+\Prob{X_t \in \ball(\z,\delta), \tau_N \le t}\,.
\end{equation}
Now  $|X_{t \wedge \tau_N}-B_{t \wedge \tau_N}|  \le Mt$. 
We pick ${t \le \min\{\tfrac{1}{M},\tfrac{\delta}{M}\}}$. Then
\begin{equation}\label{eq:UB}
\frac{\Prob{X_t \in \ball(\z,\delta)}}{\Prob{B_t \in \ball(\z,\delta)}} 
\quad\le\quad \frac{\Prob{B_t \in \ball(\z,\delta+Mt)}}{\Prob{B_t \in \ball(\z,\delta)}}+\frac{\Prob{\tau_N\le t}}{\Prob{B_t \in \ball(\z,\delta)}}
\end{equation}
and (using \(t<\delta/M\))
\begin{equation}\label{eq:LB}
\frac{\Prob{X_t \in \ball(\z,\delta)}}{\Prob{B_t \in \ball(\z,\delta)}} 
\quad\ge\quad 
\frac{\Prob{B_t \in \ball(\z,\delta-Mt)}}{\Prob{B_t \in \ball(\z,\delta)}}-\frac{\Prob{\tau_N\le t}}{\Prob{B_t \in \ball(\z,\delta)}}\,.
\end{equation}
Also (using \(t<1/M\) to control the difference between \(B\) and \(X\))
\begin{align*}
\Prob{\tau_N \le t} \quad& \le\quad \Prob{\sup_{s \le t}|B_s-\x_0| > N-1}\\
\quad& \le\quad
\frac{4d^2\sqrt{t}}{\sqrt{2 \pi}(N-1)}\exp\left(-\frac{(N-1)^2}{2td^2}\right)\,.
\end{align*}
Thus, there exists some constant $C$ such that,
\begin{align}\label{align:negligible}
\limsup_{t \downarrow 0} \; t\log\frac{\mathbb{P}(\tau_N\le t)}{\mathbb{P}(B_t \in \ball(\z,\delta))} 
\quad&\le\quad
\limsup_{t \downarrow 0} \; t\log\left(C\frac{\exp\left(-\frac{(N-1)^2}{2td^2}\right)}{\exp\left(-\frac{I^2}{2t}\right)}\right) 
\quad<\quad0\,.
\end{align}
By the Large Deviation principle for Brownian motion \citep{Varadhan-1984}, 
\begin{equation*}
\lim_{t \downarrow 0} \;t \log\frac{\Prob{B_t \in \ball(\z,\delta+Mt)}}{\Prob{B_t \in \ball(\z,\delta)}} 
\quad=\quad  \lim_{t \downarrow 0} \;t \log\frac{\Prob{B_t \in \ball(\z,\delta-Mt)}}{\Prob{B_t \in \ball(\z,\delta))}}
\quad=\quad0\,.
\end{equation*}
This, along with (\ref{eq:UB}), (\ref{eq:LB}) and (\ref{align:negligible}), yields the lemma.
\end{proof}
\begin{rem}
 The above lemma can be regarded as a weak form of a large deviation principle (LDP) for the diffusion $X$, specialized to a particular set $B(\z,\delta)$. 
 The general form of the LDP can be shown to hold under the additional assumption of linear growth of the drift vectorfield, 
 which is used to control the moments of the Radon-Nikodym derivative of the law of $X$ with respect to that of $B$ obtained by the Girsanov Theorem \citep{Varadhan-1984}.
\end{rem}

Note that for each fixed $(s,\x)$ the transition density $(t,\y) \mapsto p(s,\x; t,\y)$ satisfies the Kolmogorov forward equation
\begin{equation}\label{eq:Kolmog}
\partial_t p\quad=\quad\generator^*p
\end{equation}
where $\generator^*$ is the adjoint of the operator $\generator$. 
Under assumptions (A1) and (A2) the above equation can be rewritten as
$$
(\mathcal{A}+h)p\quad=\quad0\,,
$$
where $\mathcal{A}$ is a \emph{uniformly parabolic operator} \citep[p.~173]{ProtterWeinberger-1984} and $h$ is bounded on compact subsets of $[0,\infty) \times \Reals^d$. 
We now state the \emph{Strong Maximum Principle} for uniformly parabolic equations in the following form
(see Theorem 5, Theorem 7 and part (ii) of the remark following Theorem 7, pp.~173--175 of \citealp[]{ProtterWeinberger-1984}). 
\begin{lem}\label{lem:StrongMP}
Let $u$ be a solution of
$$
(\mathcal{A}+h)u \quad\ge\quad 0
$$
on a domain of the form $\Omega_T=(0,T] \times \Omega$, where $\Omega$ is a bounded and connected open set and the coefficients of $\mathcal{A}$,
and the function $h$ are bounded on closed subsets of $\Omega_T$. Suppose $u \le 0$ on $\Omega_T$ and $u(T,x')=0$ for some $x' \in \Omega$. Then $u \equiv 0$ on $\Omega_T$.
\end{lem}
It is now possible to state and prove the main result of this section, which can be seen as a stronger version of \citet[Proposition 3.9]{Kuwada-2009}, although our proof is quite different and slightly shorter.
\begin{thm}\label{thm:zeroset}
Take any $s>0$. For any $(\x,\y) \in \mathcal{M}(\mu_s)$, the following equalities hold:
\begin{align*}
I(\x_0,\y_0,s)\quad&=\quad H(\x,\y)\,,\\
I^{-}(\x_0,\y_0,s)\quad&=\quad H^-(\x,\y)\,,\\
I^{+}(\x_0,\y_0,s)\quad&=\quad H^+(\x,\y)\,.
\end{align*}
\end{thm}
\begin{proof}
By continuity of $\alpha(s,\cdot)$, it suffices to prove that $H^-(\x,\y) \subseteq I^{-}(\x_0,\y_0,s)$ and $H^+(\x,\y) \subseteq I^{+}(\x_0,\y_0,s)$.

\medskip

We will first show that $\alpha(s,\z^*) \ge 0$ for all $\z^* \in H^-(\x,\y)$. Suppose, in contradiction, that $\alpha(s,\z^*)<0$ for some $\z^* \in H^-(\x,\y)$. 

Since $H^-(\x,\y)$ is open and $\alpha$ is continuous, 
we can choose $\delta>0$ such that $\ball(\z^*,\delta) \subseteq H^-(\x,\y)$ and $\alpha(s+s',\z)<0$ for all $\z \in \ball(\z^*,\delta)$ for sufficiently small $s'>0$. By Lemma \ref{lem:disjsupp} this implies that 
\[
\mu(X_{s+s'} \in \ball(\z^*,\delta), \tau>s+s')\quad=\quad0
\]
for all sufficiently small $s'>0$.
Let $B_1, B_2$ be Brownian motions starting from $\x$ and $\y$ respectively.
Since \(\z^* \in H^-(\x,\y)\), it follows that $\Prob{B_{1,t} \in \ball(\z^*,\delta)} > \Prob{B_{2,t} \in \ball(\z^*,\delta)}$ for all $t>0$.
By Lemma \ref{lem:LD}, 
if $s'>0$ is sufficiently small then it follows that
\begin{equation}\label{eq:probineq}
\mu\left((\theta_sX)_{s'} \in \ball(\z^*,\delta)\ \Big| \ X_s=\x\right) \quad>\quad \mu\left((\theta_sY)_{s'} \in \ball(\z^*,\delta)\ \Big| \ Y_s=\y\right)\,.
\end{equation}
By continuity of the transition densities, 
for all sufficiently small \(s'>0\) and for small enough open sets $U_1$ containing $\x$ and $U_2$ containing $\y$, for any $(\mathbf{u}_1 , \mathbf{u}_2) \in (U_1 \times U_2) \cap \mathcal{M}(\mu_s)$, 
 \begin{align}\label{align:smalltime}
 &\mu\left(X_{s+s'} \in \ball(\z^*,\delta) , \tau>s+s'\ \Big| \ X_s=\mathbf{u}_1, Y_s = \mathbf{u}_2\right)
 \quad=\quad \int_{\ball(\z^*,\delta)} \alpha^+(s,\mathbf{u}_1,\mathbf{u}_2,s+s',\z)\d\z \nonumber\\
 &\quad \ge \quad \int_{\ball(\z^*,\delta)} \alpha(s,\mathbf{u}_1,\mathbf{u}_2,s+s',\z)\d\z \nonumber\\
 &\qquad\quad =\quad \mu\left((\theta_sX)_{s'} \in \ball(\z^*,\delta)\ \Big| \ X_s=\mathbf{u}_1\right)
 - \mu\left((\theta_sY)_{s'} \in \ball(\z^*,\delta)\ \Big| \ Y_s=\mathbf{u}_2\right) \quad>\quad 0\,.
 \end{align}
(Here, the first equality follows from Lemma \ref{lem:disjsupp} and Lemma \ref{lem:flow}.)
Since $(\x,\y) \in \mathcal{M}(\mu_s)$, 
it follows that $\mu((X_s,Y_s) \in (U_1 \times U_2) \cap \mathcal{M}(\mu_s))>0$, yielding (for all sufficiently small \(s>0\))
\[
\mu(X_{s+s'} \in \ball(\z^*,\delta), \tau>s+s')\quad>\quad0\,,
\]
contradicting our assumption. Hence $\alpha(s,\z^*)\geq0$ for all $\z^* \in H^-(\x,\y)$.  Similarly, $\alpha(s,\z^*)\le0$ for all $\z^* \in H^+(\x,\y)$.

\medskip

We have thus shown that
\begin{align*}
H^-(\x,\y) \quad&\subseteq\quad I(\x_0,\y_0,s) \cup I^{-}(\x_0,\y_0,s)\,,\\
H^+(\x,\y) \quad&\subseteq\quad I(\x_0,\y_0,s) \cup I^{+}(\x_0,\y_0,s)\,.
\end{align*}
Suppose \(H^-(\x,\y) \cap I(\x_0,\y_0,s)\) is non-empty, and pick $\z^* \in  H^-(\x,\y) \cap I(\x_0,\y_0,s)$. 
Since $\alpha(s,\cdot)$ is nonnegative on the open set $H^-(\x,\y)$, there exists $\delta>0$ such that $\alpha(s,\z)\ge 0$ for all $\z \in \ball(\z^*,\delta)$. 
Choose open sets $U_1$ containing $\x$ and $U_2$ containing $\y$, and possibly smaller $\delta>0$, such that $|\x'-\z| < |\y' - \z|$ for all $\x' \in U_1, \y' \in U_2$ and $\z \in \ball(\z^*,\delta)$. 
It is given that $(\x,\y) \in \mathcal{M}(\mu_s)$; since the process $((X_t, Y_t): t \ge 0)$ has continuous paths there must be $\eta>0$ such that $\mu_t(U_1 \times U_2)>0$ for all $t \in [s-\eta, s]$.

The function $(t,\z) \mapsto \alpha(t,\z)$ solves the Kolmogorov forwards equation (\ref{eq:Kolmog}). 
Thus we can apply Lemma \ref{lem:StrongMP} to $-\alpha$ on $\Omega_{\eta}= (s-\eta, s] \times \ball(\z^*,\delta)$, and deduce that either $\alpha(t,\z)=0$ for all $s-\eta<t<s$ and all $\z \in \ball(\z^*,\delta)$, or there exists $s' \in (s-\eta,s)$, $0<\eps < s-s'$ and an open set $U \subseteq \ball(\z^*,\delta)$ such that $\alpha(t,\z) <0$ for all $\z \in U$ and all $t \in [s', s'+ \eps)$.
In either case (taking \(U=\ball(\z^*,\delta)\) and any $s' \in (s-\eta,s)$, $\epsilon \in (0,s-s')$ in the first case), for all $t \in [s', s'+ \eps)$
\begin{equation}\label{eq:MP}
\mu(X_t \in U,\tau>t)=0\,.
\end{equation}

Now choose $(\x', \y') \in (U_1 \times U_2) \cap \mathcal{M}(\mu_{s'})$ (non-empty, since $U_1$ and $U_2$ are disjoint and $\mu_{s'}(U_1 \times U_2)>0$) and apply the same argument as the one used in obtaining (\ref{align:smalltime}), but with $\x', \y'$ replacing $\x, \y$ and $s'$ replacing $s$.
We obtain 
\[
\mu\left(X_{s'+s''} \in U,\tau>s'+s''\right)\quad>\quad0
\]
for some $s'' \in [s', s'+ \eps)$, contradicting (\ref{eq:MP}). The lemma follows.
\end{proof}
\begin{rem}\label{rem:time-varying-mirror}
 The above theorem shows that for a Markovian maximal coupling, for any time $s$, the locus $I(\x_0, \y_0,s)$ can be viewed as a (possibly time-varying) \emph{mirror} which 
 realizes the coupling in a very explicit way, using a (possibly time-varying) reflection isometry.
 \end{rem}
The following corollary to the above lemma shows that the coupling time $\tau$ is, in fact, 
the hitting time of the deterministic space-time set $\{(s,I(\x_0,\y_0,s)) : s>0\}$ by the process $((s,X_s):s\geq0)$ (equivalently, $((s,Y_s):s>0)$). 
In particular, $X$ and $Y$ will couple at the first time they meet. 
Furthermore, the interface representation described in Theorem \ref{thm:zeroset} will hold almost surely for all time before coupling occurs. 
\begin{cor}\label{cor:alltime}
Consider a Markovian maximal coupling, with coupling time \(\tau\).
Set $\tau'=\inf\{s>0: X_s \in I(\x_0,\y_0,s)\}$.
Almost surely $\tau=\tau'$. Furthermore, $\mu$-almost surely, for all $t< \tau$,
\begin{align}\label{align:constalltime}
I(\x_0,\y_0,t)\quad&=\quad H(X_t,Y_t),\nonumber\\
I^{-}(\x_0,\y_0,t)\quad&=\quad H^-(X_t,Y_t),\nonumber\\
I^{+}(\x_0,\y_0,t)\quad&=\quad H^+(X_t,Y_t).
\end{align}
\end{cor}
\begin{proof}
Note that, by Lemma \ref{lem:disjsupp},
\begin{equation*}
\mu\left(Y_q \in I^{-}(\x_0,\y_0,q) \text{ for some rational } q < \tau\right)\quad=\quad0 \,.
\end{equation*}
Since the trajectories of \(Y\) are continuous, it follows that almost surely \(Y_t\) is contained in the complement of \(I^{-}(\x_0,\y_0,t)\) for all \(t<\tau\). 
This implies that before time $\tau'$, $X$ and $Y$ are supported on disjoint subsets of the state space and hence
 \begin{equation}\label{eq:stopineq}
\mu\left(\tau' \le \tau\right)\quad=\quad1\,.
\end{equation}
For any $t>0$, we define the event
\begin{align}\label{eq:event}
E_t\quad=&\quad \Big[ \text{ Either }X_t = Y_t\,, \text{ or } X_t \neq Y_t\, \text{ and all three equalities } I(\x_0,\y_0,t)=H(X_t,Y_t),\nonumber\\
&\qquad \quad I^{-}(\x_0,\y_0,t)=H^-(X_t,Y_t)\,, \;I^{+}(\x_0,\y_0,t)=H^+(X_t,Y_t) \text{ hold. }\Big]\,.
\end{align}
Theorem \ref{thm:zeroset} implies the assertion
\begin{equation}\label{eq:hyprat}
\mu\left(E_q \text{ is true for all rational } q \right)\quad=\quad1\,,
\end{equation}
hence almost surely $\displaystyle{E=\cap_{q \in \mathbb{Q}}E_q}$ holds. 
Take any $t > 0$ with $X_t \neq Y_t$ and let $\z \in H(X_t,Y_t)$. 
Then it follows from the definition of \(H(\x,\y)\) and the continuity of sample paths of \(X\) and \(Y\) 
that there is a rational sequence $t_n \downarrow t$ and $\z_n \in H(X_{t_n},Y_{t_n})$ such that $\z_n \rightarrow \z$. 
Thus, on the event \(E\), the continuity of $\alpha$ implies that $H(X_t,Y_t) \subseteq I(\x_0,\y_0,t)$.

\medskip

Now, take $\z \in H^+(X_t,Y_t)$ when \(X_t\neq Y_t\). 
The continuity of sample paths of $X$ and $Y$ implies that there exist $\eta, \delta>0$ with $\ball(\z,\eta) \subseteq H^+(X_s,Y_s)$ for all $s \in [t-\delta,t]$. 
On the event $E$, the continuity of $\alpha$ implies $\alpha(s,\z') \le 0$ for all $s \in [t-\delta,t]$ when $\z' \in \ball(\z,\eta)$. 
Thus, as $\alpha(q,\z)<0$ for all rational $q \in [t-\delta,t]$, Lemma \ref{lem:StrongMP} implies $\alpha(t,\z)<0$. 
Thus, $H^+(X_t,Y_t) \subseteq I^{+}(\x_0,\y_0,t)$. Similarly, $H^-(X_t,Y_t) \subseteq I^{-}(\x_0,\y_0,t)$. 
As $\mu(E)=1$, it follows that
\begin{equation}\label{eq:hypalltime}
\mu\left(E_t \text{ is true for all } t\right)\quad=\quad1\,.
\end{equation}
Note that, in particular, (\ref{eq:stopineq}) and (\ref{eq:hypalltime}) imply that if $\tau' < \infty$, then $X_{\tau'}=Y_{\tau'}$ almost surely. 
For \(\|X_t-Y_t\|=\tfrac12\dist(X_t,H(X_t,Y_t))=\tfrac12\dist(X_t,I(\x_0,\y_0,t))\) (when \(t<\tau'\)), by definition of \(H(X_t,Y_t)\). 
 
The corresponding argument for $Y$ implies that $\tau'$ also satisfies $\tau'=\inf\{s>0: Y_s \in I(\x_0,\y_0,s)\}$. 
Therefore, $\tau'$ is a stopping time for both $X$ and $Y$. Since $X_{\tau'}=Y_{\tau'}$, 
we can extend $X$ and $Y$ synchronously beyond time $\tau'$. Combined with (\ref{eq:stopineq}), 
this implies $\tau=\tau'$ almost surely, since the maximal coupling time \(\tau\) must be stochastically smaller than all other coupling times. 
Consequently 
$$
\mu\left(X_t\neq Y_t \text{ for all } t< \tau\right)\quad=\quad1\,.
$$
 
This, together with (\ref{eq:hypalltime}), yields (\ref{align:constalltime}) and thus the corollary is proved.
\end{proof}

\subsection{Time evolution of the mirror}\label{subsection:mirror}
We now analyze the time-evolution of the mirror.
From Theorem \ref{thm:zeroset}, it follows that the mirror $I(\x_0,\y_0,t)$ is a hyperplane for each $t>0$. 
We
parametrize 
this hyperplane by
its
signed distance from the origin, say $l(t)$,
together with 
the normal vector to the hyperplane, say $\mathbf{n}(t)$. 
There is an ambiguity of sign in the choice of \(\mathbf{n}(t)\); however the next lemma states that 
\(\mathbf{n}(t)\) 
can be chosen
to make this parametrization continuous up to the coupling time \(\tau\).
\begin{lem}\label{lem:contpar}
Suppose that a Markovian maximal coupling exists for $X$ and $Y$. Then there exists a continuous parametrization $\left((l(t), \mathbf{n}(t)): t\in [0,\tau)\right)$ of $I(\x_0,\y_0,\cdot)$.
\end{lem}
\begin{proof}
Corollary \ref{cor:alltime}, together with the remark following Lemma \ref{lem:flow}, 
shows that the following subset of coupled path-space \(C[0,\infty)^2\) is non-empty for any $S>0$,
and indeed of full \(\mu\)-measure in the subset corresponding to \(\tau>S\):
\[
A_{S}\quad=\quad\{\omega \in C[0,\infty)^2: I(\x_0,\y_0,t)=H(X_{t}(\omega),Y_{t}(\omega)) \text{ for all } t\le S, \tau > S\}\,.
\] 
Consider any coupled pair of paths $\omega \in A_S$. Define $(l(t), \mathbf{n}(t))$ on $[0,S]$ by
\begin{align}\label{align:parametrization}
\mathbf{n}^{(S)}(t)\quad&=\quad\frac{X_t(\omega)-Y_t(\omega)}{|X_t(\omega)-Y_t(\omega)|},\nonumber\\
l^{(S)}(t)\quad&=\quad\mathbf{n}^\top(t)\left(\frac{X_t(\omega)+Y_t(\omega)}{2}\right).
\end{align}
This gives a continuous parametrization $(l^{(S)},\mathbf{n}^{(S)})$ on $[0,S\wedge\tau)$.

This recipe can be used to define $(l^{(N)},\mathbf{n}^{(N)})$ on $[0,N\wedge\tau)$ for each positive integer $N$.
By continuity of $\mathbf{n}^{(N)}$ and $\mathbf{n}^{(N+1)}$ on the (connected) interval $[0,N\wedge\tau)$,
we see that either $\mathbf{n}^{(N)}\equiv\mathbf{n}^{(N+1)}$ or $\mathbf{n}^{(N)}\equiv-\mathbf{n}^{(N+1)}$ on $[0,N\wedge\tau)$. But 
\[
\lim_{t \downarrow 0}\mathbf{n}^{(N)}(t)\quad=\quad\lim_{t \downarrow 0}\mathbf{n}^{(N+1)}(t)=\frac{\x_0-\y_0}{|\x_0-\y_0|}\,,
\]
implying $\mathbf{n}^{(N)}\equiv \mathbf{n}^{(N+1)}$ on $[0,N\wedge\tau)$. 
Consequently $l^{(N)}=l^{(N+1)}$ on $[0,N\wedge\tau)$. So we can consistently and continuously define the parametrization as $\left((l(t),\mathbf{n}(t)): t \in [0,\tau)\right)$, thus proving the lemma.
\end{proof}
In fact the parametrization is not simply continuous but is also continuously differentiable:
\begin{lem}\label{lem:IFT}
Suppose that a Markovian maximal coupling exists for $X$ and $Y$. Then the parametrization $(l(t), \mathbf{n}(t))$ of the mirror $I(\x_0,\y_0,t)$ (defined for  $t \in [0,\tau)$) is continuously differentiable in \(t\).
\end{lem}
\begin{proof}
We use the fact that the map given by reflection in the hyperplane parametrized by $(l(t), \mathbf{n}(t))$,
\[
F(t,\x)\quad=\quad
({\Identity}-2\mathbf{n}(t)\mathbf{n}^\top(t))\x+2l(t)\mathbf{n}(t)\,,
\]
takes $X_t$ to $Y_t$  for  $t \in [0,\tau)$ (this follows from \(I(\x_0,\y_0,t)=H(X_t,Y_t)\)). 
Take any $\x \in I^{-}(\x_0,\y_0,t)$. 
Let $U$ be an open ball containing $\x$ and contained in $I^{-}(\x_0,\y_0,t)$. Let $\tau_U= \inf\{s > t: X_s \notin U\}$. 
Consider the corresponding stopped processes $X^U_s= X_{s \wedge \tau_U}$ and $Y^U_s= Y_{s \wedge \tau_U}$ for $s \ge t$. 
We write expectation with respect to $\mu$ using $\mathbb{E}$.

By general properties of diffusions \citep[Chapter 11]{Nelson-1967},
\begin{eqnarray}\label{eqnarray:nelder}
\mathbf{b}(t,\x)\quad&=\quad&\lim_{s \downarrow t}\; \Expect{\frac{X^U_s-\x}{s-t} \; \Big|\; \ X^U_t=\x}\,,\nonumber\\
\mathbf{b}(t,F(t,\x))\quad&=\quad&\lim_{s \downarrow t}\; \Expect{\frac{Y^U_s-F(t,\x)}{s-t} \; \Big|\; \ Y^U_t=F(t,\x)}\,.
\end{eqnarray}
Note that under the coupling $\mu$ we may use Corollary \ref{cor:alltime} to see that $Y^U_s=F(s, X^U_s)$ for all $s \ge t$ with probability one. 
Thus, we can write the last expression above as
\begin{eqnarray*}
\mathbf{b}(t,F(t,\x))\quad&=\quad&\lim_{s \downarrow t}\; \Expect{\frac{F(s,X^U_s)-F(t,\x)}{s-t} \;\Big|\; \ X^U_t=\x}\\
\quad&=\quad&\lim_{s \downarrow t}\; \Expect{\frac{F(s,X^U_s)-F(s,\x)}{s-t} \;\Big|\; \ X^U_t=\x} + \lim_{s \downarrow t}\frac{F(s,\x)-F(t,\x)}{s-t}\,,
\end{eqnarray*}
in the sense that if the limit of \(\Expect{\tfrac{F(s,X^U_s)-F(s,\x)}{s-t} \Big| \ X^U_t=\x}\) exists then also the limit of \(\tfrac{F(s,\x)-F(t,\x)}{s-t}\) exists and is defined by the above.
By linearity of $F$ in $\x$, we see that the first summand becomes
\begin{eqnarray*}
\lim_{s \downarrow t}\; \Expect{\frac{F(s,X^U_s)-F(s,\x)}{s-t} \ \Big| \ X^U_t=\x}\quad&=\quad& ({\Identity}-2\mathbf{n}(t)\mathbf{n}^\top(t))\lim_{s \downarrow t} \; \Expect{\frac{X^U_s-\x}{s-t} \ \Big| \ X^U_t=\x}\\
\quad&=\quad& ({\Identity}-2\mathbf{n}(t)\mathbf{n}^\top(t))\mathbf{b}(t,\x)\,.
\end{eqnarray*}
This shows that ${\lim_{s \downarrow t}\tfrac{F(s,\x)-F(t,\x)}{s-t}}$ exists for each $\x$ and for all $t \in [0,\tau)$ and indeed is continuous in $t$. 
This is enough to show that $t \mapsto F(t, \x)$ is continuously differentiable for each $\x$ \citep[Theorem 1.3]{Bruckner-1978}. 
This follows from the facts that ${t \mapsto ({\Identity}-2\mathbf{n}(t)\mathbf{n}^\top(t))}$ and ${t \mapsto l(t)\mathbf{n}(t)}$ are continuously differentiable,
and actually requires these facts to be true:
consider \(F(t,\x)\) for \(\x\) varying over an orthonormal basis and also for \(\x=0\).

Now, take any $t_0 \in[0,\tau)$. 
Let $n_i$ denote the \(i^\text{th}\) component of $\mathbf{n}$. 
As $|\mathbf{n}(t_0)|=1$, there is an $i$ such that $n_i(t) \neq 0$ in a neighbourhood $V$ of $t_0$. The continuous differentiability of $\displaystyle{t \mapsto ({\Identity}-2\mathbf{n}(t)\mathbf{n}^\top(t))}$ implies $n_in_j$ is continuously differentiable in $V$ for all $1 \le j \le d$. This implies $n_j$ is continuously differentiable in $V$ for all $j$. Differentiability of $\displaystyle{t \mapsto l(t)\mathbf{n}(t)}$ then shows that $l$ is continuously differentiable on $V$. This proves the lemma.
\end{proof}

\subsection{Structure of the coupling}
All the tools having been assembled, 
it is now possible to present a rather explicit description of drifts $\mathbf{b}$ which permit the existence of a Markovian maximal coupling of two copies \(X\) and \(Y\) of a Euclidean diffusion 
with the required regularity conditions.

We begin with a notational remark. For any $\x \in \Reals^d$ and any hyperplane $\underline{h}$, we denote by $\underline{h}\x$ the reflection of $\x$ in $\underline{h}$. 
We write $\underline{h}_k$ for the hyperplane $\{x_k=0\}$.

\medskip

The first lemma of this subsection concerns an observation concerning rotations and shifts of these Euclidean diffusions.
\begin{lem}\label{lem:rotdiff}
Let $X$ be an Euclidean diffusion satisfying assumptions (A1), (A2). 
Let $Q: [0,\infty) \to \textbf{O}(d)$ be a continuously differentiable function taking values in the space of orthogonal \((d\times d)\) matrices,
and let $l: [0,\infty) \rightarrow \Reals$ be a continuously differentiable real-valued function. 
Then the new process given by
\begin{equation}\label{eq:rotdiff}
\widetilde{X}_t=Q(t)X_t -l(t)\mathbf{e}_1
\end{equation}
satisfies the stochastic differential equation
\begin{equation}\label{eq:rotSDE}
\d \widetilde{X}_t=\widetilde{\mathbf{b}}(t,\widetilde{X}_t)\d t + \d \widetilde{B}_t
\end{equation}
where
\begin{equation}\label{eq:tildeb}
\widetilde{\mathbf{b}}(t,x)=\dot{Q}(t)Q^{T}(t)(x+l(t)\mathbf{e}_1)+Q(t)\mathbf{b}(t,Q^{T}(t)(x+l(t)\mathbf{e}_1))-\dot{l}(t)\mathbf{e}_1
\end{equation}
and
\begin{equation}\label{eq:rotbrow}
\d\widetilde{B}_t=Q(t)\d B_t\,.
\end{equation}
Here, $\dot{Q}$ and $\dot{l}$ denote the respective time-derivatives and $Q^\top$ denotes the matrix transpose.
\end{lem}
\begin{proof}
The result follows by direct calculation using It\^o calculus.
\end{proof}
\begin{rem}
Note that the transformed drift given by (\ref{eq:tildeb}) satisfies the regularity Assumptions (A1) and (A2). (A1) follows via the explicit form of (\ref{eq:tildeb}) from the fact that $\mathbf{b}$ satisfies (A1) and $Q$ and $l$ are continuously differentiable. (A2) for the new process $\widetilde{X}$ follows from (\ref{eq:rotdiff}) and the fact that $X$ satisfies (A2).
\end{rem}
The following theorem describes Markovian maximal couplings for the class of time-nonhomogeneous Euclidean diffusions satisfying suitable regularity conditions. 
The intuitive content of the theorem is, given an MMC $(X,Y)$,
\emph{applying deterministic time-varying rotations and translations to the ambient Euclidean space reduces this MMC to a reflection coupling in a fixed hyperplane}.
Thus, in a certain sense, reflection coupling is the only type of Markovian coupling that can possibly preserve maximality.
\begin{thm}\label{thm:MMCchar}
Let $X$ be an Euclidean diffusion starting from $\x_0$ and satisfying assumptions (A1), (A2). 
\begin{itemize}
\item[(i)] Suppose the following holds for every $x \in \Reals^d$, for the fixed hyperplane \(\underline{h}_1=\{x_1=0\}\). 
\begin{equation}\label{eq:bsymmetry}
\mathbf{b}(t,\underline{h}_1\x)=\underline{h}_1\mathbf{b}(t,\x)
\end{equation}
Then, for $\tau_0=\inf\{t\ge 0: X_t \in \underline{h}_1\}$, the reflection-coupling
\begin{eqnarray}\label{eq:reflcoup}
Y_t
\quad&=\quad&\begin{cases}
\underline{h}_1X_t & \mbox{ if } t< \tau_0 \\
X_t & \mbox{ if } t \ge \tau_0
\end{cases}
\end{eqnarray}
gives a Markovian maximal coupling between two copies of the diffusion starting from $\x_0$ and $\underline{h}_1\x_0$ respectively.
\item[(ii)] 
Let \(Y\) be a coupled copy of \(X\).
Then $(X,Y)$ is a Markovian maximal coupling up to the maximal coupling time \(\tau\) if and only if there exist $C^1$ curves $Q:[0,\tau) \rightarrow \mathbf{O}(d)$ and $l:[0,\tau) \rightarrow \Reals$ (compare Lemma \ref{lem:rotdiff}) with ${Q(0)\tfrac{\x_0-\y_0}{|\x_0-\y_0|}=\mathbf{e}_1}$ and ${l(0)=\tfrac{|\x_0|^2-|\y_0|^2}{2|\x_0-\y_0|}}$, such that $(\widetilde{X}, \widetilde{Y})$ obtained from $(X,Y)$ using the transformation \eqref{eq:rotdiff} are reflection-coupled according to the recipe \eqref{eq:reflcoup}. In particular, the transformed time-varying drift $\widetilde{\mathbf{b}}$ given by (\ref{eq:tildeb}) must satisfy 
\begin{equation}\label{eq:tildebsymmetry}
\widetilde{\mathbf{b}}(t,\underline{h}_1\x)=\underline{h}_1\widetilde{\mathbf{b}}(t,\x).
\end{equation}
\end{itemize}
\end{thm}
\begin{proof}
\begin{itemize}
\item[(i)] Equation \eqref{eq:bsymmetry} implies that the process $(\underline{h}_1X_t\;:\;t \ge 0)$ has the same law as the diffusion starting from $\underline{h}_1\x_0$ and thus, the reflection-coupling (\ref{eq:reflcoup}) gives a valid coupling. Reflection in the hyperplane $\underline{h}_1$ thus gives a \textit{reflection structure} in the sense of \citet[Definition 2.1]{Kuwada-2007}. Maximality follows from \citet[Proposition 2.2]{Kuwada-2007}.
\item[(ii)] First, note that if $\widetilde{X}$ and $\widetilde{Y}$ are reflection-coupled according to (\ref{eq:reflcoup}), then analysis of generators of $\underline{h}_1\widetilde{X}_t$ and $\widetilde{Y}_t$ yields \eqref{eq:tildebsymmetry}. Now, applying part (i) of the theorem, we deduce that $(\widetilde{X},\widetilde{Y})$ is a Markovian maximal coupling.
Furthermore, as
$$
(t,x) \mapsto (t, Q^\top(t)(x+l(t)\mathbf{e}_1))
$$
 is a bijective, bimeasurable function, so application of Lemma \ref{lem:mapsto} to $(t,\widetilde{X}_t) \rightarrow (t,X_t)$ and $(t,\widetilde{Y}_t) \rightarrow (t,Y_t)$ shows that $(X,Y)$ is a Markovian maximal coupling.
 
Conversely, let $(X,Y)$ be a Markovian maximal coupling of two copies of the diffusion starting from $\x_0$ and $\y_0$. 
Then the results of subsections \ref{subsection:interface} and \ref{subsection:mirror} show that 
there exist continuously differentiable functions $l: [0,\infty) \rightarrow \Reals$ and $\mathbf{n}: [0,\infty) \rightarrow \mathbb{S}^{d-1}$ parametrising the mirror $I(\x_0,\y_0,t)$. Moreover, these functions should satisfy $\mathbf{n}(0)=\frac{\x_0-\y_0}{|\x_0-\y_0|}$ and $l(0)=\frac{|\x_0|^2-|\y_0|^2}{|\x_0-\y_0|}$. To see this, take $t \downarrow 0$ in \eqref{align:parametrization}.
Furthermore, Theorem \ref{thm:zeroset} and the corollary following it show that $X$ and $Y$ are coupled on $t<\tau$ according to the relationship
\begin{equation}\label{eq:YfuncX}
Y_t\quad=\quad
({\Identity}-2\mathbf{n}(t)\mathbf{n}^\top(t))X_t+2l(t)\mathbf{n}(t)\,.
\end{equation}

The construction of $Q$ follows by applying Gram-Schmidt orthogonalization to extend $\mathbf{n}(0)$ to an orthonormal basis $(\mathbf{n}(0),\mathbf{v}_1, \dots, \mathbf{v}_{d-1})$ of $\Reals^d$. 
Note that the vectors $\mathbf{v}_i$ lie in the tangent space of $\mathbb{S}^{d-1}$ based at $\mathbf{n}(0)$. 
The vector function $(\mathbf{n}(t):t \ge 0)$ traces out a $C^1$ curve \(\gamma\) on the sphere $\mathbb{S}^{d-1}$. 
Parallel transport \citep[p.~75]{GallotHulin-2004} can be applied along $\gamma$ to each vector $\mathbf{v}_i$;
this produces $C^1$ vectorfields $\mathbf{X}_i: [0,\infty) \rightarrow \Reals^d$ along $\gamma$. 
\citet[Proposition 2.74]{GallotHulin-2004} shows that $(\mathbf{n},\mathbf{X}_1,\dots,\mathbf{X}_{d-1})$ produces a $C^1$ orthonormal frame along $\gamma$, so set 
$$
Q^\top(t)\quad=\quad
(\mathbf{n}(t),\mathbf{X}_1(t),\dots,\mathbf{X}_{d-1}(t)).
$$

We now produce a new pair of diffusions with time-varying drifts, $(\widetilde{X}, \widetilde{Y})$, by applying the transformation \eqref{eq:rotdiff} to $(X,Y)$ with drift $\widetilde{\mathbf{b}}$ and driving Brownian motion $\widetilde{B}$ as described in Lemma \ref{lem:rotdiff}.
This new pair is also a Markovian maximal coupling (use Lemma \ref{lem:mapsto}), and from equation \eqref{eq:YfuncX}
it follows that the coupled pair $(\widetilde{X},\widetilde{Y})$ is described by the transformation \eqref{eq:reflcoup}.
As discussed in part (i) of this proof, the relationship \eqref{eq:tildebsymmetry} follows as a direct consequence.
\end{itemize}
\end{proof}

Inverting the relationship \eqref{eq:tildeb}, and using the relationship \eqref{eq:tildebsymmetry}, the above theorem yields the following characterisation of drifts which permit MMC:
\begin{cor}\label{cor:MMCchar}
Under assumptions (A1) and (A2), 
the Markovian coupling of \(d\)-dimensional Euclidean diffusions $(X,Y)$ is a Markovian maximal coupling if and only if there exist function $Q:[0,\tau)\to\textbf{O}(d)$ and $l:[0,\tau)\to\Reals$, 
as prescribed in Theorem \ref{thm:MMCchar}, such that
\begin{equation}\label{eq:driftchar}
\mathbf{b}(t,\x)=Q^\top(t)\widetilde{\mathbf{b}}(t,Q(t)\x-l(t)\mathbf{e}_1)-Q^\top(t)\dot{Q}(t)\x+\dot{l}(t)\mathbf{n}(t)
\end{equation}
for some $\widetilde{\mathbf{b}}$ satisfying Assumptions (A1) and (A2) and fulfilling the relationship \eqref{eq:tildebsymmetry}.
\end{cor}

\subsection{Rigidity theorems for time-homogeneous diffusions}
The previous subsection established an implicit classification of all {time-nonhomogeneous} diffusions that can be coupled by a Markovian maximal coupling. 
But, as noted in the literature, not many examples of such couplings are known for \emph{time-homogeneous} diffusions. 
It is a matter of general belief that the class of such time-homogeneous diffusions is very small, but little rigorous work appears to have been done to specify this class.

In this subsection we obtain a constraint equation on the drift, leading to certain general conditions on the drift and the starting points which are necessary for the existence of Markovian maximal couplings. 
In the case of affine drifts the constraint equations are explicit enough to classify all affine drifts leading to Markovian maximal couplings. 
We then state and prove the main theorem of this subsection: if there are two balls $\ball(\x_0,r)$ and $\ball(\y_0,r)$ in $\Reals^d$, such that a Markovian maximal coupling exists from all pairs of points $(\x,\y) \in \ball(\x_0,r) \times \ball(\y_0,r)$, then the drift has to be of a very simple affine form, verifying the popular belief that Markovian maximal couplings are indeed very rare.

We conclude by showing a stronger result for one-dimensional diffusions, which states that for such a coupling to exist for a specific pair of starting points, 
either the drift must be an odd function centred at a point, or it must be affine.

The following lemma supplies the constraint equation on the drift. Recall that 
\begin{equation}\label{eq:reflection}
F(t,\x)\quad=\quad
({\Identity}-2\mathbf{n}(t)\mathbf{n}^\top(t))\x+2l(t)\mathbf{n}(t)
\end{equation}
is a linear tranformation sending \(\x\in \Reals^d\) to its reflection in the mirror $I(\x_0,\y_0,t)$. 
For the sake of concise exposition, in the following two lemmas and their proofs we suppress the argument $t$ when writing $l$ and $\mathbf{n}$.
\begin{lem}\label{lem:driftconstraint}
Assume (A1), (A2) hold. A Markovian maximal coupling $(X,Y)$ exists from starting points $\x_0$ and $\y_0$ if and only if there exist continuously differentiable functions $l: [0,\infty) \rightarrow \Reals$ and $\mathbf{n}: [0,\infty) \rightarrow \mathbb{S}^{d-1}$, with $\mathbf{n}(0)=\frac{\x_0-\y_0}{|\x_0-\y_0|}$ and $l(0)=\frac{|\x_0|^2-|\y_0|^2}{|\x_0-\y_0|}$, for which the drift vectorfield $\mathbf{b}$ satisfies the following equation:
\begin{equation}\label{eq:driftconstraint}
\mathbf{b}(\x)\quad=\quad
2(\dot{\mathbf{n}}\mathbf{n}^\top-\mathbf{n}\dot{\mathbf{n}}^\top)\x + 2(\dot{l}\mathbf{n}-l\dot{\mathbf{n}}) + ({\Identity}-2\mathbf{n}\mathbf{n}^\top)\mathbf{b}(F(t,\x)).
\end{equation}
\end{lem}
\begin{proof}
First, assume that a Markovian maximal coupling $(X,Y)$ exists. Note from equation \eqref{eq:YfuncX} that
\[
Y_t\quad=\quad F(t,X_t)
\]
for $t \in[0, \tau)$, with $\{(l(t),\mathbf{n}(t)): t \in [0,\tau)\}$ obtained from Lemmas \ref{lem:contpar} and \ref{lem:IFT}. 
Applying stochastic calculus to the function $F$ for $t \in [0, \tau)$, substituting in 
\[
X_t\quad=\quad
({\Identity}-2\mathbf{n}\mathbf{n}^\top)(Y_t-2l\mathbf{n})\,,
\]
and simplifying, we obtain
\begin{equation}\label{eq:Itoconstraint}
\d Y_t\quad=\quad
\left(2(\dot{\mathbf{n}}\mathbf{n}^\top-\mathbf{n}\dot{\mathbf{n}}^\top)Y_t + 2(\dot{l}\mathbf{n}-l\dot{\mathbf{n}}) 
 + ({\Identity}-2\mathbf{n}\mathbf{n}^\top)\mathbf{b}(F(t,Y_t))\right)\d t
 + ({\Identity}-2\mathbf{n}\mathbf{n}^\top)\d B_t\,.
\end{equation}
The diffusion term is clearly a Brownian motion, as can be verified by the L\'evy criterion. 
On the other hand, the drift term in the semimartingale decomposition of $Y$ is given by $\mathbf{b}(t,Y_t) \d t$. 
Equating the two drifts yields the necessity of the drift constraint condition \eqref{eq:driftconstraint}.

Now, suppose $\mathbf{b}$ satisfies (\ref{eq:driftconstraint}) for $l$ and $\mathbf{n}$ as given in the lemma. Let $\tau=\inf\{t>0: X_t \in I(\x_0, \y_0,t)\}$. 
Then (\ref{eq:Itoconstraint}) shows that $Y_t=F(t,X_t){\Identity}(t < \tau) + X_t {\Identity}(t \ge \tau)$ gives a valid coupling $\mu$ of the two copies $(X,Y)$ with coupling time $\tau$. To see that this is indeed the maximal coupling, obtain the $C^1$ curve $Q:[0,\tau) \rightarrow \mathbf{O}(d)$ from $\mathbf{n}$ by the procedure given in the proof of Theorem \ref{thm:MMCchar} (ii). Now, $(\widetilde{X}, \widetilde{Y})$ obtained from $(X,Y)$ by \eqref{eq:rotdiff} is reflection-coupled according to the recipe in \eqref{eq:reflcoup}. Theorem \ref{thm:MMCchar} (ii) then implies that $(X,Y)$ is a Markovian maximal coupling.
\end{proof}
Equation (\ref{eq:driftconstraint}) provides the constraint only in implicit form, and the main task is to extract as much information from it as possible. 
In what follows, we decompose the gradient matrix $\nabla \mathbf{b}$ into symmetric and skew-symmetric parts \emph{via}
\begin{equation}\label{eq:graddec}
\nabla \mathbf{b}(\x)\quad=\quad S(\x)+T(\x)\,,
\end{equation}
where ${S(\x)=\tfrac{\nabla \mathbf{b}(\x)+(\nabla \mathbf{b})^\top(\x)}{2}}$ and ${T(\x)=\tfrac{\nabla \mathbf{b}(\x)-(\nabla \mathbf{b})^\top(\x)}{2}}$. 
The next lemma records relations for \(S(\x)\) and \(T(\x)\) which are direct consequences of \eqref{eq:driftconstraint}.
\begin{lem}\label{lem:consteqns}
Under the hypotheses of Lemma \ref{lem:driftconstraint} and \eqref{eq:driftconstraint}, the following hold for all $\x \in \Reals^d$ and $t>0$:
\begin{enumerate}
\item[(i)]\begin{equation}\label{eq:eigenset}
S(\x)\quad=\quad
({\Identity}-2\mathbf{n}\mathbf{n}^\top)S(F(t,\x))({\Identity}-2\mathbf{n}\mathbf{n}^\top)\,,
\end{equation}
and
\begin{equation}\label{eq:skewform}
T(\x)=2(\dot{\mathbf{n}}\mathbf{n}^\top-\mathbf{n}\dot{\mathbf{n}}^\top)+({\Identity}-2\mathbf{n}\mathbf{n}^\top)T(F(t,\x))({\Identity}-2\mathbf{n}\mathbf{n}^\top)\,.
\end{equation}
In particular, $S(\x)$ and $S(F(t,\x))$ have the same set of eigenvalues.
\item[(ii)] There exists a continuous function $\lambda(\cdot,\cdot): [0,\infty) \times \Reals^d \rightarrow \Reals$ such that
\begin{equation}\label{eq:eigen}
\left(\frac{S(\x)+S(F(t,\x))}{2}\right)\mathbf{n}=\lambda(t,\x)\mathbf{n}.
\end{equation}
\item[(iii)]\begin{equation}\label{eq:skew}
\left(\frac{T(\x)+T(F(t,\x))}{2}\right)\mathbf{n}=\dot{\mathbf{n}}.
\end{equation}
\end{enumerate}
\end{lem}
\begin{proof}
Differentiating both sides of (\ref{eq:driftconstraint}), while recalling the reflection form of \(F(t,\x)\) as given in \eqref{eq:reflection}, we obtain
\begin{equation}\label{eq:grad}
\nabla \mathbf{b}(\x)=2(\dot{\mathbf{n}}\mathbf{n}^\top-\mathbf{n}\dot{\mathbf{n}}^\top)+({\Identity}-2\mathbf{n}\mathbf{n}^\top)\nabla\mathbf{b}(F(t,\x))({\Identity}-2\mathbf{n}\mathbf{n}^\top).
\end{equation}
This immediately yields part (i). The equality of the set of eigenvalues follows from the fact that the reflection matrix $({\Identity}-2\mathbf{n}\mathbf{n}^\top)$ is symmetric and orthogonal.

Parts (ii) and (iii) follow by post-multiplying the equations of part (i) by $\mathbf{n}$, 
bearing in mind that as \(\mathbf{n}\) is a unit vector therefore \(\mathbf{n}\) and \(\dot{\mathbf{n}}\) must be orthogonal.
\end{proof}
Because ${\mathbf{n}(0)=\tfrac{\x_0-\y_0}{|\x_0-\y_0|}}$ and ${l(0)=\mathbf{n}(0).\tfrac{\x_0+\y_0}{2}}$, we know $F(0,\cdot)$ explicitly. 
Even in the generality of the hypotheses of  Lemma \ref{lem:driftconstraint}, one can obtain the following necessary condition on the drift of a Euclidean diffusion for existence of a Markovian maximal coupling:
use (ii) of the above lemma and take $t \downarrow 0$.
\begin{cor}
Under the hypotheses of Lemma \ref{lem:driftconstraint} and \eqref{eq:driftconstraint}, $\mathbf{n}(0)$ must be an eigenvector of ${\tfrac{S(\x)+S(F(0,\x))}{2}}$ corresponding to some eigenvalue $\lambda(\x)$, for every $\x \in \Reals^d$.
\end{cor}
Briefly restrict attention to the case where $\mathbf{b}(\x)$ is affine in \(\x\). 
The following theorem completely classifies the set of such drifts which ensure Markovian maximal coupling.
\begin{thm}\label{thm:affine-drift}
Assume (A1), (A2). 
Let $\mathbf{b}(\x)=A\x +\mathbf{c}$ for some $(d \times d)$ matrix $A$ and some $d$-dimensional vector $\mathbf{c}$. Denote ${S=\tfrac{A+A^\top}{2}}$ and ${T=\tfrac{A-A^\top}{2}}$. 
Then a Markovian maximal coupling $(X,Y)$ exists from starting points $\x_0$ and $\y_0$ if and only if 
there exists an eigenvalue $\lambda_0$ of $S$ such that the vectors $T^k(\x_0-\y_0)$ (for $0 \le k \le d-1$) all lie in the eigenspace of $S$ corresponding to $\lambda_0$. 
In this case (using matrix exponentials \(\operatorname{\mathtt{exp}}\)),
\begin{equation}\label{eq:nform}
\mathbf{n}(t)\quad=\quad \operatorname{\mathtt{exp}}\left({Tt}\right)\frac{\x_0-\y_0}{|\x_0-\y_0|}\,, \text{ and }
\end{equation}
\begin{equation}\label{eq:lform}
l(t)\quad=\quad e^{\lambda_0 t}\frac{|\x_0|^2-|\y_0|^2}{2|\x_0-\y_0|} + e ^{\lambda_0 t} \int_0^t \frac{(\x_0-\y_0)^\top}{|\x_0-\y_0|}\operatorname{\mathtt{exp}}\left({-(T+\lambda_0 {\Identity})s}\right)\mathbf{c}\d s\,.
\end{equation}
\end{thm}
\begin{proof}
Suppose there exists a Markovian maximal coupling $(X,Y)$ starting from $\x_0$ and $\y_0$. From (ii) and (iii) of Lemma \ref{lem:consteqns} we get the following:
\begin{equation}\label{eq:symmpart}
S\mathbf{n}(t)\quad=\quad \lambda(t)\mathbf{n}(t)
\end{equation}
(where we note that $\lambda$ is a function of $t$ only) and
\begin{equation}\label{eq:skewpart}
T\mathbf{n}(t)\quad=\quad \dot{\mathbf{n}}(t)\,.
\end{equation}

Solving (\ref{eq:skewpart}), we get \eqref{eq:nform}. 
Since $T$ is skew-symmetric, the above formula implies $|\mathbf{n}(t)|=1$ for all $t$.

The finite symmetric matrix \(S\) has discrete spectrum;
by this, and the continuity of $\mathbf{n}(\cdot)$ and $\lambda(\cdot)$, it follows immediately from (\ref{eq:symmpart}) that $\lambda(\cdot) \equiv \lambda_0$ for some constant $\lambda_0$.
Thus $\mathbf{n}(t)$, as given by \eqref{eq:nform}, must lie in the eigenspace of $S$ corresponding to $\lambda_0$, for all time $t$. 
Substituting this formula for $\mathbf{n}(t)$ in equation \eqref{eq:symmpart} and differentiating \eqref{eq:nform} $k$ times with respect to $t$ (for $k=0,1,\dots, d-1$), then setting $t=0$, we obtain that the vectors $T^k(\x_0-\y_0)$ for $0 \le k \le d-1$ must all lie in the eigenspace of $S$ corresponding to $\lambda_0$. 
As $T$ solves its characteristic equation, it is clear that all the higher powers $T^k(\x_0-\y_0)$ for $k\ge d$ must also lie in this eigenspace. 
Using the series representation of $\operatorname{\mathtt{exp}}\left({Tt}\right)$, this means that $\mathbf{n}(t)$ must also lie in this eigenspace for all $t$.

To solve for $l$, note that computation with \eqref{eq:reflection}, \eqref{eq:driftconstraint}, \eqref{eq:grad} yields the following expression for \(\mathbf{n}=\mathbf{n}(t)\) and $l=l(t)$:
\begin{equation}\label{eq:leqn}
2(\dot{l}\mathbf{n}-l\dot{\mathbf{n}})+2l({\Identity}-2\mathbf{n}\mathbf{n}^\top)A\mathbf{n}-2\mathbf{n}\mathbf{n}^\top\mathbf{c}\quad=\quad0\,.
\end{equation}
On the other hand, \eqref{eq:symmpart} and \eqref{eq:skewpart} yield 
\[
A\mathbf{n}\quad=\quad
\lambda_0\mathbf{n}+\dot{\mathbf{n}}\,.
\]
Substituting into \eqref{eq:leqn} and simplifying, 
\begin{equation}\label{eq:ldiffeq}
\dot{l}\quad=\quad
\lambda_0l+ \mathbf{n}^\top \mathbf{c}\,.
\end{equation}
Solving this equation, using the solution for $\mathbf{n}=\mathbf{n}(t)$ obtained from \eqref{eq:nform}, we get \eqref{eq:lform}.

Conversely, suppose there exists an eigenvalue $\lambda_0$ of $S$ such that the vectors $T^k(\x_0-\y_0)$ (for $0 \le k \le d-1$) all lie in the eigenspace of $S$ corresponding to $\lambda_0$. To prove the existence of a Markovian maximal coupling $(X,Y)$ starting from $\x_0$ and $\y_0$, we will show that \eqref{eq:driftconstraint} holds with $\mathbf{n}$ and $l$ as given in the theorem.

Clearly, for this choice of $\mathbf{n}$ and $l$, \eqref{eq:skewpart} and \eqref{eq:ldiffeq} hold. Using these, we obtain
$$
\dot{\mathbf{n}}\mathbf{n}^\top - \mathbf{n}\dot{\mathbf{n}}^\top\quad=\quad T\mathbf{n}\mathbf{n}^\top + \mathbf{n}\mathbf{n}^\top T
$$
and
$$
\dot{l}\mathbf{n}-l\dot{\mathbf{n}}\quad=\quad \lambda_0l\mathbf{n} + \mathbf{n}\mathbf{n}^\top\mathbf{c}-lT\mathbf{n}.
$$
Now, observe that $S\mathbf{n}=\lambda_0\mathbf{n}$ and
$$
\mathbf{n}^\top A \mathbf{n}\quad=\quad \mathbf{n}^\top S \mathbf{n} \quad=\quad \lambda_0.
$$
Using these, we can write
\begin{align*}
({\Identity}-2\mathbf{n}\mathbf{n}^\top)\mathbf{b}(F(t,\x))\quad&=\quad ({\Identity}-2\mathbf{n}\mathbf{n}^\top)(A({\Identity}-2\mathbf{n}\mathbf{n}^\top)\x+2lA\mathbf{n}+ \mathbf{c})\\
&=\quad ({\Identity}-2\mathbf{n}\mathbf{n}^\top)A({\Identity}-2\mathbf{n}\mathbf{n}^\top)\x - 2\lambda_0l\mathbf{n} + 2lT\mathbf{n} + ({\Identity}-2\mathbf{n}\mathbf{n}^\top)\mathbf{c}.
\end{align*}
Applying the above relations, the right hand side of \eqref{eq:driftconstraint} becomes
\begin{align*}
[2(T\mathbf{n}\mathbf{n}^\top + \mathbf{n}\mathbf{n}^\top T) + ({\Identity}-2\mathbf{n}\mathbf{n}^\top)A({\Identity}-2\mathbf{n}\mathbf{n}^\top)]\x + \mathbf{c} \quad&=\quad (A\x+\mathbf{c})\\
&\quad + [-2S\mathbf{n}\mathbf{n}^\top - 2\mathbf{n}\mathbf{n}^\top S + 4(\mathbf{n}^\top A \mathbf{n})\mathbf{n}\mathbf{n}^\top]\x,
\end{align*}
where we used $A=S+T$. Now, using $S\mathbf{n}=\lambda_0\mathbf{n}$ and $\mathbf{n}^\top A \mathbf{n}=\lambda_0$ again, we get
$$
-2S\mathbf{n}\mathbf{n}^\top - 2\mathbf{n}\mathbf{n}^\top S + 4(\mathbf{n}^\top A \mathbf{n})\mathbf{n}\mathbf{n}^\top\quad=\quad 0,
$$
and thus, \eqref{eq:driftconstraint} holds, proving the theorem.
\end{proof}
The following corollary is immediate from the above theorem.
\begin{cor}\label{cor:twodimlin}
If $d=2$, then under the hypotheses of Theorem \ref{thm:affine-drift}, $A$ is either a symmetric matrix or of the form $\lambda_0{\Identity}+T$ for some real scalar $\lambda_0$ and a skew-symmetric matrix $T$.
\end{cor}
\begin{proof}
If the skew-symmetric part $T$ of $A$ is non-zero, then $\x_0-\y_0$ and $T(\x_0-\y_0)$ are non-zero, mutually orthogonal vectors which lie in the eigenspace of $S$ corresponding to $\lambda_0$. Thus, this eigenspace is the whole of $\Reals^2$ and $S=\lambda_0{\Identity}$.
\end{proof}

Now, we state and prove the main theorem of this section. 
Recall the Local Perturbation condition {\LPC} described in the introduction.
\begin{thm}\label{thm:rigidity}
Assume (A1) and (A2) hold for a time-homogeneous Euclidean diffusion. 
Then {\LPC} holds if and only if there exist a real scalar $\lambda_0$, a skew-symmetric matrix $T$ and a vector $\mathbf{c} \in \Reals^d$ such that the diffusion drift is given by
\[
 \mathbf{b}(\x)\quad=\quad\lambda_0\x +T\x + \mathbf{c}
\]
for all $\x \in \Reals^d$.
\end{thm}
\begin{proof}
We need to show that the set of eigenvalues of $S(\x)$ for any $\x \in \Reals^d$ is the singleton $\{\lambda_0\}$ and the skew-symmetric part $T(\x)$ is a constant matrix $T$. Write 
\[
\Mirrors_0\quad=\quad\{H(\x,\y)\;:\;  \x \in \ball(\x_0,r), \y \in \ball(\y_0,r)\}\,.
\] 
Our approach is to choose an appropriate set of \emph{mirrors} $\Mirrors \subseteq  \Mirrors_0$ and then to consider the \emph{orbit} of a point $\z \in \Reals^d$ under repeated reflections in this set of mirrors, defined as 
\[
\mathcal{O}(\z)\quad=\quad
\left\{ \mathbf{w} \in \Reals^d\;:\; \text{ there exist } h_1,\dots, h_k \in \Mirrors \text{ such that } \mathbf{w} =h_k\dots h_1\z\right\}\,.
\] 
We then use the constraint relations between a point and its reflection obtained in Lemma \ref{lem:consteqns}.

This idea is made more precise in the following internal lemmas. 
\begin{lem}\label{lem:symmrig}
Under the hypotheses of Theorem \ref{thm:rigidity}, there exists $\lambda_0 \in \Reals$ such that $S(\x)=\lambda_0{\Identity}$ for all $\x \in \Reals^d$.
\end{lem}
\begin{proof}
Suppse that \(X\) and \(Y\) start at $\x \in \ball(\x_0,r)$ and $\y \in \ball(\y_0,r)$ respectively.
It follows from letting $t \downarrow 0$ in part (i) of Lemma \ref{lem:consteqns} that, for all $\z \in \Reals^d$, $S(\z)$ and $S(H(\x,\y)\z)$ have the same set of eigenvalues.
(Recall that $H(\x,\y)\z$ represents reflection of $\z$ in the hyperplane $H(\x,\y)$.)

Denote $\x^*=(\x_0+\y_0)/2$ and let $\mathbf{v}_1=\x_0-\x^*$. 
Extend $\mathbf{v}_1$ to a basis $\{\mathbf{v}_1,\dots, \mathbf{v}_d\}$. 
If \(\eps\) is sufficiently small then the linearly independent vectors $\mathbf{n}_i=\mathbf{v}_1+ \eps \mathbf{v}_i, \ i=1,\dots d$ are such that $\{\x^*+\mathbf{n}_i: i=1,\dots d\} \subset \ball(\x_0,r)$ and $\{\x^*-\mathbf{n}_i: i=1,\dots d\} \subset \ball(\y_0,r)$. 
Defining $\x_i=\x^*+\mathbf{n}_i$ and $\y_i=\x^*-\mathbf{n}_i$, it follows that $\x^* \in H(\x_i,\y_i)$ for all $i$. 
For each \(i\), consider maximally coupled diffusions begun at $(\x_i,\y_i)$: applying part (ii) of Lemma \ref{lem:consteqns} and letting $t \downarrow 0$, it follows that $\mathbf{n}_i$ is an eigenvector of $S(\x^*)$.
By construction, no $\mathbf{n}_i$ is orthogonal to any other $\mathbf{n}_j$.
Since $S(\x^*)$ is symmetric, it follows that $\{\mathbf{n}_i: i=1,\dots, d\}$ correspond to the same eigenvalue, say $\lambda_0$ and thus, $S(\x^*)=\lambda_0\Identity$.

Choosing the set of mirrors $\Mirrors = \Mirrors_0$, consider the {orbit} $\mathcal{O}(\x^*)$ of $\x^*$ in $\Mirrors$.
If  $\mathcal{O}(\x^*)=\Reals^d$, then the lemma follows from the previous observation that for any $\z \in \mathcal{O}(\x^*)$, the set of eigenvalues of $S(\z)$ agrees with that of $S(\x^*)$.

To see this, let $L$ be the line that passes through $\x_0$ and $\y_0$. 
Let ${\mathbf{v}_0= \tfrac{\x_0- \y_0}{|\x_0- \y_0|}}$. 
Write $\x_{\delta}=\x_0 + \delta\mathbf{v}_0$ and $\y_{\delta}=\y + \delta\mathbf{v}_0$ for all $\delta \in (-r,r)$. Thus the mirrors $h_{\delta}=H(\x_{\delta},\y_{\delta}) \in \Mirrors$ for all such $\delta$, and the orbit of $\x^*$ under reflection in $\{h_{\delta}: \delta \in (-r,r)\}$ is the whole of $L$.
Thus $L \subseteq \mathcal{O}(\x^*)$.

Now, for any $\z \in \Reals^d$, let $H$ be a plane (dimension of $H$ is two) containing the line $L$ and the point $\z$. 
For sufficiently small $\eps>0$, for all $\delta \in (-\eps,\eps)$ the mirror $h_{\delta}'$ containing $\x^*$ and having normal vector $\mathbf{v}_{\delta} \in H$ and making an angle $\delta$ with $\mathbf{v}_0$ lies in $\Mirrors$. 
Denote by $C$ the circle centred at $\x^*$, lying in $H$ and passing through $\z$. Let $\mathbf{\hat{z}} \in L \cap C$. 
Then the orbit of $\mathbf{\hat{z}}$ under reflection in $\{h_{\delta}': \delta \in (-\eps,\eps)\}$ is the whole of $C$. 
In particular, $\z \in \mathcal{O}(\x^*)$. This shows that $\mathcal{O}(\x^*)=\Reals^d$ and the lemma follows.
\end{proof}

Before proceeding further with the proof of Theorem \ref{thm:rigidity}, we record a general fact about real skew-symmetric matrices which follows by spectral decomposition  \citep{Gallier-2011}.
\begin{lem}\label{lem:Darboux}
If $\mathcal{N}$ is the null space of a $(d \times d)$ real skew-symmetric matrix $T$, then $d-\dim(\mathcal{N})$ is even.
\end{lem}

We now show that the skew-symmetric part $T(\x)$ is a constant matrix $T$.
\begin{lem}\label{lem:skewrig}
Under the hypotheses of Theorem \ref{thm:rigidity}, $T(\x)\equiv T$ for all $\x \in \Reals^d$.
\end{lem}
\begin{proof}
The proof breaks into three steps.
\begin{enumerate}
\item[\textbf{Step 1.}] \emph{If $\x \in \ball(\x_0,r)$ and $\y \in \ball(\y_0,r)$, then for all $\z,\z' \in H(\x,\y)$, $T(\z)=T(\z')$.}

Set ${\z^*=\tfrac{\z + \z'}{2}}$, ${\mathbf{v}_1=\tfrac{\z-\z'}{|\z-\z'|}}$ and ${\mathbf{v}_2=\tfrac{\x-\y}{|\x-\y|}}$. 
Extend $\mathbf{v}_1, \mathbf{v}_2$ to an orthonormal basis $\mathbf{v}_1,\dots, \mathbf{v}_d$ of $\Reals^d$. 
Using the method of the proof of Lemma \ref{lem:symmrig}, construct independent vectors $\mathbf{n}_i=\mathbf{v}_2 + \eps \mathbf{v}_i, \ i=2,\dots d$, 
choosing $\eps>0$ small enough so that 
\[
H(\z^*+\mathbf{n}_i, \z^*-\mathbf{n}_i)\x \quad\in\quad
\ball(\y_0,r)
\]
for all $i=2,\dots,d$. 
Writing $\x_i=\z^*+\mathbf{n}_i$ and $\y_i=\z^*-\mathbf{n}_i$, and with a possibly smaller choice of $\eps>0$, the hyperplane $H(\x_i,\y_i)$ lies in $\Mirrors_0$ and the line joining $\z$ and $\z'$ is contained in $H(\x_i,\y_i)$ for all $i=2,\dots,d$. 
Thus, $H(\x_i,\y_i)\z=\z$ and $H(\x_i,\y_i)\z'=\z'$ for all $i=2,\dots,d$. 
Taking $t \downarrow 0$ in part (iii) of Lemma \ref{lem:consteqns}, it follows that 
$$
(T(\z)-T(\z'))\mathbf{n}_i=0
$$
 for all $i=2,\dots,d$, implying $d-\mathcal{N}(T(\z)-T(\z')) \le 1$. Together with Lemma \ref{lem:Darboux}, this establishes Step 1.
\item[\textbf{Step 2.}] \emph{There is $\eps>0$ such that $T(\z)=T(\z')$ for all $\z,\z' \in \{\mathbf{w} \in \Reals^d\;:\; \dist(\mathbf{w}, H(\x_0,\y_0)) < \eps\}$, where $\dist(\mathbf{w},A)$ denotes the distance of $\mathbf{w}$ from the set $A$.}

Choose $\x \in \ball(\x_0,r)$ such that the vector $\x-\y_0$ is not parallel to $\x_0-\y_0$. 
It follows from Step 1 that $T(\z)=T(\z')$ for all $\z,\z' \in H(\x,\y)$. 
Choose $\eps>0$ such that ${\y_{\delta}=\y_0 +\delta\tfrac{\x_0- \y_0}{|\x_0- \y_0|} \in \ball(\y_0,r)}$
for all $\delta \in (-2 \eps, 2\eps)$. 
Note that 
the vector $\x-\y_\delta$ is not parallel to $\x_0-\y_0$ for any $\delta \in (-2 \eps, 2\eps)$.
Using Step 1 again, $T(\z)=T(\z')$ for all $\z,\z' \in H(\x_0,\y_{\delta})$. 
The assertion now follows from Step 1 and the fact that $H(\x_0,\y_{\delta}) \cap H(\x,\y)$ is non-empty for each $\delta \in (-2 \eps, 2\eps)$.
\item[\textbf{Step 3.}] Now we work with the set of {mirrors} 
$$
\Mirrors \quad=\quad \left\lbrace H(\x_0,\y_{\delta}): \delta \in (-2\eps, 2\eps)\right\rbrace\,,
$$ 
where $\eps$ is chosen as in Step 2. For notational convenience, we write $h_{\delta}=H(\x_0,\y_{\delta})$. 
The \(\y_\delta=\y_0 +\delta\tfrac{\x_0- \y_0}{|\x_0- \y_0|}\) all lie on the same line through \(\x_0\), and therefore
all these mirrors have a common normal vector, which we write $\mathbf{n}^*$. 
Let $(l_{\delta},\mathbf{n}_{\delta})$ parametrize the interface $I(\x_0, \y_{\delta}, \cdot)$ corresponding to the starting points $\x_0$ and $\y_{\delta}$ of the diffusions $X$ and $Y$ respectively. 
For each $\delta$, $\mathbf{n}_{\delta}(0)=\mathbf{n}^*$. 
Furthermore, by letting $t \downarrow 0$ in part (iii) of Lemma \ref{lem:consteqns}, 
\[
\dot{\mathbf{n}}_{\delta}(0)\quad=\quad T\left(\tfrac{\x_0+\y_{\delta}}{2}\right)\mathbf{n}^*\,.
\] 
Given $\delta \in (-2\eps, 2\eps)$, the distance of the point ${\tfrac{\x_0+\y_{\delta}}{2}}$ from the hyperplane $H(\x_0, \y_0)$ is less than $\eps$.
Consequently Step 2 implies that $\dot{\mathbf{n}}_{\delta}(0)=\dot{\mathbf{n}}_0(0)=\mathbf{n}'$ (say) for all $\delta \in (-2\eps, 2\eps)$.

Choose any $\z, \z' \in \Reals^d$ such that ${\z'=\z +\delta\tfrac{\x_0- \y_0}{|\x_0- \y_0|}}$ for some $\delta \in (-2 \eps, 2\eps)$. 
Set $\z^*=h_0\z$ so that $\z=h_0\z^*$. Noting that \(\z\), \(\z^*\), \(\z'\) lie on the same line perpendicular to $H(\x_0, \y_0)$, it follows from an argument about one-dimensional reflections that $\z'=h_{\delta}\z^*$. 

Then, by part (i) of Lemma \ref{lem:consteqns}, we get
\begin{eqnarray}\label{eqnarray:genskewrig}
T(\z^*)\quad&=\quad&2(\mathbf{n}'\mathbf{n}^{*\top}-\mathbf{n}^*\mathbf{n}'^\top)+({\Identity}-2\mathbf{n}^*\mathbf{n}^{*\top})T(\z)({\Identity}-2\mathbf{n}^*\mathbf{n}^{*\top})\nonumber\\
\quad \quad&=\quad&2(\mathbf{n}'\mathbf{n}^{*\top}-\mathbf{n}^*\mathbf{n}'^\top)+({\Identity}-2\mathbf{n}^*\mathbf{n}^{*\top})T(\z')({\Identity}-2\mathbf{n}^*\mathbf{n}^{*\top})
\end{eqnarray}
from which we get 
$$
({\Identity}-2\mathbf{n}^*\mathbf{n}^{*\top})(T(\z)-T(\z'))({\Identity}-2\mathbf{n}^*\mathbf{n}^{*\top})=0
$$

which gives $T(\z)=T(\z')$. Hence the lemma follows.
\end{enumerate}
\end{proof}

Lemmas \ref{lem:symmrig} and \ref{lem:skewrig} together are sufficient to prove Theorem \ref{thm:rigidity}.
\end{proof}

Theorem \ref{thm:rigidity} can be strengthened if $\dot{\mathbf{n}}(t)=0$ for all $t$, i.e., the interface \emph{translates but does not rotate in time}.
We state this in the following theorem.
Since there is no rotation, the driving Brownian motions in the stochastic differential equation for $X$ and $Y$ are constant reflections of each other.
So we can assume without loss of generality that $l(0)=0$ and $\mathbf{n}(t) \equiv \mathbf{e}_1$.
\begin{thm}\label{thm:parallelmirr}
Assume (A1) and (A2) hold for a time-homogeneous Euclidean diffusion. 
Suppose there exists a Markovian maximal coupling of $X$ and $Y$ starting from $\x_0$ and $\y_0$ respectively, 
such that the interface $I(\x_0,\y_0,t)$ is parametrized by $((l(t),\mathbf{e}_1)\;:\; t \ge 0)$ with $l(0)=0$. Then there are only two possibilities:
\begin{itemize}
\item[(i)] $l(t) =0$ for all $t\ge0$, in which case the drift vectorfield $\mathbf{b}$ must satisfy 
\[
\mathbf{b}(h_1\x)\quad=\quad h_1\mathbf{b}(\x)
\] for all $\x \in \Reals^d$.
\item[(ii)] $l(t) \neq 0$ for some $t>0$, in which case the drift vectorfield $\mathbf{b}$ must satisfy 
\[
\mathbf{b}(x_1, \x^{(1)})\quad=\quad
\left(c_1x_1 + c_2, \mathbf{f}(\x^{(1)})\right)^\top
\]
for all $\x=(x_1,\x^{(1)}) \in \Reals^d$, where $c_1, c_2$ are constants and $\mathbf{f}: \Reals^{d-1} \rightarrow \Reals^{d-1}$ is continuously differentiable.
\end{itemize}
\end{thm}
\begin{proof}
Part (i) follows from the fact that the generators of $Y$ and $h_1X$ are the same.

To prove part (ii), note that by part (i) of Lemma \ref{lem:consteqns}:
\begin{equation}\label{eq:bformfixed}
\nabla\mathbf{b}(x_1, \x^{(1)})\quad=\quad
\begin{bmatrix}
\partial_1b_1(x_1,\x^{(1)}) & \mathbf{0}\\
\mathbf{0} & \nabla^{(1)}\mathbf{b}^{(1)}(x_1, \x^{(1)})
\end{bmatrix}.
\end{equation}
for all $\x=(x_1,\x^{(1)}) \in \Reals^d$, where $\mathbf{b}^{(1)}=(b_2,\dots,b_d)^\top$ and $\nabla^{(1)}$ denotes partial derivatives with respect to the variables of $\x^{(1)}$. From \eqref{eq:bformfixed}, we deduce that $b_1(x_1, \x^{(1)})=f_1(x_1)$ and $\mathbf{b}^{(1)}(x_1, \x^{(1)})=\mathbf{f}(\x^{(1)})$ for continuously differentiable functions $f_1: \Reals^{d-1} \rightarrow \Reals$ and $\mathbf{f}: \Reals^{d-1} \rightarrow \Reals^{d-1}$. 

We may assume that (without loss of generality) $(0,\eps) \subset \operatorname{Range}(l)$ for some $\eps>0$. Choose the set of mirrors 
\[
\Mirrors\quad=\quad
\{H(\x_0, \y_{\delta}): \delta \in (0,\eps)\}
\]
where, as before, ${\y_{\delta}=\x_0 +\delta\tfrac{\x_0- \y_0}{|\x_0- \y_0|}}$.
Now, iterated reflections in $\Mirrors$ as in the proof of Theorem \ref{thm:rigidity} yield $f_1'(x_1+a)=f_1'(x_1)$ for all $x_1, a \in \Reals$. Hence, $f_1'(x_1)=c_1$ for all $x_1 \in \Reals$, for some constant $c_1$.
Thus, $\mathbf{b}$ has to be of the required form.
\end{proof}
The case of one-dimensional diffusions is a trivial consequence of the above theorem, as noted in the next corollary.
\begin{cor}\label{cor:onedim}
Assume (A1) and (A2) hold for a one-dimensional time-homogeneous Euclidean diffusion. 
Then there exists a Markovian maximal coupling of $X$ and $Y$ starting from $x_0$ and $y_0$ respectively
if and only if either the drift vectorfield $b$ is affine or it obeys the reflection symmetry $b(x)=-b(x_0+y_0-x)$ for all $x \in \Reals$.
\end{cor}
\begin{rem}
 Corollary \ref{cor:onedim} completely characterises all one-dimensional time-homogeneous diffusions subject to the regularity conditions (A1) and (A2) and permitting Markovian maximal couplings, 
 even with a varying twice-continuously-differentiable diffusion coefficient $\sigma(\cdot):\Reals \rightarrow [c,\infty)$ for some $c>0$. 
 Let $X$ be given by
\begin{equation}\label{eq:onedimstochastic differential equation}
\d X_t\quad=\quad
b(X_t) \d t+\sigma(X_t)\d B_t
\end{equation}
and similarly for $Y$. Define the function 
\[
F(x)\quad=\quad
\int_0^x\frac{1}{\sigma(z)}\d z\,,
\]
and set $U_t=F(X_t)$. Then, it follows from It\^o calculus that
\begin{equation}\label{eq:sigmaone}
\d U_t\quad=\quad \d B_t + \left(\frac{b \circ F^{-1}(U_t)}{\sigma\circ F^{-1}(U_t)}-\frac{\sigma'\circ F^{-1}(U_t)}{2}\right)\d t\,.
\end{equation}
Thus, the conditions on $b$ derived in the case $\sigma \equiv 1$ readily carry over to conditions on the drift term of (\ref{eq:sigmaone}) for general $\sigma$.
\end{rem}

\section{Markovian Maximal Couplings for manifolds}\label{sec:manifold}
In this section, we analyse rigidity phenomena for Markovian maximal couplings (MMC) for smooth elliptic diffusions, 
and demonstrate that there are powerful geometric consequences arising from a natural connection to the theory of diffusion processes on manifolds (specifically, the notion of Riemannian Brownian motion with drift). 
The main task of this section is to understand how the Euclidean arguments of section \ref{sec:Euclidean} carry over to the manifold case. 
In particular, the existence of Markovian maximal couplings (together with {\LPC}) has profound rigidity consequences for the geometry of the manifold.

We commence by summarizing the Riemannian geometry required to establish these consequences.
Let $M$ be a connected smooth manifold of dimension $d$ (the results which follow are actually significant even in the case when \(M=\Reals^d\)). 
Following \cite{Dynkin-1963}, 
a strong Markov process $X$ on $M$ is said to be a diffusion process if each $C^2$ function $f$ belongs to the domain of definition of the \emph{characteristic operator} $\generator$ given by
\begin{equation}
\generator f(\x)\quad=\quad 
\lim_{N \downarrow \x}\frac{\mathbb{E}_{\x}\left[f(X_{\tau_N})\right]-f(\x)}{\mathbb{E}_{\x}[\tau_N]}
\end{equation}
where $N$ denotes a system of neighbourhoods shrinking to $\x$, $\tau_N$ denotes the first exit time from $N$ and $\mathbb{E}$ denotes expectation with respect to the measure induced by the Markov process. In any local system of coordinates $(x^1,\dots,x^d)$, the operator $\generator$ takes the form
\begin{equation}\label{eq:generatorloc}
\generator f(\x)\quad=\quad
\sum_{i,j=1}^da_{ij}(\x)\frac{\partial^2f}{\partial x^i \partial x^j} + \sum_{i=1}^dv_i(\x)\frac{\partial f}{\partial x^i}
\end{equation}
where the \emph{diffusion matrix} $A=\{a_{ij}\}$ is non-negative definite and $\{v_i\}$ denotes the \emph{drift vectorfield}. We will assume $a_{ij}$ and $v_i$ are smooth functions. 
Note that the general form of the operator does not depend on the specific choice of coordinates. 
We call $X$ an \emph{elliptic} diffusion if $\generator$ is an elliptic operator (in other words, if \(A\) is positive-definite). As in the previous section, we deal only with elliptic diffusions.\label{page:hypoelliptic-issue}

Following \cite{Molchanov-1975}, if we furnish $M$ with the Riemannian metric $g$ which is given in local coordinates by $g_{ij}=(A^{-1})_{ij}$ then the operator $\generator$ can be rewritten in the form 
\begin{equation}
\generator\quad=\quad \frac{1}{2}\Delta_M + \mathbf{b}
\end{equation}
where $\Delta_M$ is the \emph{Laplace-Beltrami operator} for the Riemannian metric, and $\mathbf{b}$ is the (intrinsic) drift vectorfield. 
When $\mathbf{b}=0$, the corresponding Markov process is called \emph{Brownian motion} on $M$. 
Thus, we see that any diffusion process on $M$ can be written as `Brownian motion plus drift' if $M$ is given a suitable metric. 
Henceforth, we will assume that $M$ is endowed with this metric $g$, so that we can view \(M\) as a smooth Riemannian manifold $(M,g)$.

\textbf{Note: }Throughout this section, we will make the following assumptions:
\begin{itemize}
\item[(i)] The Riemannian manifold $(M,g)$ obtained above is \textbf{complete}
(we say that the diffusion \(X\) is \emph{diffusion-geodesic complete}). 
This is a purely technical assumption and the completeness is usually not too hard to check as we know the diffusion coefficients explicitly.
In particular, diffusion-geodesic completeness trivially holds on compact manifolds. 
Diffusion-geodesic completeness is not a necessary condition for the existence of Markovian maximal couplings,
as can be seen for dimension \(d \ge 2\) by considering reflection couplings of Brownian motions on the $d$-dimensional punctured sphere $\mathbb{S}^d - \{P\}$ obtained by deleting a point $P$ from the sphere $\mathbb{S}^d$ (and the corresponding couplings of diffusions obtained on the plane by stereographic projection).
In this example, the existence of a rich supply of MMC follows from the fact that this space has a completion $\mathbb{S}^d$ 
on which we can construct MMC of Brownian motions started from any two points (see \cite{Kuwada-2009}), 
and from the fact that if \(d \ge 2\) then the Brownian motion started in $\mathbb{S}^d - \{P\}$ almost surely does not hit $P$.
It is an interesting question whether this is the `generic' example for instances where diffusion-geodesic completeness fails but Markovian maximal couplings exist,
raising issues which seem somewhat reminiscent of the topic of resolution of singularities in algebraic geometry.
We
hope to address this in a future article.
\item[(ii)] Our diffusion process $X$ is \textbf{defined for all time}.
This is to ensure that we are dealing with probability densities which is essential for the arguments in subsection \ref{sec:genprop} to go through. 
For Brownian motion on $M$, this can be resolved by ensuring that $M$ is \emph{stochastically complete}. 
There are a number of intrinsic geometric properties of $M$ that ensure stochastic completeness,
such as the existence of a constant lower bound on the Ricci curvature. See \cite{Hsu-2002}, for example, for more details.
\end{itemize}

Let $\mathcal{G}=\operatorname{Iso}(M)$ denote the \emph{group of (global) isometries} of $M$. 
This can be shown to be a Lie group \citep{MyersSteenrod-1939}, and it plays an important r\^ole in the following arguments. 
As $M$ is complete and connected, any pair of points in $M$ are connected by a geodesic. 
Furthermore, there are no branching geodesics in Riemannian manifolds. 
(More details on these geometric notions can be found in \citealp{BuragoBuragoIvanov-2001,Chavel-1995}.)


\subsection{Brownian motion with drift on the manifold}
Not only can any smooth elliptic diffusion on $M$ be written as Brownian motion with drift on $(M,g)$, 
but also this permits a rather explicit geometric construction of the diffusion which facilitates the discussion of probabilistic coupling techniques, 
namely the \emph{Eells-Elworthy-Malliavin} construction \citep{Elworthy-1982}. 

Using terminology expounded (for example) in \cite{Hsu-2002},
let $\mathcal{O}_x(M)$ denote the set of orthonormal frames of the tangent space $T_xM$. The \emph{orthonormal frame bundle} 
\[
\mathcal{O}(M)\quad=\quad\bigcup_{x \in M}\mathcal{O}_x(M)
\]
possesses a natural smooth manifold structure of dimension $\frac{d(d+1)}{2}$. 
Denote the canonical projection map by $\pi: \mathcal{O}(M) \rightarrow M$.

A curve $u$ in $\mathcal{O}(M)$ is said to be \emph{horizontal} if $u_t$ is the parallel transport (associated with the \textit{Levi-Civita} connection)
of the frame $u_0$ along the curve $\pi u_t$. 
For each $u \in \mathcal{O}(M)$, the tangent space $T_u\mathcal{O}(M)$ can be expressed as a direct sum 
\[
T_u\mathcal{O}(M)\quad=\quad V_u\mathcal{O}(M) \bigoplus H_u\mathcal{O}(M)\,,
\]
where $V_u\mathcal{O}(M)$ is a $\frac{d(d-1)}{2}$-dimensional vector space corresponding to the isotropy group (frame rotations) at $\pi u$,
and the $d$-dimensional vector space $H_u\mathcal{O}(M)$ is the space of tangent vectors of horizontal curves passing through $u$.

For each $u \in \mathcal{O}(M)$, let $H_i(u)$ denote the unique horizontal vector lying in  $H_u \mathcal{O}(M)$ such that 
\[
\pi_*H_i(u)\quad=\quad ue_i\,,
\]
where $ue_i$ denotes the $i$-th unit vector of the orthonormal frame $u$.

This framework provides an expressive way to define smooth elliptic diffusions (and other semimartingale processes) on $M$, as follows. 

Let $\mathbf{b}$ be a smooth vectorfield on $M$. This yields a natural vectorfield $\mathbf{B}$ on $\mathcal{O}(M)$ given by 
\begin{equation}\label{eq:stratdrift}
\mathbf{B}(u)\quad=\quad \sum_i b_i(u)H_i(u)\,,
\end{equation}
 where $b_i(u)= \langle\mathbf{b}(\pi u), ue_i\rangle_{\pi u}$ (here $\langle \cdot, \cdot  \rangle$ denotes the Riemannian inner product). We will call this the \emph{lifted drift}. Consider the following Stratonovich differential equation on $\mathcal{O}(M)$:
\begin{equation}\label{eq:framestochastic differential equation}
\d U_t\quad=\quad \sum_i H_i(U_t) \circ \d W^i_t + \mathbf{B}(U_t)\d t\,.
\end{equation}
where $W$ is a \(d\)-dimensional Euclidean Brownian motion. The diffusion on $M$ with drift $\mathbf{b}$ is obtained simply as the projection $X_t= \pi U_t$. The pivotal fact justifying this construction is that we can define a second order operator on $\mathcal{O}(M)$ (Bochner's horizontal Laplacian) given by 
\[
\Delta_{\mathcal{O}(M)}\quad=\quad \sum_{i=1}^dH_i^2
\]
such that the Laplace-Beltrami operator $\Delta_M$ on $M$ satisfies 
\[
\Delta_Mf(\x)\quad=\quad\Delta_{\mathcal{O}(M)}f\circ \pi(u)
\]
 for any $u \in \mathcal{O}(M)$ such that $\pi u=\x$. 
 The generator $\generator$ of the diffusion $X$ defined at the start of section \ref{sec:manifold} satisfies
\begin{equation}\label{eq:generatorman}
\generator f(\x)\quad=\quad\frac{1}{2}\Delta_Mf(\x) + \mathbf{b}f(\x)
\end{equation}
for any $u \in \mathcal{O}(M)$ such that $\pi u=\x$, and any $C^2$ test function $f$ on $M$.

Note that, when $\mathbf{b}=0$, the above construction reduces to the classical Eells-Elworthy-Malliavin construction of Brownian motion on $M$. 

\subsection{Couplings of diffusions on manifolds}
Once we have the above construction, a natural question to ask is: when is there a Markovian maximal coupling (MMC) for two copies of the diffusion starting from $\x_0$ and $\y_0$? 
In the Euclidean case there is a complete characterization of the class of time-homogeneous diffusions under {\LPC}, which is to say, 
when two copies of the diffusion can be maximally coupled whenever they start from $\x \in \ball(\x_0, r)$ and $\y \in \ball(\y_0, r)$ 
(for $\ball(\x_0, r)$ and $\ball(\x_0, r)$ chosen to be two arbitrary disjoint open balls in $\Reals^d$).
Theorem \ref{thm:rigidity} shows that the class of such diffusions is actually very small. 

The proof of Theorem \ref{thm:rigidity} depends strongly on a wealth of isometries of Euclidean space arising \emph{via} iterated reflections. 
Very few other $d$-dimensional Riemannian manifolds have many isometries, 
and so we may expect an even stronger rigidity phenomenon to hold for the geometry of (non-Euclidean) manifolds on which there is a good supply of MMC. 
The work of this section substantiates this expectation.

We begin by recalling briefly some notions from the Euclidean case (section \ref{sec:Euclidean}).
We have noted that the Local Perturbation Condition {\LPC} (Definition \ref{def:LPC}) makes sense for any metric space, including the Riemannian manifold case.
Let $X$ and $Y$ be two copies of the elliptic diffusion derived from the stochastic differential equation (\ref{eq:framestochastic differential equation}), and starting from $\x_0$ and $\y_0$ respectively. 
Note that the assumptions of ellipticity and smoothness of the coefficients of $L$ together 
ensure that the law of $X$ (equivalently $Y$) has a smooth positive density with respect to the Riemannian volume measure $m$ for every positive time $t>0$, which we write 
as $p(\x_0;t,\z)$, $p(\y_0;t,\z)$ for $t > 0$, $\z \in M$.

We suppose that the standing assumptions of diffusion-geodesic completeness and stochastic completeness both hold for the regular elliptic diffusion \(X\),
so that the resulting Riemannian manifold \(M\) is geodesically complete and so that \(X\) stays on \(M\) for all time. 
Thus from here on we are considering the case of Brownian motion with non-explosive drift on a complete Riemannian manifold.

We note here that all the results in subsection \ref{sec:genprop} carry over to the manifold setting with $(M,\dist)$ 
being the Riemannian manifold (with the distance $\dist$ induced by the Riemannian metric) and $m$ taken to be the volume measure.

\subsection{The interface}\label{sec:interfaceman}
Varadhan small-time asymptotics and Lemma \ref{lem:flow} can be used to show the following:
that the existence of an MMC implies that, for each time $t$, there is a \emph{deterministic involutive isometry} $F_t$ which exchanges $X_t$ with $Y_t$ and fixes the set of points equidistant from both $X_t$ and $Y_t$. 
This generalizes the time-varying reflection isometry of Euclidean space which is mentioned in Remark \ref{rem:time-varying-mirror}; 
the fixed-point set of \(F_t\) corresponds to the `evolving mirror' of the Euclidean case.

The r\^ole of Varadhan's small-time asymptotics in the following is analogous to the r\^ole of Lemma \ref{lem:LD} in the Euclidean case. 
This powerful technique gives the logarithmic asymptotics of the density of $X_t$ when $t \downarrow 0$, as stated in the following lemma.
\begin{lem}\label{lem:varadhan}
Suppose that \(X\) satisfies the assumptions of both diffusion-geodesic completeness and stochastic completeness.
Let $M_1$ and $M_2$ be compact subsets of $M$. Then the density $p$ of $X_t$ satisfies the following:
\begin{equation}\label{eq:varadhan}
\lim_{t \downarrow 0}\;2t\, \log p(\x;t,\y)\quad=\quad-\dist^2(\x,\y)
\end{equation}
uniformly for all $\x, \y \in M_1 \times M_2$, where $\dist(\x,\y)$ is the Riemannian distance between \(\x\) and \(\y\).
\end{lem}
This theorem was proven by \citet{Varadhan-1967} for diffusion processes on Euclidean space. 
Later \citet{Molchanov-1975} noticed that \citeauthor{Varadhan-1967}'s arguments carry over to diffusions on closed manifolds whose generators are of the form $\generator=\frac{1}{2}\Delta_M + \mathbf{b}$. 
\citeauthor{Molchanov-1975} also showed that this result could be extended to general smooth complete manifolds by introducing a reflected diffusion in a suitably large domain $U \subset M$ containing $\x$ and $\y$, 
with the same generator $\generator$ inside, and using this process to define a natural diffusion on the `double' $U$. 
He then showed that smoothing techniques allowed the approximation of the `double' $U$ by a smooth closed manifold, 
such that the diffusion thus defined has a density that is sufficiently close to that of the original one \citep[p.~18 and further references]{Molchanov-1975}.

We can now restate the pivotal Theorem \ref{thm:zeroset} from subsection \ref{subsection:interface} in the new context of manifolds. 
The proof of the manifold case follows that of the Euclidean case, but uses Lemma \ref{lem:varadhan} in place of Lemma \ref{lem:LD}, and uses the strong maximum principle (Lemma \ref{lem:StrongMP}) in local coordinates; 
we omit details.
\begin{thm}\label{thm:zerosetman}
Suppose that the standing assumptions of diffusion-geodesic completeness and stochastic completeness both hold.
For any $(\x,\y) \in \mathcal{M}(\mu_s)$, and any \(s>0\), the following equalities hold:
\begin{align*}
I(\x_0,\y_0,s)\quad&=\quad H(\x,\y),\\
I^{-}(\x_0,\y_0,s)\quad&=\quad H^-(\x,\y),\\
I^{+}(\x_0,\y_0,s)\quad&=\quad H^+(\x,\y).
\end{align*}
\end{thm}
Let $\tau'=\inf\{s>0: X_s \in I(\x_0,\y_0,s)\}$ be the first time that \(X\) hits the interface. 
Then the following holds.
\begin{cor}\label{cor:alltimeman}
Almost surely $\tau'=\tau$, so coupling occurs when \(X\) first hits the interface. 
Furthermore, $\mu$-almost surely, for all \(t < \tau\),
\begin{equation}\label{eq:constman}
I(\x_0,\y_0,t)=H(X_t,Y_t)\,, \;
I^{-}(\x_0,\y_0,t)=H^-(X_t,Y_t)\,, \;
I^{+}(\x_0,\y_0,t)=H^+(X_t,Y_t).
\end{equation}
\end{cor}
\begin{proof}
The proof follows the lines of the proof of Corollary \ref{cor:alltime}. 
The only additional detail that we have to check here (which was immediate in the Euclidean case) is that, 
for any $t > 0$ with $X_t \neq Y_t$, any $\z \in H(X_t,Y_t)$ and any rational sequence $t_n \downarrow t$, there is $\z_n \in H(X_{t_n},Y_{t_n})$ such that $\z_n \rightarrow \z$. 
This was used in Corollary \ref{cor:alltime} to show $H(X_t,Y_t) \subseteq I(\x_0,\y_0,t))$.

Recall the event $E=\cap_{q \in Q}E_q$, where $E_q$ was defined in (\ref{eq:event}). 
Assume $E$ holds.
For notational convenience, denote $H(X_t,Y_t), X_t, Y_t$ by $H,\x,\y$ and  $H(X_{t_n},Y_{t_n}), X_{t_n}, Y_{t_n}$ by $H_n,\x_n,\y_n$ respectively. 
Let $\gamma:[0,2\dist(\x,\z)] \rightarrow M$ denote the continuous curve such that $\gamma \mid_{[0,\dist(\x,\z)]}$ is a minimal geodesic joining $\x$ and $\z$ and $\gamma \mid_{[\dist(\x,\z),2\dist(\x,\z)]}$ is a minimal geodesic joining $\z$ and $\y$. 
As $M$ has no branching geodesics, it follows that $\dist(\x,\gamma(s)) < \dist(\y,\gamma(s))$ for any $s \in [0,\dist(\x,\z))$.
Consequently for any $\delta>0$, by the compactness of $\{\gamma(s) : s \in [0,\dist(\x,\z)-\delta]\}$, $\min_{s \in [0,\dist(\x,\z)-\delta]}(\dist(\y, \gamma(s))-\dist(\x, \gamma(s))) > 0$ and hence, $\min_{s \in [0,\dist(\x,\z)-\delta]}(\dist(\y_n, \gamma(s))-\dist(\x_n, \gamma(s))) > 0$ for sufficiently large $n$. Thus, for sufficiently large $n$, $\gamma(s) \in H^-(\x_n,\y_n)=I^-(\x_0,\y_0,t_n)$ for all $s \in [0,\dist(\x,\z)-\delta]$ and consequently, $\min_{s \in [0,\dist(\x,\z)-\delta]}\alpha(t_n,\gamma(s))>0$. 
Similarly, $\min_{s \in [\dist(\x,\z)+\delta,2\dist(\x,\z)]}\alpha(t_n,\gamma(s))<0$ for sufficiently large $n$. 
Thus, as $E$ holds, the continuity of $\alpha(t_n, \cdot)$, implies that for sufficiently large $n$, there is $\z_n \in \gamma \cap H_n$ such that $\z_n \rightarrow \z$. As $\mu(E)=1$, this implies $H(X_t,Y_t) \subseteq I(\x_0,\y_0,t))$ almost surely.

The rest of the proof carries over verbatim from that of Corollary \ref{cor:alltime}.
\end{proof}

The striking fact that emerges from the above is that, almost surely under the coupling $\mu$, for each $s>0$, $H(X_t,Y_t)$ is a \emph{non-random set which depends only on $s$} and not on the specific location of $(X_t,Y_t)$. 
We will call this set $H_t$ henceforth.
Similarly, denote $H^+_t=H^+(X_t,Y_t)$ and $H^-_t=H^-(X_t, Y_t)$. 
The family $\{H_t: t \ge 0\}$ corresponds to the family of \emph{moving mirrors} from section \ref{sec:Euclidean}.

We now follow \citet{Kuwada-2009}'s construction to define a deterministic global involutive isometry $F_s$ which fixes $H_s$ and maps $X_s$ to $Y_s$ under the coupling. 
The argument of \citet[Lemma 4.6]{Kuwada-2009} applies directly to our case: we therefore omit proof.
\begin{lem}\label{lem:Isom1}
Suppose that the standing assumptions of diffusion-geodesic completeness and stochastic completeness both hold.
Take $s \ge 0$. If $\x, \y \in M$, with $\x \neq \y$, satisfies
\begin{equation}\label{eq:Isoeqn}
\dist(\x,\z)\quad=\quad \dist(\y,\z)
\end{equation}
for all $\z \in H_s$, then $(\x,\y) \in H^+_s \times H^-_s \cup H^-_s \times H^+_s$ (so \(\x\) and \(\y\) lie in opposite ``half-manifolds''). 
Furthermore, for any $\x \in M$, a point $\y \in M \backslash\{\x\}$ satisfying (\ref{eq:Isoeqn}) is unique if it exists.
\end{lem}
Whenever such a $\y$ exists, we will call $\y$ the \emph{mirror image} of $\x$ at time $s$.
With the aid of the above lemma, the isometry $F_s$ is constructed using a procedure which is similar to \citet[Theorem 4.5]{Kuwada-2009}, 
but is subject to some modification as described in the following lemma and its proof. 
\begin{lem}\label{lem:globiso}
Suppose that the standing assumptions of diffusion-geodesic completeness and stochastic completeness both hold.
Assume $(X,Y)$ is a Markovian maximal coupling with starting points $\x_0$ and $\y_0$.
Then, for each $s \in [0,\tau)$, there is a deterministic involutive isometry $F_s$ with fixed point set $H_s$ such that $Y_s=F_s(X_s)$, 
furthermore $F_s(H^-_s)=H^+_s$.
\end{lem}
\begin{proof}
Define the set 
\[
A_s\quad=\quad
\left\{\x \in M \;:\; \text{ there exists } \y \in M \backslash \{\x\} \text{ such that } (\ref{eq:Isoeqn}) \text{ holds}\right\}\,.
\]
For $\x \in A_s$, define $F_s(\x)$ to be the unique $\y$ for which (\ref{eq:Isoeqn}) holds. 
For $\x \in H_s$, define $F_s(\x)=\x$. 
Following the proof of \citet[Theorem 4.5]{Kuwada-2009}, the set $\hat{A_s}=A_s \cup H_s$ is closed. 
Furthermore, by Theorem \ref{thm:zerosetman} and Lemma \ref{lem:disjsupp}, on the event $[0<s < \tau]$ the support of $X_s$ (equivalently $Y_s$) is the whole of $H^-_s$ (respectively $H^+_s$). 
This, by Lemma \ref{lem:flow} and Theorem \ref{thm:zerosetman}, implies $\hat{A}_s=M$ for all $s>0$. 

A little more argument is required for $s=0$. 
By Theorem \ref{thm:zerosetman}, Lemma \ref{lem:varadhan} and Lemma \ref{lem:disjsupp}, for any $\x \in H^-_0$, there is a sequence $t_n \downarrow 0$ and $\x_n \rightarrow \x$ such that $\x_n \in A_{t_n}$ with $\y_n \in M$ being its \emph{mirror image} at time $t_n$, 
for all $n$. 
Take any $\z_0 \in H_0$. 
Following the proof of Corollary \ref{cor:alltimeman}, for sufficiently large $n$, there is $\z_n \in H_{t_n}$ such that $\z_n \rightarrow \z_0$. 
As $\dist(\x_n,\z_n)=\dist(\y_n,\z_n)$, it follows that the set of distances $\{\dist(\z_0, \y_n)\}_{n \ge 1}$ is bounded.
Consequently the properness of $M$ implies that there is a subsequence $\{n_k\}$ such that $\y_{n_k} \rightarrow \y$ for some $\y \in M$. 
Now, for any $\z \in H_0$, take $\z'_n \in H_{t_n}$ such that $\z'_n \rightarrow \z$. 
Thus,
\begin{eqnarray*}
\dist(\y,\z) \quad&=\quad& \lim_{k \rightarrow \infty}\dist(\y_{n_k},\z'_{n_k})\\
\quad&=\quad& \lim_{k \rightarrow \infty}\dist(\x_{n_k},\z'_{n_k})=\dist(\x,\z)\,.
\end{eqnarray*}
This implies $\hat{A}_0=M$. Note that, by Lemma \ref{lem:Isom1}, the limit $\y$ is uniquely determined by $\x$ and $H_0$, and thus, does not depend on the subsequence chosen. This implies $\y_n \rightarrow \y$. Define $F_0(\x)=\y$.

Thus $F_s$ is defined on the whole of $M$ for every $s \ge 0$. Continuity of $F_s$ for $s \ge 0$ follows exactly along the lines of the proof of continuity of the map $R$ in \citet[Theorem 4.5]{Kuwada-2009}. Further, by definition, $F_s$ is involutive. Thus, in particular, $F_s$ is an open map. 

To prove that $F_s$ is, in fact, an isometry, we have to modify the proof of \citet[Lemma 5.3]{Kuwada-2009} appropriately, as we outline in the following. 

First, consider $s>0$. If $\x, \y \in H_s$ or $\x \in H^-_s, \y \in H^+_s$, then $\dist(\x, \y)= \dist(F_s(\x), F_s(\y))$ follows from the definition of $F_s$. So, assume $\x , \y \in H^-_s$. Take $\delta>0$ small enough such that
$$
\overline{\ball(\x, \delta)} \subset H^-_s, \, \overline{\ball(\y, \delta)} \subset H^-_s, \, \overline{\ball(F_s(\x), \delta)} \subset H^+_s, \, \overline{\ball(F_s(\y), \delta)} \subset H^+_s.
$$
Let
$$
V_1= \ball(\x, \delta) \cap F_s(\ball(F_s(\x), \delta)), \, V_2= \ball(\y, \delta) \cap F_s(\ball(F_s(\y), \delta)), \, U_2= \ball(\y, \delta/2) \cap F_s(\ball(F_s(\y), \delta/2)).
$$
For $t >0$, by the strong Markov property, Corollary \ref{cor:alltimeman} and Lemma \ref{lem:disjsupp}, we have
\begin{equation}\label{eq:large1}
\mu(X_{s+t} \in U_2, X_s \in V_1, \tau> s+t)= \int_{V_1}\alpha^+(s, \z) \left\lbrace \int_{U_2}\left(p(\z;t,\mathbf{w})-p(F_s(\z);t,\mathbf{w}))m(d\mathbf{w}\right)\right\rbrace m(d\z).
\end{equation}
Similarly,
\begin{align}\label{align:large2}
\mu(Y_{s+t} \in F_{s+t}(U_2), & Y_s \in F_s(V_1), \tau> s+t)\nonumber\\
& = \int_{F_s(V_1)}\alpha^-(s, \z) \left\lbrace \int_{F_{s+t}(U_2)}\left(p(\z;t,\mathbf{w})-p(F_s(\z);t,\mathbf{w}))m(d\mathbf{w}\right)\right\rbrace m(d\z).
\end{align}
Observe that if $\z, \mathbf{w} \in H^-_s$ or $\z, \mathbf{w} \in H^+_s$, then $\dist(\z, \mathbf{w}) < \dist(F_s(\z), \mathbf{w})$. To see this, let $\gamma$ be the minimal geodesic joining $\mathbf{w}$ and $F_s(\z)$ and let $\z_0 \in \gamma \cap H_s$. Then
$$
\dist(\z, \mathbf{w}) \le \dist(\z, \z_0) + \dist(\z_0, \mathbf{w}) = \dist(F_s(\z), \z_0) + \dist(\z_0, \mathbf{w}) = \dist(F_s(\z), \mathbf{w}).
$$
If equality holds in the first inequality above, then we can take a minimal geodesic joining $\z$ and $\mathbf{w}$ that branches from $\gamma$ at $\z_0$ which gives a contradiction.

Next, we claim that there is $\epsilon>0$ such that for $t \in [0,\epsilon]$, $F_{t+s}(U_2) \subseteq F_s(V_2)$. Suppose not. Then there is a sequence $t_n \downarrow 0$ and $\x_n \in U_2$ such that $\y_n = F_{s+t_n}(\x_n) \in F_s(V_2^c)$. As $U_2$ is bounded, we obtain a subsequence $n_k$ such that $\x_{n_k} \rightarrow \x^o \in \overline{U_2}$ as $k \rightarrow \infty$. Take any $\z^o \in H_s$. Following the proof of Corollary \ref{cor:alltimeman}, for sufficiently large $n$, there is $\z^o_n \in H_{s+t_n}$ such that $\z^o_n \rightarrow \z^o$. As $\dist(\x_{n_k}, \z^o_{n_k})= \dist(\y_{n_k}, \z^o_{n_k})$,
\begin{align*}
\dist(\y_{n_k}, \z^o) & \le \dist(\y_{n_k}, \z^o_{n_k}) + \dist(\z^o_{n_k},\z^o) = \dist(\x_{n_k}, \z^o_{n_k}) + \dist(\z^o_{n_k},\z^o)\\
& \le \dist(\x_{n_k}, \x^o) + \dist(\x^o, \z^o) + 2\dist(\z^o_{n_k},\z^o).
\end{align*}
Thus, $\y_{n_k}$ is bounded and we can extract a further subsequence $n_{k_l}$ such that $\y_{n_{k_l}} \rightarrow \y^o$ as $l \rightarrow \infty$. As $F_s$ is a bijective open map, $F_s(V_2^c)$ is closed and hence, $\y^o \in F_s(V_2^c)$. Now, take any $\z \in H_s$. Taking a sequence $\z_{n_{k_l}} \in H_{s + t_{n_{k_l}}}$ such that $\z_{n_{k_l}} \rightarrow \z$, we observe
$$
\dist(\x^o, \z) = \lim_{l \rightarrow \infty} \dist(\x_{n_{k_l}},\z_{n_{k_l}}) = \lim_{l \rightarrow \infty}\dist(\y_{n_{k_l}},\z_{n_{k_l}}) = \dist(\y^o, \z).
$$
By Lemma \ref{lem:Isom1}, $\y^o=F_s(\x^o)$, which gives a contradiction as $\x^o \in \overline{U_2} \subseteq V_2$ but $\y^o \in F_s(V_2^c)$. The claim follows from this.

The above two observations along with Lemma \ref{lem:varadhan} applied to \eqref{eq:large1} and \eqref{align:large2} yield
\begin{align*}
\lim_{t \downarrow 0} \ 2t \log\left[\mu(X_{s+t} \in U_2, X_s \in V_1, \tau> s+t)\right] \quad&=\quad -\inf_{\z \in V_1, \mathbf{w} \in U_2} \dist^2(\z, \mathbf{w}),\\
\limsup_{t \downarrow 0} \ 2t \log\left[\mu(Y_{s+t} \in F_{s+t}(U_2), X_s \in F_s(V_1), \tau> s+t)\right] \quad&\le \quad -\inf_{\z \in V_1, \mathbf{w} \in V_2} \dist^2(F_s(\z), F_s(\mathbf{w})).
\end{align*}
Since the left hand side of \eqref{eq:large1} is the same as that of \eqref{align:large2}, we take $\delta \downarrow 0$ above to get
$$
\dist(\x,\y) \ge \dist(F_s(\x), F_s(\y)).
$$
As $F_s$ is involutive, applying a symmetric argument with $\x, \y$ replaced by $F_s(\x), F_s(\y)$ yield the opposite inequality. Hence, $\dist(\x,\y) = \dist(F_s(\x), F_s(\y))$ for all $\x, \y \in M$. Thus, $F_s$ is an isometry for every $s>0$.

Finally, consider the case $s=0$. Again, for $\x,\y \in H_0$ or $\x \in H^-_0, \y \in H^+_0$, $\dist(\x,\y)=\dist(F_0(\x), F_0(\y))$ follows from the definition of $F_0$. For $\x, \y \in H^-_0$, by the same procedure used to define $F_0$ earlier in the proof, we obtain sequences $t_n \downarrow 0$ and $\x_n \in H^-_{t_n}$ and $\y_n \in H^+_{t_n}$ such that $\x_n \rightarrow \x$, $\y_n \rightarrow \y$, $F_{t_n}(\x_n) \rightarrow F_0(\x)$ and $F_{t_n}(\y_n) \rightarrow F_0(\y)$. Thus,
$$
\dist(F_0(\x), F_0(\y))= \lim_{n \rightarrow \infty}\dist(F_{t_n}(\x_n), F_{t_n}(\y_n))= \lim_{n \rightarrow \infty}\dist(\x_n, \y_n) = \dist(\x,\y),
$$
which proves that $F_0$ is an isometry.

Now, $F_s(H^-_s)=H^+_s$ follows from Lemma \ref{lem:Isom1}. This completes the proof of the lemma.
\end{proof}

Following \citet[Chapter 10, Proposition 24]{Petersen-2006}, as $H_s$ is the fixed point set of an isometry therefore each connected component of $H_s$ is a totally geodesic submanifold (in particular, a smooth submanifold). 
Furthermore, as $H_s$ partitions $M$ into two disjoint open subsets, it can be verified (for example by referring to normal coordinates based around a point in $H_s$) that $H_s$ must be of codimension $1$. 
Furthermore, this discussion also implies that for any $\x, \y \in M$ there is \emph{at most one isometry} whose set of fixed points is the set $H(\x,\y)$. 
We will refer to this isometry, if it exists, as $f_{\x,\y}$. 
In fact 
Lemmas \ref{lem:Isom1} and \ref{lem:globiso}
together imply that for any \(s\geq0\) there does indeed exist such a $f_{\x,\y}$ for each $(\x,\y) \in \mathcal{M}(\mu_s)$, given by 
\[
f_{\x,\y}\quad=\quad F_s\,.
\]
To get an intuitive picture of how $F_s$ acts locally around a point $\x^* \in H_s$ (hence, fixed by \(F_s\)), recall that 
\[
\d F_s: T_{\x^*}M \rightarrow T_{\x^*}M
\]
is a linear isometry. 
We can form an orthonormal basis $e_1, \dots e_d$ of $T_{\x^*}M$ such that $e_1,\dots,e_{d-1}$ form a basis of the tangent space $T_{\x^*}H_s$ viewed as a subspace of $T_{\x^*}M$. 
Because \(H_s\) is totally geodesic, these vectors correspond to geodesics through $\x^*$ that stay in $H_s$.
As $H_s$ is the fixed point set of $F_s$, the basis vectors $e_1,\dots,e_{d-1}$ must be fixed by $\d F_s$, while $e_d$ is mapped by \(\d F_s\) to $-e_d$. 
Thus, locally, one geodesic passing through $\x^*$ is inverted by \(F_s\), while geodesics starting in directions orthogonal to the inverted geodesic are fixed by \(F_s\).

\subsection{Structure of the manifold $M$}
In this section, we will use the isometries $f_{\x,\y}$ constructed above for every pair of points $\x \in \ball(\x_0,r)$ and $\y \in \ball(\y_0, r)$ 
to show that the underlying complete Riemannian manifold $M$ is \emph{homogeneous} 
(i.e.~the isometry group acts transitively) and \emph{isotropic} about a chosen point $\x^*$ (i.e.~there are $\tfrac{d(d-1)}{2}$ independent rotations about $\x^*$). 
This will imply that $M$ is a \emph{maximally symmetric space}, i.e.~the isometry group $\mathcal{G}$ of $M$ has the maximal dimension possible (namely, $\tfrac{d(d+1)}{2}$) for any $d$-dimensional manifold. 
It is an almost immediate consequence that the space $M$ can be classified (up to scaling) as one of the three model space forms of constant curvatures respectively \(-1\), \(0\), and \(+1\).
\begin{lem}\label{lem:homogeneous}
Suppose that the standing assumptions of diffusion-geodesic completeness and stochastic completeness both hold.
Under {\LPC}, $(M,g)$ is a homogeneous space.
\end{lem}
\begin{proof}
We want to show that $\mathcal{G}$ acts transitively on $M$. 
Together with {\LPC}, the work of the previous subsection shows that for each $\x \in \ball(\x_0, r)$ and $\y \in \ball(\y_0, r)$, there exists an involutive isometry $f_{\x, \y}$. 
This implies that, for any $\x \in \ball(\x_0, r)$, there is an isometry $G_{\x_0, \x}= f_{\y_0,\x} \circ f_{\x_0,\y_0}$ which takes $\x_0$ to $\x$. 
Consider the set of isometries 
\[
\mathcal{I}\quad=\quad\{G_{\x_0, \x}: \x \in \ball(\x_0, r)\}\,.
\]
Let $\Mirrors$ be the closure of the subgroup generated by $\mathcal{I}$, so $\Mirrors$ is a closed subgroup of $\mathcal{G}$. 
Denote by $\mathcal{O}(\x_0)$, the \emph{orbit} or \emph{set of equivalent points} of $\x_0$ under $\Mirrors$. 
By construction, $\ball(\x_0,r) \subseteq \mathcal{O}(\x_0)$. 
In order to prove that $M$ is homogeneous, we need to prove $\mathcal{O}(\x_0)=M$, which we will show by proving that $\mathcal{O}(\x_0)$ is both open and closed in $M$. 
Let $\z$ be a limit point of $\mathcal{O}(\x_0)$. 
Then, there is a sequence of isometries $G_n \in \Mirrors$ such that $G_n(\x_0) \rightarrow \z$. 
By \citet[p.~7]{MyersSteenrod-1939}, there exists an isometry $G \in \Mirrors$ and a subsequence $G_{n_k} \in \Mirrors$ such that $G_{n_k} \rightarrow G$ in the topology of isometries (i.e.~$G_{n_k}(\x) \rightarrow G(\x)$ for all $\x \in M$), and consequently, $G(\x_0)=\z$. 
This shows that $\mathcal{O}(\x_0)$ is closed. 
On the other hand, if $\y \in \mathcal{O}(\x_0)$, then there is an isometry $G \in \Mirrors$ such that $\y = G(\x_0)$. 
Therefore, $\ball(\y,r)=G\left(\ball(\x_0,r)\right) \subseteq \mathcal{O}(\x_0)$ (as $\ball(\x_0,r) \subseteq \mathcal{O}(\x_0)$) implying $\mathcal{O}(\x_0)$ is open. 
Thus, $\mathcal{O}(\x_0)= M$, proving the lemma.
\end{proof}

In the following lemma, we will write $\x^*$ for the midpoint of a minimal geodesic $\gamma_{\x_0,\y_0}$ connecting $\x_0$ and $\y_0$. 
If two vectors $u,v$ belong to the same tangent space then we denote the angle between them by $\angle(u,v)$.
\begin{lem}\label{lem:isotropic}
Suppose that the standing assumptions of diffusion-geodesic completeness and stochastic completeness both hold.
Under {\LPC}, $M$ is \emph{isotropic} at $\x^*$.
\end{lem}
\begin{proof}
Let $\gamma(v)$ denote the geodesic issuing from $\x^*$ in direction $v$. Suppose $\gamma(v_0)=\gamma_{\x_0,\y_0}$, thus defining a unit vector \(v_0\). 
The proof proceeds in three steps as follows.

\textbf{Step 1. }First, we want to show that there is $\eps>0$ such that, for any $v \in T_{\x^*}M$ with $\angle(v,v_0) < \eps$, there is an isometry $g_v$ leaving $\x^*$ fixed and $\d g_v(v_0)=v$.

By continuity of geodesics in the starting direction, we can choose \(\eps>0\) sufficiently small so that $\gamma(v')$ intersects $\ball(\x_0, r)$ and $\gamma(-v')$ intersects $\ball(\y_0, r)$ whenever $\angle(v',v_0) < \eps$. 
By \citet[Proposition 20, p.~141]{Petersen-2006}, with a possibly smaller choice of $\eps>0$, we can take $\x_{v'} \in \gamma(v') \cap \ball(\x_0,r)$ and $\y_{v'} \in \gamma(-v') \cap \ball(\y_0,r)$ such that 
$\gamma(v')$ realises the distance $\dist(\x^*,\x_{v'})$ and $\gamma(-v')$ realises the distance $\dist(\x^*,\y_{v'})$. 
Furthermore, by continuity of the metric, when $\eps>0$ is small enough, we can take such $\x_{v'}$, $\y_{v'}$ satisfying $\dist(\x_{v'},\x^*)=\dist(\y_{v'},\x^*)$ whenever $\angle(v',v_0) < \eps$. 
Thus, from the developments of the previous subsection, there is an involutive isometry $f_{\x_{v'},\y_{v'}}$ which fixes $\x^*$, 
inverts the geodesic passing through $\x^*$ in direction $v'$, and fixes all the geodesics which pass through $\x^*$ in directions orthogonal to $v'$.

Now, take any unit vector $v \in T_{\x^*}M$ with $\angle(v,v_0)<2\eps$. Let ${v'=\tfrac{v+v_0}{|v+v_0|}}$. 
By the properties of rhombuses, $\angle(v',v_0)=\tfrac{1}{2}\angle(v,v_0) < \eps$, and thus $f_{\x_{v'},\y_{v'}}$ exists as specified in the preceding paragraph. 
Now, consider the isometry $g_v= f_{\x_{v'},\y_{v'}}\circ f_{\x_0,\y_0}$. 
Note that $g_v$ fixes $\x^*$ and a straightforward calculation reveals $\d g_v(v_0)=v$. 
This $g_v$ is our required isometry.

\textbf{Step 2. }Take any unit vector $w \in T_{\x^*}M$ such that $w$ and $v_0$ are linearly independent. 
Let $\Pi$ be the two-dimensional subspace of $T_{\x^*}M$ generated by $v_0$ and $w$ and denote by $\mathbb{S}(v_0,w)$ the circle in $T_{\x^*}M$ centred at the origin of $T_{\x^*}M$ and running through $v_0$ and $w$. 
Let  $U$ be a normal neighbourhood around $\x^*$. 
Let $S_{\Pi}=\exp_{\x^*}(\Pi) \cap U$ denote the two-dimensional fragment of $M$ corresponding to $\Pi$ and lying in \(U\). 

Denote by $\Mirrors(v_0,w)$ the closed subgroup of isometries generated by $\{g_v: v \in \mathbb{S}(v_0, w), \\ \angle(v,v_0) < \eps\},$ where $g_v$ are the isometries constructed in Step 1. 
Note that the set $\{g_v: v \in \mathbb{S}(v_0, w), \\ \angle(v,v_0) < \eps\}$, and hence $\Mirrors(v_0,w)$, fixes $\x^*$ and keeps vectors orthogonal to $\{v_0, w\}$ fixed. 
Let 
\[
O(v_0)\quad=\quad\{dg(v_0): g \in \Mirrors(v_0,w)\}\,.
\]
We want to show that $O(v_0)=\mathbb{S}(v_0, w)$. 

Note that, if $v_n= \d g_n(v_0)$ such that $v_n \rightarrow v$, then, by the fact that $g_n(\x^*)=\x^*$ for all $n$, we can choose a subsequence $g_{n_k}$ and a $g \in \Mirrors(v_0,w)$ such that $g_{n_k} \rightarrow g$ in the topology of isometries \citep[p.~7]{MyersSteenrod-1939}. 
Thus, by \citet[Lemma 4]{MyersSteenrod-1939}, $dg_{n_k}(v_0) \rightarrow dg(v_0)$ implying $O(v_0)$ is closed. 
Furthermore, if $g \in \Mirrors(v_0,w)$ then $dg$ is a linear isometry on $T_{\x^*}M$.
So the same argument as in the previous lemma shows that $O(v_0)$ is open. 
Thus, $O(v_0)=\mathbb{S}(v_0, w)$. 

Thus, in particular, the subgroup of isometries $\mathcal{G}_{\x^*}$ which fix $\x^*$ (the \emph{isotropy group} at $\x^*$) generates all the rotations of $T_{\x^*}M$ based at $\x^*$ in 2-planes containing $v_0$.
We describe the isometries in $\Mirrors(v_0,w)$ as \emph{rotations in $\mathbb{S}(v_0, w)$}.

\textbf{Step 3. }
We will now show that, given two ordered orthonormal frames based at $T_{\x^*}M$, there is a sequence of isometries in $\mathcal{G}_{\x^*}$ that take one to the other. 
In particular this implies that $M$ is isotropic at $\x^*$. 
Let $(e_1, \dots, e_d)$ and $(e'_1, \dots, e'_d)$ be ordered orthonormal frames in $T_{\x^*}M$. 
We can apply rotations in $\mathbb{S}(v_0, e_1)$ (respectively $\mathbb{S}(v_0, e'_d)$) to align $e_1$ with $v_0$ (respectively $e'_d$ with $v_0$). 
Thus, without loss of generality, we consider frames of the form $(v_0, e_2, \dots, e_d)$ and $(e'_1, \dots, e'_{d-1},v_0)$.

Now, apply a rotation in $\mathbb{S}(v_0,e'_1)$ to transform $(v_0, e_2, \dots, e_d)$ to $(e'_1,e^{(1)}_2 \dots, e^{(1)}_d)$ for some unit vectors $e^{(1)}_2, \dots, e^{(1)}_d$ in $T_{\x^*}M$. 
If $v_0$ and $e^{(1)}_2$ are linearly independent, then apply a rotation in $\mathbb{S}(v_0,e^{(1)}_2)$, to bring $(e'_1,e^{(1)}_2, \dots, e^{(1)}_d)$ to $(e'_1,v_0, e^{(2)}_3, \dots, e^{(2)}_d)$. 
If $e^{(1)}_2=-v_0$, then achieve the same result using the reflection $f_{\x_0,\y_0}$. 
Note that these operations both keep $e'_1$ fixed as it is orthogonal to $\{v_0, e^{(1)}_2\}$.

The same procedure is applied inductively to $(e'_1,v_0, e^{(2)}_3, \dots, e^{(2)}_d)$ to obtain $(e'_1, e'_2, v_0, e^{(4)}_4, \dots, e^{(4)}_d)$ 
(note that these operations leave $e'_1$ fixed), and so on. Finally we obtain $(e'_1, \dots, e'_{d-1}, v_0)$, which proves the lemma.
\end{proof}
The above two lemmas imply the following rigidity theorem which completely classifies the space $M$.
\begin{thm}\label{thm:Mclass}
Suppose that the complete, connected Riemannian manifold \(M\) supports Brownian motion with drift for which there is a Markovian maximal coupling and moreover {\LPC} holds.
Then $M$ has constant sectional curvature. 
Moreover $M$ must be simply connected and therefore (up to scaling) $M$ must be one of the three model spaces $\Reals^d$, $\mathbb{S}^d$ and $\mathbb{H}^d$.
\end{thm}
\begin{proof}
By Lemmas \ref{lem:homogeneous} and \ref{lem:isotropic}, we see that $M$ is a \emph{maximally symmetric space}, i.e., the dimension of $\operatorname{Iso}(M)$ is $\frac{d(d+1)}{2}$ \citep[p.~195]{Sharan-2009}.
In particular, this implies that \(M\) has constant sectional curvature \citep[p.~190]{Petersen-2006}.
For the second part of the corollary, the argument of \citet[p.~190]{Petersen-2006} shows that a complete, connected maximally symmetric Riemannian manifold must be one of the three model spaces above, or $\mathbb{RP}^d$. 
But, as observed in \citet[Example 6.4]{Kuwada-2009}, there is no involutive isometry of $\mathbb{RP}^d$ of the form described in Lemma \ref{lem:globiso}. This proves the theorem.
\end{proof}
\begin{rem}\label{rem:mirrorshape}
For the three model spaces described above, for every $\x, \y \in M$, the reflection isometry $f_{\x,\y}$, and hence the set of its fixed points $H(\x,\y)$, can be explicitly described (see, for example, \citep[Example 4.6]{Kuwada-2007}). It follows from this explicit description that the submanifold $H(\x,\y)$ with the induced metric is again one of the three model spaces with the same curvature as the ambient manifold $M$ and having codimension one.
\end{rem}
\subsection{Evolution of the mirror isometries}\label{sec:evolution-mirror-isometries}
Having classified the space \(M\), we must now classify the set of drift vectorfields $\mathbf{b}$ which permit \textbf{MMC} with {\LPC}. 
This necessitates analysis of the evolution of the isometries $F_s$ as $s$ varies. 
As noted above, \citet{MyersSteenrod-1939} proved that the set of isometries $\mathcal{G}$ has the structure of a Lie group. 
The first objective is to prove that the curve of isometries $(F_s:s \ge 0)$ is a $C^1$ curve in this Lie group.
\begin{lem}\label{lem:Isocont}
Suppose that the standing assumptions of diffusion-geodesic completeness and stochastic completeness both hold.
The curve $s \mapsto F_s$ is a $C^1$ curve in the Lie group $\mathcal{G}$.
\end{lem}
\begin{proof}
Recall that any point in $M$ has a neighbourhood, called a $\sigma$-neighbourhood, such that any point in this neighbourhood is in a normal coordinate ball of any other point in the same neighbourhood. 
We study continuity and continuous differentiability of $(F_s:s \ge 0)$ at \(s=t\).
As we are investigating a local property, we work in two separate sets of normal coordinates; 
one set describing a $\sigma$-neighbourhood $U$ around $\x$ and the other set describing another $\sigma$-neighbourhood $V$ around $F_t(\x)$ such that $F_t(\overline{U}) \subset V$.

The first step is to prove that $s \mapsto F_s$ is continuous in $\mathcal{G}$ at \(s=t<\tau\). 
To show this, it suffices to show that any set of $d+1$ points $\x_i \in M$, all of which lie in a $\sigma$-neighbourhood and are linearly independent 
(i.e.~do not belong in the same $(d-1)$-dimensional geodesic hypersurface), 
produces continuous curves $s \mapsto F_s(\x_i)$ in $M$ \citep{MyersSteenrod-1939}. We note here that we can obtain such a set of $d+1$ points in any dense subset of any open set in $M$.
To show the continuity of these curves, we will use the continuity of the diffusion paths and the fact that, by Corollary \ref{cor:alltimeman}, $Y_s=F_s(X_s)$ when $s < \tau$. 

Define the new distance $$\overline{\dist}(\x,\y) \quad=\quad \frac{\dist(\x,\y)}{1+\dist(\x,\y)}$$ for $\x, \y \in M$. Note that $\overline{\dist}(\cdot,\cdot)$ is bounded and it defines a distance that produces the same topology on $M$ as $\dist(\cdot,\cdot)$ does. Now, take any sequence $\{s_n\}_{n \ge 1}$ with $\lim_{n \rightarrow \infty} s_n = t$. Then,
\begin{align*}
\limsup_{n \rightarrow \infty} & \Expect{\overline{\dist}(F_{s_n}(X_t), F_t(X_t)) \mathbb{I}(\tau>t)}\\
& \le \limsup_{n \rightarrow \infty} \Expect{\overline{\dist}(F_{s_n}(X_t), F_{s_n}(X_{s_n})) \mathbb{I}(\tau>t)} + \limsup_{n \rightarrow \infty} \Expect{\overline{\dist}(F_{s_n}(X_{s_n}), F_t(X_t)) \mathbb{I}(\tau>t)}\\
&= \limsup_{n \rightarrow \infty} \Expect{\overline{\dist}(X_t, X_{s_n})\mathbb{I}(\tau>t)} + \limsup_{n \rightarrow \infty} \Expect{\overline{\dist}(Y_{s_n}, Y_t)\mathbb{I}(\tau>t)}\\
&=0.
\end{align*}
Here, the equality in the second step follows from the fact that $F_{s_n}$ is an isometry, $Y_s=F_s(X_s)$ when $s<\tau$, and the dominated convergence theorem. The last equality follows from the path continuity of $X$ and $Y$ and another application of the dominated convergence theorem. Thus, $\overline{\dist}(F_{s_n}(x), F_t(x))$ converges to zero in $L^1$ with respect to the law of $X_t$ restricted on $\{\tau >t\}$. Hence, we can extract a subsequence $n_k$ such that $F_{s_{n_k}}$ converges to $F_t$ almost everywhere with respect to the same measure. As the law of $X_t$ restricted on $\{\tau>t\}$ has full support on $H^-_t$, therefore the set of $\x \in H^-_t$ for which $F_{s_{n_k}}(\x) \rightarrow F_t(\x)$ is a dense subset of $H^-_t$. Hence, by the previous discussion, $F_{s_{n_k}} \rightarrow F_t$ in $\mathcal{G}$. As the limit does not depend on the chosen subsequence $n_k$, we conclude that $F_{s_n} \rightarrow F_t$ in $\mathcal{G}$, proving continuity of $s \mapsto F_s$.

It is necessary to address the question of right-continuity at $t=0$. 
Take $\x \in H^-_0$ and consider the case when $t_n \downarrow 0$. 
Take a sequence $\x_n \rightarrow \x$ such that $\x_n \in H^-_{t_n}$.
An argument following the treatment of the case $s=0$ in the proof of Lemma \ref{lem:globiso} shows that $F_{t_n}(\x_n) \rightarrow F_0(\x)$. 
As $F_{t_n}$ is an isometry for each $n$, we can deduce that $F_{t_n}(\x) \rightarrow F_0(\x)$, thus proving right-continuity.

The next step is to prove differentiability at $t>0$. With $\sigma$-neighbourhoods $U$, $V$ of \(\x\), \(F_t(\x)\) as described above, let $\tau_U=\inf\{s \ge t: X_s \notin U\}$. 
Because the coupling is Markovian, $\tau_U$ is a stopping time with respect to the filtration generated by the coupling process $(X,Y)$. 
Consider the stopped processes $X_s^U=X_{s \wedge \tau_U}$ and $Y_s^U=Y_{s \wedge \tau_U}$. 
In a slight abuse of notation, we use the same notation $X_s^U$ for the coordinate representation for this stopped process in $U$, and similarly for $Y_s^U$. 
Also we continue to write $F_s$ for the coordinate representation of $F_s: U \rightarrow V$. 
 
By Lemma 8 of \cite{MyersSteenrod-1939} it suffices to prove differentiability at $t$ of the continuous curve $s \mapsto F_s(\x)$ for $\x \in H^-_t$ such that $(\x,F_t(\x)) \in \mathcal{M}(\mu_t)$.
Take $U$, $V$ and normal coordinate systems for $\x$ and $F_t(\x)$ as above. 
Using these coordinates, we may write the stochastic differential equation for $X^U$ as
\[
\d X_s^{U,i}\quad=\quad
\mathbf{b}^i(X^U_s)\d s+ \sum_{j=1}^d\sigma^{i,j}(X^U_s)\d W^j_s
\]
for some Brownian motion $W$ in $U$. 
A similar expression holds for $Y^U$ with $\mathbf{b}^i_F$ and $\sigma^{i,j}_F$ representing the corresponding quantities. 
General properties of diffusions \citep[Chapter 11]{Nelson-1967} yield the following expressions in coordinate form:
\begin{eqnarray}\label{eqnarray:Nelson}
\mathbf{b}^i(\x)\quad&=\quad&\lim_{s \downarrow t}\;\Expect{\frac{X^{U,i}_s-x^i}{s-t}\ \Bigg| \  X^U_t=\x}\,,\nonumber\\
\sigma^{i,j}(\x)\quad&=\quad&\lim_{s \downarrow t}\;\Expect{\frac{(X^{U,i}_s-x^i)(X^{U,j}_s-x^j)}{s-t}\ \Bigg| \ X^U_t=\x}\,,\nonumber\\
\mathbf{b}^i_F(F_t(\x))\quad&=\quad&\lim_{s \downarrow t}\;\Expect{\frac{Y^{U,i}_s-F^i_t(\x)}{s-t}\ \Bigg| \ Y^U_t=F_t(\x)}\,.
\end{eqnarray}
By Corollary \ref{cor:alltimeman}, $Y_s=F_s(X_s)$ when $s< \tau$. Thus, we can write
\begin{equation}\label{equation:dereq}
\Expect{\frac{F^i_s(X_s^U)-F^i_t(\x)}{s-t} \ \Bigg| \ Y^U_t=F_t(\x)}\quad=\quad \Expect{\frac{F^i_s(X_s^U)-F^i_s(\x)}{s-t}\ \Bigg| \  X^U_t=\x}\ + \ \frac{F^i_s(\x)-F^i_t(\x)}{s-t}\,.
\end{equation}
The third expression in \eqref{eqnarray:Nelson} gives
\begin{equation*}
\lim_{s \downarrow t}\;\Expect{\frac{F^i_s(X_s^U)-F^i_t(\x)}{s-t} \ \Bigg| \ Y^U_t=F_t(\x)}\quad=\quad \mathbf{b}^i_F(F_t(\x))\,.
\end{equation*}
As $s \mapsto F_s$ is a continuous curve in $\mathcal{G}$, we may deduce by \citet[Lemma 7]{MyersSteenrod-1939} that the (space) derivatives of $F_s$ are continuous in $s$. By a Taylor expansion of $F_s$ in $U$ based at $\x$ and \eqref{eqnarray:Nelson},
\begin{align*}
\lim_{s \downarrow t}\; & \Expect{\frac{F^i_s(X_s^U)-F^i_s(\x)}{s-t}\ \Bigg| \  X^U_t=\x}\\
\quad&=\quad \lim_{s \downarrow t}\;\left( \sum_{j=1}^d\partial_jF^i_s(\x)\Expect{\frac{X^{U,j}_s-x^j}{s-t}\ \Bigg| \ X^U_t=\x}\right.\\
&\quad+\left.\frac{1}{2}\sum_{j=1}^d\sum_{k=1}^d\partial_{j,k}F^i_s(\x)\Expect{\frac{(X^{U,j}_s-x^j)(X^{U,k}_s-x^k)}{s-t}\ \Bigg| \  X^U_t=\x} + o(1)\right)\\
\quad&= \quad \sum_{j=1}^d\partial_jF^i_t(\x)\mathbf{b}^j(\x) + \frac{1}{2}\sum_{j=1}^d\sum_{k=1}^d\partial_{j,k}F^i_t(\x)\sigma^{j,k}(\x)\,.
\end{align*}
Thus, from \eqref{equation:dereq}, we deduce that the curve $s \mapsto F_s(\x)$ has a continuous right-derivative given by
\begin{equation}\label{eq:derivativelocal}
\lim_{s \downarrow t}\;\frac{F^i_s(\x)-F^i_t(\x)}{s-t}\quad=\quad
\mathbf{b}^i_F(F_t(\x))-\sum_{j=1}^d\partial_jF^i_t(\x)\mathbf{b}^j(\x)-\frac{1}{2}\sum_{j=1}^d\sum_{k=1}^d\partial_{j,k}F^i_t(\x)\sigma^{j,k}(\x)\,.
\end{equation}
This, together with \citet[Theorem 1.3]{Bruckner-1978}, implies uniformly continuous differentiability of $s \mapsto F_s(\x)$ at $t>0$. 
Note that the Mean Value Theorem and right-continuity of the right hand side of (\ref{eq:derivativelocal}) now gives us right-differentiability at $t=0$. This proves the lemma.
\end{proof}
\begin{cor}\label{cor:derivtime}
All the partial derivatives with respect to \(\x\) of $(s,\x) \mapsto F_s(\x)$ are continuously differentiable in $s$. 
Furthermore, $\left.\tfrac{\d}{\d s}\right|_{s=t}F_s(\x)$ is smooth in $\x$.
\end{cor}
\begin{proof}
Using the argument of \citet[Section 8]{MyersSteenrod-1939}, we can deduce the following representation in local coordinates $(x^i)$: 
\[
F_t(x^1,\dots,x^d)\quad=\quad\Psi(x^1,\dots, x^d, F_t(\x_0), \dots, F_t(\x_d))\,,
\]
where $\Psi$ is a smooth function and $\x_0,\dots, \x_d$ are fixed points in $M$. The corollary follows from this representation and the previous lemma.
\end{proof}
The derivative vectorfield $\kappa$ defined on $M$ by 
\[
\kappa(\x)\quad=\quad\left.\frac{\d}{\d s}\right|_{s=0}F_s(F_0(\x)) 
\]
possesses a special significance.
This is the \emph{Killing vectorfield} corresponding to the $C^1$ curve $s \mapsto G_s$ in $\mathcal{G}$ given by $G_s(\x)=F_s(F_0(\x))$ for $\x \in M$.
Vectorfields of this form correspond to the natural action of elements in the Lie algebra of $\mathcal{G}$ on the manifold $M$ 
(recall that $F_0 \circ F_0$ is the identity map, and the Lie algebra of $\mathcal{G}$ corresponds to the tangent space of $\mathcal{G}$ at the identity).
Killing vectorfields will play a crucial r\^ole in the following subsections.

\subsection{Structure of the coupling}
The processes $X$ and $Y$ can be constructed as projections $X_t=\pi U_t$ and $Y_t=\pi \widetilde{U}_t$, 
where $U$ and $\widetilde{U}$ are solutions to Stratonovich stochastic differential equations which are defined on the orthonormal frame bundle $\mathcal{O}(M)$ by
\begin{eqnarray}\label{eq:stochdev}
\d U_t\quad&=\quad&\sum_i H_i(U_t) \circ \d W^i_t + \mathbf{B}(U_t)\d t\,,\nonumber\\
\d\widetilde{U}_t\quad&=\quad&\sum_i H_i(\widetilde{U}_t) \circ \d\widetilde{W}^i_t + \mathbf{B}(\widetilde{U}_t)\d t\,,
\end{eqnarray}
for \(d\)-dimensional Euclidean Brownian motions $W$ and $\widetilde{W}$ and the vectorfield and the lifted drift vectorfield \(\mathbf{B}\) given by (\ref{eq:stratdrift}).

Any isometry $F$ on $M$ has a natural lift to a smooth mapping $\hat{F}:\mathcal{O}(M)\to\mathcal{O}(M)$, given by
\begin{equation}\label{eq:isolift}
\hat{F}(\pi u, ue_1, \dots, ue_d)\quad=\quad (F(\pi u), \d F(ue_1), \dots , \d F(ue_d))\,.
\end{equation}
The following lemma shows that \(\hat{F}\) respects the structure of horizontal vectorfields on \(\mathcal{O}(M)\).
\begin{lem}\label{lem:Stratlem}
Let $F$ be an isometry on $M$ and let $\hat{F}$ be the lift to $\mathcal{O}(M)$ as defined above. 
For $1 \le i \le d$ and $u \in \mathcal{O}(M)$, 
\[
\d\hat{F}(H_i(u))\quad=\quad H_i(\hat{F}(u))\,.
\]
\end{lem}
\begin{proof}
Let $\gamma$ be the unit speed geodesic in $M$ starting from $\pi u$ in direction $ue_i$, defined on some interval $[0,\eps]$ for some $\eps>0$.
For each $1 \le j \le d$, let $u^j_t$ denote the parallel transport of $ue_j$ along $\gamma$. 
Define the curve $\gamma^u$ in $\mathcal{O}(M)$ given by 
$$
\gamma^u(t)\quad=\quad
(\gamma_t, u^1_t,\dots,u^d_t)
$$
 for $t \in [0,\eps]$. 
As the covariant derivative commutes with the push-forward of vector fields by isometries \citep[Proposition 5.6]{Lee-1997}, for each $1 \le j \le d$, $\d F(u^j_t)$ provides a parallel transport of $\d F(ue_j)$ along $F\circ\gamma_t$. Hence, 
$$
\gamma^{\hat{F}(u)}(t)\quad=\quad\hat{F}\circ\gamma^u(t)=(F\circ\gamma_t, \d F(u^1_t),\dots, \d F(u^d_t))\,.
$$
 Now $(\gamma^u)'(0)=H_i(u)$.
 Thus
\begin{eqnarray*}
\d\hat{F}(H_i(u))\quad&=\quad&\d\hat{F}((\gamma^u)'(0))
\quad=\quad
(\hat{F}\circ\gamma^u)'(0)=(\gamma^{\hat{F}(u)})'(0)\\
\quad&=\quad& H_i(\hat{F}(u))\,,
\end{eqnarray*}
proving the lemma.
\end{proof}
The stochastic differential equation \eqref{eq:stochdev} for $U$ delivers a diffusion $V$ on $\mathcal{O}(M)$ given by 
\[
V_t\quad=\quad\hat{F_t}(U_t)\,,
\]
where $F_t$ is the time-varying deterministic involutive isometry constructed in previous subsections. 
Note that this automatically implies $Y_t=F_t(X_t)=\pi V_t$ on $t < \tau$. 
Thus, $V$ lifts $Y$ up to the orthonormal frame bundle $\mathcal{O}(M)$. 
We now derive the stochastic differential equation for $V$.

From \citet[Equation (2.3)]{Kendall-1987} it follows that
\begin{equation}\label{eq:Strat1}
\d V_t\quad=\quad \sum_i (\d\hat{F_t}(H_i(U_t))) \circ \d W^i_t + \d\hat{F_t}(\mathbf{B}(U_t))\d t + \hat{\chi}_t(U_t)\d t\,,
\end{equation}
where 
\[
\hat{\chi}_t(u)\quad=\quad\left.\frac{\d}{\d s}\right|_{s=t}\hat{F_s}(u)
\]
exists by Lemma \ref{lem:Isocont} and Corollary \ref{cor:derivtime}.

Lemma \ref{lem:Stratlem} implies that
\[
\d\hat{F_t}(H_i(U_t))\quad=\quad H_i(\hat{F_t}(U_t))\quad=\quad H_i(V_t)\,,
\]
 and
\[
\d\hat{F_t}(\mathbf{B}(U_t))\quad=\quad \sum_i b_i(U_t)\d\hat{F_t}(H_i(U_t))
\quad=\quad\sum_i  b_i(\hat{F_t}(V_t))H_i(V_t)
\]
where we have used Lemma \ref{lem:Stratlem} and the fact that $\hat{F_t}^2=\operatorname{Id}$ in the last step.

Thus, the stochastic differential equation for $V$ takes the form
\begin{equation}\label{eq:stochastic differential equationV}
\d V_t\quad=\quad \sum_i H_i(V_t) \circ \d W^i_t + \sum_i b_i(\hat{F_t}(V_t))H_i(V_t)\d t + \hat{\chi}_t(\hat{F_t}(V_t))\d t\,.
\end{equation}
Considering differentiation along the curve $\gamma^u$ introduced in the proof of Lemma \ref{lem:Stratlem}, it can be seen that 
$$
d\pi (H_i(u))=ue_i.
$$
 Also, as $F_t$ is an involutive isometry,
\begin{eqnarray*}
b_i(\hat{F_t}(V_t))\quad&=\quad& \left\langle \mathbf{b}(F_t(Y_t)),\d F_t(V_te_i) \right\rangle_{F_t(Y_t)}
\quad=\quad\langle \d F_t(\mathbf{b}(F_t(Y_t)), V_te_i \rangle_{Y_t}\\
\quad&=\quad&\langle F_{t*}\mathbf{b}(Y_t),V_te_i\rangle_{Y_t}\,,
\end{eqnarray*}
where $F_{t*}\mathbf{b}$ is the pushforward of the vectorfield $\mathbf{b}$ on $M$ by the isometry $F_t$.\\\\
Finally, writing 
$$
\chi_t(\x)\quad=\quad
\left.\frac{\d}{\d s}\right|_{s=t}F_s(\x)
$$
 for $\x \in M$, note that, for $u \in \mathcal{O}(M)$ and a smooth function $f: M \rightarrow \Reals$,
\[
\d\pi(\hat{\chi}_t(u))(f)\quad=\quad
\left.\frac{\d}{\d s}\right|_{s=t}(f \circ \pi \circ \hat{F_s})(u)
\quad=\quad \left.\frac{\d}{\d s}\right|_{s=t}f(F_s(\pi(u)))
\quad=\quad \chi_t(\pi u)(f)\,.
\]
Thus, writing
\begin{equation}\label{eq:killingt}
\kappa_t(\x)\quad=\quad \chi_t(F_t(\x))
\end{equation}
for $\x \in M$, we obtain
\[
\d\pi(\hat{\chi}_t(\hat{F_t}(u)))\quad=\quad\kappa_t(\pi u)\,.
\]
 Note that $\kappa_t$ is the Killing vectorfield corresponding to the $C^1$ curve of isometries $(F_s\circ F_t: s \ge t)$,
as introduced at the end of subsection \ref{sec:evolution-mirror-isometries}.

Using the above relations, we can project down the stochastic differential equation \eqref{eq:stochastic differential equationV} for $V$ onto $M$ as follows.
\begin{eqnarray*}
\d Y_t\quad&=\quad&\sum_i \d\pi(H_i(V_t)) \circ \d W^i_t + \sum_i b_i(\hat{F_t}(V_t))\d\pi(H_i(V_t))\d t + \d\pi(\hat{\chi}_t(\hat{F_t}(V_t)))\d t\\
\quad&=\quad&\sum_i V_te_i \circ \d W^i_t + \sum_i \langle F_{t*}\mathbf{b}(Y_t),V_te_i\rangle_{Y_t}V_te_i \d t + \kappa_t(Y_t)\d t\\
\quad&=\quad& \sum_i V_te_i \circ \d W^i_t + F_{t*}\mathbf{b}(Y_t)\d t + \kappa_t(Y_t)\d t\,.
\end{eqnarray*}
From the above expression, we see that the generator of $Y$ at $(t,\x)$ is 
$$
\generator\quad=\quad\frac{1}{2}\Delta_M + F_{t*}b(\x)+\kappa_t(\x)\,.
$$
 Comparing this with (\ref{eq:generatorman}), we deduce the following important relation:
\begin{thm}\label{thm:couplingfundthm}
Suppose that the standing assumptions of diffusion-geodesic completeness and stochastic completeness both hold.
For a Markovian maximal coupling $(X,Y)$ to exist from starting points $(\x_0,\y_0)$, the following relation must hold:
\begin{equation}\label{eq:couplingfundeqn}
\mathbf{b}(\x)\quad=\quad F_{t*}\mathbf{b}(\x)+\kappa_t(\x)
\end{equation}
for all $\x \in M$ and $t \ge 0$, where $(F_s:s \ge 0)$ is the $C^1$ curve of isometries introduced in Lemma \ref{lem:globiso}.
\end{thm}
\begin{rem}\label{rem:bzero}
If $\mathbf{b}=0$ in the above theorem, we get $\kappa_t(\x)=0$ for all $\x \in M$ and all $t \ge 0$. In particular, $\kappa_t(F_t(\x))=0$, which by \eqref{eq:killingt} gives
$$
\left.\frac{\d}{\d s}\right|_{s=t}F_s(\x)=0
$$
for all $\x \in M$ and all $t \ge 0$. Thus, $F_t \equiv F_0$ for all $t \ge 0$. As $H_t$ is precisely the set of fixed points of $F_t$, we deduce that the mirror $H_t$ does not depend on time $t$. This was also proved in \citet[Proposition 4.2]{Kuwada-2009}.
\end{rem}
\subsection{Classification of the drift}
  Finally it is possible to produce a complete characterization of the drift $\mathbf{b}$ under {\LPC}. 
Recall that $M$ can only be a scaled  version of one of the model spaces $\mathbb{S}^d$, $\mathbb{H}^d$ or $\Reals^d$ corresponding to the curvature $K$ being constant and equal to \(+1\), \(-1\), or \(0\).

For this section, special attention  is paid to the equation \eqref{eq:couplingfundeqn} at time $0$. 
When the context makes it plain there is no ambiguity, we will write $F$ for $F_0$ and $\kappa$ for $\kappa_0$.

Let $\nabla$ represent the covariant derivative with respect to the Riemannian connection compatible with the metric $g$. 
We will need the following useful fact about Killing vectorfields \citep[Prop.~27]{Petersen-2006}.
\begin{lem}\label{lem:Killing}
If $\kappa$ is a Killing vectorfield, then for any $\x \in M$ and any $u\in T_{\x}M$,
\begin{equation}
\langle\nabla_u\kappa(\x),u\rangle \quad=\quad 0
\end{equation}
\end{lem}
Isometries take geodesics to geodesics, so any Killing vectorfield is a \emph{Jacobi field}, i.e.~the variation field of a variation through geodesics. 
Thus, Killing vectorfields satisfy the Jacobi equation, as given by the following lemma \citep[Theorem 10.2]{Lee-1997}.
\begin{lem}\label{lem:jacobi}
Let $\kappa$ be a Killing vectorfield. Then $\kappa$ satisfies the Jacobi equation along any (unit speed) geodesic $\gamma$:
\begin{equation}\label{eq:jacobi}
\nabla_{\dot{\gamma}}\nabla_{\dot{\gamma}}\kappa + R(\kappa, \dot{\gamma})\dot{\gamma}\quad=\quad 0\,.
\end{equation}
\end{lem}
Because of Theorem \ref{thm:Mclass}, we can confine attention to the case when $M$ is of constant curvature $K$,
in which case there is a simple representation for the curvature tensor $R$ \citep[Lemma 8.10]{Lee-1997}:
\begin{equation}\label{eq:constcurvendo}
R(X,Y)Z\quad=\quad
K(\langle Y, Z \rangle X - \langle X, Z \rangle Y)\,.
\end{equation}

We now define the \emph{symmetric \(2\)-form} associated with the drift vectorfield \(\textbf{b}\): for $u,v\in T_{\x}M$,
\begin{equation}\label{eq:symmtwoform}
S_{\x}(u,v)\quad=\quad
\frac{1}{2}\left(\langle\nabla_u\mathbf{b},v\rangle + \langle \nabla_v\mathbf{b},u\rangle\right)\,.
\end{equation}
The following lemma describes this symmetric \(2\)-form \(S_{\x}\) under {\LPC}.
\begin{lem}\label{lem:symmclass}
Suppose that the standing assumptions of diffusion-geodesic completeness and stochastic completeness both hold.
Under {\LPC}, there is a scalar $\lambda$ such that, for all $\x \in M$ and all $u,v \in T_{\x}M$,
$$
S_{\x}(u,v)\quad=\quad
\lambda \langle u,v \rangle\,.
$$
\end{lem}
\begin{proof}
Recall that $\x^*$ is the midpoint of a minimal geodesic connecting $\x_0$ and $\y_0$. 
Let $\{e_1, \dots, e_d\}$ denote the canonical orthonormal frame of $T_{\x^*}M$.
From previous discussions, $F$ `inverts' one geodesic through $\x^*$ (the minimal geodesic joining $\x_0$ and $\y_0$) and keeps all geodesics orthogonal to this one fixed. Let $\mathbf{n} \in T_{\x^*}M$ denote the direction of the inverted geodesic.

Now, consider any isometry $G$ that satisfies
\begin{equation}\label{eq:isodrift}
\mathbf{b}(\x)\quad=\quad
G_*\mathbf{b}(\x) + \kappa(\x)
\end{equation}
for some Killing vectorfield $\kappa$, for all $\x \in M$. Then, it follows that for any $\x \in M$ and $u,v \in T_{\x}M$,
\begin{eqnarray*}
\langle \nabla_u\mathbf{b}(\x),v\rangle \quad&=\quad&\langle \nabla_u(G_*\mathbf{b})(\x),v\rangle + \langle\nabla_u\kappa(\x),v\rangle\\
\quad&=\quad&\langle \nabla_{\d G^{-1}(u)}\mathbf{b}(G^{-1}(\x)),\d G^{-1}(v)\rangle + \langle\nabla_u\kappa(\x),v\rangle
\end{eqnarray*}
which, along with Lemma \ref{lem:Killing}, yields
\begin{equation}\label{eq:Shomo}
S_{\x}(u,v)\quad=\quad
S_{G^{-1}(\x)}(\d G^{-1}(u),\d G^{-1}(v))\,.
\end{equation}
In particular, equation (\ref{eq:couplingfundeqn}) at time $t=0$ gives
\begin{equation}\label{eq:ShomoF}
S_{\x^*}(u,v)\quad=\quad
S_{\x^*}(\d F(u),\d F(v))\,.
\end{equation}
where (\ref{eq:ShomoF}) follows from (\ref{eq:Shomo}) by noting that $F$ fixes $\x^*$ and $F^{-1}=F$. Let $S(\x^*)$ denote the matrix
$$
(S(\x^*))_{ij}=S_{\x^*}(e_i,e_j).
$$
Using the description above of \(F\) as 'inverting' the geodesic with tangent vector \(\mathbf{n}\) at \(\x^*\), and leaving orthogonal geodesics at \(\x^*\) fixed,
\eqref{eq:ShomoF} 
yields
\begin{equation}\label{eq:symmmat}
S(\x^*)\quad=\quad
({\Identity}-2\mathbf{n}\mathbf{n}^\top)S(\x^*)({\Identity}-2\mathbf{n}\mathbf{n}^\top)\,.
\end{equation}
By {\LPC}, we can choose $d$ pairs of starting points $\{(\x_i,\y_i): \x_i \in \ball(\x_0,r), \  \y_i \in \ball(\y_0,r), \ 1 \le i \le d\}$ 
such that the directions of the inverted geodesics $\mathbf{n}_i$ (for $1 \le i \le d$) based at \(\x^*\) form $d$ linearly independent vectors in \(T_{\x^*}M\) and $\mathbf{n}_i$ is not orthogonal to $\mathbf{n}_j$ for any $i \neq j$. 
Now, noting from equation \eqref{eq:symmmat} that $\mathbf{n}_i$ are eigenvectors of $S(\x^*)$, 
we find
\begin{equation}\label{eq:Sid}
S(\x^*)\quad=\quad
\lambda(\x^*)\Identity
\end{equation}
for some scalar $\lambda(\x^*)$. In coordinate-free terms, this is the assertion of the lemma at point $\x^*$.\\\\
Now, we want to show that the assertion of the lemma holds at any $\x \in M$. Denote
$$
\mathcal{Z}=\{G \in \mathcal{G}: G \text{ satisfies (\ref{eq:isodrift}) for some Killing vectorfield } \kappa \text{ and all } \x \in M\}.
$$
Recall that (\ref{eq:Shomo}) holds for all $G \in \mathcal{Z}$. Thus, by (\ref{eq:Sid}), we get 
$$
S_{G^{-1}(\x^*)}(u,v)=\lambda(\x^*) \langle u,v \rangle
$$
 for all $u,v \in T_{G^{-1}(\x^*)}M$.\\\\
By continuity of the map 
$$
G \mapsto S_{G^{-1}(\x^*)}(\d G^{-1}(u),\d G^{-1}(v))
$$
 in the topology of isometries \citep[Lemma 4]{MyersSteenrod-1939}, \eqref{eq:Shomo} holds for all $G \in \overline{\mathcal{Z}}$, where $\overline{\mathcal{Z}}$ denotes the closed subgroup generated by $\mathcal{Z}$.\\\\ 
Now, from the developments in subsection \ref{sec:interfaceman}, observe that, under {\LPC}, for any $\x \in \ball(\x_0,r)$ and $\y \in \ball(\y_0,r)$, there exists a unique involutive isometry $f_{\x,\y}$ whose fixed point set is exactly the set $H(\x,\y)$. These isometries satisfy (\ref{eq:isodrift}) as this equation corresponds to (\ref{eq:couplingfundeqn}) at time $t=0$ when the starting points of $X$ and $Y$ are taken to be $\x$ and $\y$ respectively. Furthermore, exactly along the lines of the proof of Lemma \ref{lem:homogeneous}, we see that the orbit of $\x^*$ under the closed subgroup of isometries generated by $\{f_{\x,y} : \x \in \ball(\x_0,r), \y \in \ball(\y_0,r)\}$ is the whole of $M$. In particular, the orbit of $\x^*$ under $\overline{\mathcal{Z}}$ is $M$. Thus, for all $\x \in M$, 
$$
S_{\x}(u,v)\quad=\quad
\lambda(\x^*) \langle u,v \rangle
$$
 for all $u,v \in T_{\x}M$, proving the lemma.
\end{proof}
Now we describe the drift vectorfield along geodesics issuing from $\x^*$, the midpoint of a minimal geodesic joining $\x_0$ and $\y_0$. In the following, we will denote the canonical orthonormal basis of $T_{\x^*}M$ by $\{e_1,\dots,e_d\}$. Also, for any vector $u \in T_{\x^*}M$ and any $d \times d$ matrix $T$, $Tu$ will denote the vector obtained by matrix multiplication when we identify $T_{\x^*}M$ with $\Reals^d$.
\begin{lem}\label{lem:driftclass}
Suppose that the standing assumptions of diffusion-geodesic completeness and stochastic completeness both hold.
If the drift vectorfield $\mathbf{b}$ permits \textbf{MMC} with {\LPC}, then it must satisfy the following. Let $\x^* \in M$ be the midpoint of a minimal geodesic connecting $\x_0$ and $\y_0$ and $u,v \in T_{\x^*}M$ be unit vectors with $u \perp v$. Let $\gamma$ represent the geodesic issuing from $\x^*$ in direction $u$ and let $V_t$ represent the parallel transport of $v$ along $\gamma$. Then the following holds.
\begin{equation}\label{eq:symm}
\langle \mathbf{b}(\gamma(t)), \dot{\gamma}_t\rangle\quad=\quad\lambda t + \langle \mathbf{b}(\x^*),u\rangle  
\end{equation}
where $\lambda$ is as in Lemma \ref{lem:symmclass}, and 
\begin{eqnarray}\label{eqnarray:antsymm}
\langle \mathbf{b}(\gamma(t)), V_t\rangle\quad=\quad
\left\{
\begin{array}{lll}
\langle \mathbf{b}(\x^*),v \rangle \cos \sqrt{K}t + \displaystyle{\langle Tu,v\rangle \frac{\sin \sqrt{K}t}{\sqrt{K}}} & \mbox{if } K>0\,,\\
\langle \mathbf{b}(\x^*),v \rangle + \langle Tu,v\rangle t & \mbox{if } K=0\,,\\
\langle \mathbf{b}(\x^*),v \rangle \cosh \sqrt{-K}t + \displaystyle{\langle Tu,v\rangle \frac{\sinh \sqrt{-K}t}{\sqrt{-K}}} & \mbox{if } K<0\,.
\end{array}
\right.
\end{eqnarray}
where the matrix $T$ given by $T_{ij}=\langle\nabla_{e_i}\mathbf{b}(\x^*),e_j\rangle - \lambda\langle e_i, e_j\rangle$ is a skew-symmetric matrix.
\end{lem}
\begin{proof}
To see (\ref{eq:symm}), note that 
\[
\frac{\d}{\d t}\langle \mathbf{b}(\gamma(t)), \dot{\gamma}_t\rangle \quad=\quad
\langle \nabla_{\dot{\gamma}_t}\mathbf{b}(\gamma(t)), \dot{\gamma}_t\rangle
\quad=\quad S( \dot{\gamma}_t, \dot{\gamma}_t)\quad=\quad\lambda\,.
\]
Take any $\x \in \ball(\x_0,r)$ and $\y \in \ball(\y_0,r)$ such that $\x^* \in H(\x,\y)$. 
Since $H(\x,\y)$ is the fixed point set of the isometry $f_{\x,\y}$, it is therefore a totally geodesic submanifold of $M$.
Let $\kappa$ denote the Killing vectorfield for which (\ref{eq:couplingfundeqn}) holds at time $t=0$ with $F_0=f_{\x,\y}$. 
Take any unit speed geodesic $\gamma$ passing through $\x^*$ and lying in \(H(\x,\y)\). 
(Note that, if a geodesic lies in $H(\x,\y)$ for a short time, it should lie in $H(\x,\y)$ for all time. See, for example, the proof of Proposition 24 of \citealp[p.~145]{Petersen-2006}.) 

Let $(n_t : t \ge 0)$ be the parallel transport of the vector normal to the hypersurface $H(\x,\y)$ at $\x^*$ along the geodesic $\gamma$.
Note that, as $H(\x,\y)$ is totally geodesic, the second fundamental form 
vanishes identically on $H(\x,\y)$ \citep[Exercise 8.4]{Lee-1997}. 
This fact implies that parallel transportation of a vector $v \in T_{\x^*}H(\x,\y)$ with respect to the induced metric on $H(\x,\y)$ 
agrees with parallel transportation of $v$ in the ambient manifold $M$ \citep[Lemma 8.5]{Lee-1997}. 
Thus, $n_t$ is precisely the direction that is reversed at $\gamma(t)$ by $f_{\x,\y}$.

Equation (\ref{eq:couplingfundeqn}) gives us 
\begin{equation}\label{eq:eqone}
\langle \mathbf{b}(\gamma(t)), n_t\rangle \quad=\quad \frac{1}{2}\langle \kappa(\gamma(t)), n_t\rangle\,.
\end{equation}
Differentiating the above twice with respect to $t$ along the geodesic \(\gamma\), and using the fact that \(\nabla_{\dot{\gamma}(t)}n_t=0\) because \(n_t\) was defined using parallel transport along \(\gamma\), we obtain 
$$
\langle D_t^2\mathbf{b}(\gamma(t)), n_t\rangle \quad=\quad \frac{1}{2}\langle D_t^2\kappa(\gamma(t)), n_t\rangle 
$$
(using \(D_t\) as shorthand for covariant differentiation \(\nabla_{\dot{\gamma}}\) along the geodesic \(\gamma\))
which, along with (\ref{eq:jacobi}) and (\ref{eq:constcurvendo}), gives
\begin{equation}\label{eq:bdiffone}
\frac{\d^2}{\d t^2}\langle \mathbf{b}(\gamma(t)), n_t\rangle + \frac{K}{2}\langle \kappa(\gamma(t)), n_t\rangle\quad=\quad0\,.
\end{equation}
Consequently equation \eqref{eq:eqone} shows that the function $t \mapsto \langle \mathbf{b}(\gamma(t)), n_t\rangle$ satisfies the following differential equation
\begin{equation}\label{eq:bdifftwo}
\frac{\d^2}{\d t^2}\langle \mathbf{b}(\gamma(t)), n_t\rangle + K\langle \mathbf{b}(\gamma(t)), n_t\rangle\quad=\quad 0\,.
\end{equation}
For any geodesic $\gamma$ passing through $\x^*$, not necessarily lying in $H(\x,\y)$, and for any parallel vectorfield $V_t$ along $\gamma$ orthogonal to $\dot{\gamma}_t$, 
a similar technique
uses \eqref{eq:couplingfundeqn}, \eqref{eq:jacobi} and \eqref{eq:constcurvendo} to
give us
\begin{equation}\label{eq:bdiffthree}
\frac{\d^2}{\d t^2}\langle \mathbf{b}(\gamma(t)), V_t\rangle + K\langle \mathbf{b}(\gamma(t)), V_t\rangle 
\quad=\quad
\frac{\d^2}{\d t^2}\langle \mathbf{b}(f_{\x,\y}\circ\gamma(t)), \d f_{\x,\y}(V_t)\rangle + K\langle \mathbf{b}(f_{\x,\y}\circ\gamma(t)), \d f_{\x,\y}(V_t)\rangle\,.
\end{equation}
Now, following the lines of the proof of Lemma \ref{lem:isotropic},
we can iteratively compose the isometries in  
$$
\mathcal{S}\quad=\quad
\left\lbrace f_{\x,\y} \in \mathcal{G} \;:\; \x \in \ball(\x_0,r)\,, \y \in \ball(\y_0,r)\,, \dist(\x,\x^*)=\dist(\y,\x^*)=\frac{1}{2}\dist(\x,\y) \right \rbrace
$$
to deduce that the closed subgroup of isometries $\mathcal{G}^*$ generated by $\mathcal{S}$ is the whole isotropy group of $\x^*$ in $\mathcal{G}$. Further, from Step 1 and Step 2 in the proof of Lemma \ref{lem:isotropic}, it can be seen that for any pair of linearly independent unit vectors $u, v \in T_{\x^*}M$, there is a sequence of isometries $\{F_k\}_{k \ge 1}$ such that for each $k$, $F_k$ is a composition of isometries in $\mathcal{S}$, $\d F_k$ fixes vectors in $T_{\x^*}M$ that are orthogonal to $\{u,v\}$, and $\d F_k(u) \rightarrow v$ as $k \rightarrow \infty$.

Take any geodesic $\gamma$ issuing from $\x^*$ and lying in $H(\x,\y)$ for some $\x \in \ball(\x_0,r)$, $\y \in \ball(\y_0,r)$ and let $n_t$ denote the parallel vectorfield along $\gamma$ that is inverted by $f_{\x,\y}$. Let $G \in \mathcal{G}$ be a composition of isometries in $\mathcal{S}$ which fix $\gamma$ and let $v=\d G(n_0)$. Let $V^v_t$ denote the parallel transport of $v$ along $\gamma$. As $G$ is an isometry, \citep[Proposition 5.6 (b)]{Lee-1997} implies $G_*n_t=V^v_t$. Applying (\ref{eq:bdiffthree}) at each composition corresponding to $G$, we get
\begin{equation}\label{eq:bdifffour}
\frac{\d^2}{\d t^2}\langle \mathbf{b}(\gamma(t)), n_t\rangle + K\langle \mathbf{b}(\gamma(t)), n_t\rangle 
\quad=\quad
\frac{\d^2}{\d t^2}\langle \mathbf{b}(\gamma(t)), V^v_t\rangle + K\langle \mathbf{b}(\gamma(t)), V^v_t\rangle\,.
\end{equation}
By (\ref{eq:bdifftwo}), the left hand side of the above is zero. 
Thus, the right hand side should vanish too. Solving this gives (\ref{eqnarray:antsymm}) with $V^v$ in place of $V$ and the given matrix $T$.

Now, consider any parallel vectorfield $V_t$ along $\gamma$ which is orthogonal to $\dot{\gamma}_t$. By the discussion following the definition of $\mathcal{S}$, there exists a sequence of isometries $\{F_k\}_{k \ge 1}$ such that each $F_k$ is a composition of isometries in $\mathcal{S}$, $F_k$ fixes $\gamma$, and $\d F_k(n_0) \rightarrow V_0$ as $k \rightarrow \infty$. As $F_k$ fixes $\x^*$ for each $k$, by \citep[p. 7]{MyersSteenrod-1939}, we can choose a subsequence $k_l$ such that $F_{k_l} \rightarrow F$ in $\mathcal{G}$ as $l \rightarrow \infty$. Write $V^{(k)}_t= F_{k*}n_t$. By \citep[Lemma 4]{MyersSteenrod-1939}, for each $t \ge 0$, $V^{(k_l)}_t \rightarrow \d F(n_t)$ in $T_{\gamma(t)}M$ as $l \rightarrow \infty$. In particular, $\d F(n_0) = V_0$, and as $F$ is an isometry fixing $\gamma$, $\d F(n_t) = V_t$ for all $t \ge 0$. Thus, we have $V^{(k_l)}_t \rightarrow V_t$ in $T_{\gamma(t)}M$ for each $t \ge 0$. From the discussion in the previous paragraph, (\ref{eqnarray:antsymm}) holds with $V^{(k_l)}$ in place of $V$ for each $l \ge 1$. Taking $l \rightarrow \infty$, we obtain (\ref{eqnarray:antsymm}) for the vectorfield $V$.

Finally, take any pair of unit vectors $u, v \in T_{\x^*}M$ satisfying $u \perp v$. Let $\sigma$ be the geodesic issuing from $\x^*$ such that $\dot{\sigma}(0)=u$. We can obtain a sequence of isometries $\{G_k\}_{k \ge 1}$ such that each $G_k$ is a composition of isometries in $\mathcal{S}$ and $\d G_k(\dot{\gamma}(0)) \rightarrow u$ as $k \rightarrow \infty$. Write $u_k=\d G_k(\dot{\gamma}(0))$ and let $\sigma_k$ be the geodesic issuing from $\x^*$ in the direction $u_k$. Denote by $V^{v,k}_t$ and $V^v_t$ the parallel transport of $v$ along $\sigma_k$ and $\sigma$ respectively. By the previous discussion, we know that (\ref{eqnarray:antsymm}) holds with $V^{v,k}$ in place of $V$ and $\sigma_k$ in place of $\gamma$ for each $k \ge 1$. Observe that for each fixed $t \ge0$, both sides of (\ref{eqnarray:antsymm}) depend continuously on $u$ and $v$ (this observation for the left hand side follows from the fact that the solution to the geodesic and parallel transport equations depends continuously on the initial data). Thus, we can take $k \rightarrow \infty$ to get (\ref{eqnarray:antsymm}) with $V^v$ in place of $V$ and $\sigma$ in place of $\gamma$.

The fact that $T$ is skew-symmetric follows from the observation that $S_{\x^*}(e_i,e_j)=\lambda\langle e_i, e_j\rangle$ (by Lemma \ref{lem:symmclass}) and therefore
\begin{equation*}
\langle\nabla_{e_i}\mathbf{b}(\x^*),e_j\rangle - \lambda\langle e_i, e_j\rangle \quad=\quad
\frac{1}{2}\left(\langle\nabla_{e_i}\mathbf{b}(\x^*),e_j\rangle - \langle \nabla_{e_j}\mathbf{b}(\x^*),e_i\rangle\right)\,.
\end{equation*}
\end{proof}
Since $M$ is a maximally symmetric space (by Theorem \ref{thm:Mclass}), the dimension of its set of Killing vectorfields is $\frac{d(d+1)}{2}$. 
Thus, for any vector $w \in T_{x^*}M$ and any skew-symmetric matrix $T$, there exists a unique Killing vectorfield $\mathcal{K}$ with $\mathcal{K}(\x^*)=w$ and $\langle\nabla_{e_i}\mathcal{K}(\x^*),e_j\rangle=T_{ij}$. 
Moreover, as every Killing vectorfield is a Jacobi field (i.e.~satisfies (\ref{eq:jacobi})),
it follows that $\mathcal{K}$ satisfies the following equation analogous to (\ref{eqnarray:antsymm}), for unit vectors $u,v \in T_{\x^*}M$ with $u \perp v$.
\begin{eqnarray}\label{eqnarray:Killjacobi}
\langle \mathcal{K}(\gamma(t)), V_t\rangle\quad=\quad
\left\{
\begin{array}{lll}
\langle w,v \rangle \cos \sqrt{K}t + \displaystyle{\langle Tu,v\rangle \frac{\sin \sqrt{K}t}{\sqrt{K}}} & \mbox{if } K>0\,,\\
\langle w,v \rangle + \langle Tu,v\rangle t & \mbox{if } K=0\,,\\
\langle w,v \rangle \cosh \sqrt{-K}t + \displaystyle{\langle Tu,v\rangle \frac{\sinh \sqrt{-K}t}{\sqrt{-K}}} & \mbox{if } K<0\,.
\end{array}
\right.
\end{eqnarray}

Thus, if we set $\mathcal{K}_{\x^*}$ as the Killing vectorfield uniquely determined by $w=\mathbf{b}(\x^*)$ and $T_{ij}=\langle\nabla_{e_i}\mathbf{b}(\x^*),e_j\rangle - \lambda\langle e_i, e_j\rangle$, we see from Lemmas \ref{lem:symmclass} and \ref{lem:driftclass} that the vectorfield $\mathbf{b}$ can be written as
\begin{equation}\label{eq:driftdec}
\mathbf{b}\quad=\quad\mathcal{D}_{\x^*}^{\lambda} + \mathcal{K}_{\x^*}
\end{equation}
where $\mathcal{D}_{\x^*}^{\lambda}$ is the \emph{dilation vectorfield about} $\x^*$ with dilation coefficient $\lambda$ defined as
\begin{equation}\label{eq:dilation}
\mathcal{D}_{\x^*}^{\lambda}(\gamma(t))\quad=\quad\lambda t \, \dot{\gamma}(t)
\end{equation}
for any geodesic $\gamma$ issuing from $\x^*$. Now, we claim that dilation vectorfields do not arise in the case of non-zero-curvature.
\begin{lem}\label{lem:lambdazero}
$K \neq 0$ implies $\lambda=0$.
\end{lem}
\begin{proof}
Under {\LPC}, the description of $\mathbf{b}$ given in Lemma \ref{lem:driftclass} holds for $\x^*$ replaced by $\hat{x} \in \ball(\x^*,\rho)$ for some $\rho>0$. 
Take any two points $\x_1,\x_2 \in \ball(\x^*,\rho)$ with $\x_1 \neq \x_2$. 
Lemmas \ref{lem:symmclass} and \ref{lem:driftclass}, applied at $\x_1$ and $\x_2$, show that $\mathbf{b}$ satisfies
\begin{equation}\label{eq:drifttwopts}
\mathbf{b}\quad=\quad\mathcal{D}_1^{\lambda} + \mathcal{K}_1\quad=\quad\mathcal{D}_2^{\lambda} + \mathcal{K}_2
\end{equation}
where $\mathcal{K}_1$ and $\mathcal{K}_2$ are Killing vectorfields and $D_1^{\lambda}$ and $D_2^{\lambda}$ are dilation vectorfields with the same coefficient $\lambda$ about $\x_1$ and $\x_2$ respectively.

Denote by $\sigma$ the geodesic issuing from $\x_2$ and passing through $\x_1$, and set
$\gamma$ to be a geodesic issuing from $\x_2$ in a direction orthogonal to $\sigma$. 
Locate $\z=\gamma(\dist(\x_1,\x_2))$. Taking $\rho$ sufficiently small, we can ensure that $\gamma$ restricted to $[0, \dist(\x_1,\x_2)]$ is a minimal geodesic from $\x_2$ to $\z$. Finally, denote the geodesic issuing from $\x_1$ and passing through $\z$ by $\eta$. Consider the geodesic triangle $\Delta$ formed by $\x_1$, $\x_2$ and $\z$. Thus, the sides of $\Delta$ are formed by the geodesics $\sigma$, $\gamma$ and $\eta$.\\\\
Now, recall that the curvature $K$ can also be interpreted in terms of the rate at which geodesics diverge  when they issue from a point in different directions. 
Thus \citep[Proposition 2.6]{Maubon-2004} we see that if $\x_1$ is taken sufficiently close to $\x_2$, then
\begin{eqnarray}\label{eqnarray:lengthcurv}
\dist(\x_1,\z)&\quad<\quad&\sqrt{2}\dist(\x_1,\x_2) \ \text{if } K>0\,,\nonumber\\
\dist(\x_1,\z)&\quad>\quad&\sqrt{2}\dist(\x_1,\x_2) \ \text{if } K<0\,.
\end{eqnarray}
Applying the triangle version of the Toponogov comparison theorem \citep[Theorem 79, p.~339]{Petersen-2006},
we see that the interior angle $\theta$ formed at the vertex $\z$ of $\Delta$ satisfies $\theta \ge \pi/4$ if $K>0$ and $\theta \le \pi/4$ if $K<0$.
But \eqref{eq:dilation} implies
\begin{eqnarray*}
\langle D_1^{\lambda}(\z),\dot{\gamma}(\dist(\x_1,\x_2))\rangle\quad&=\quad& \langle D_1^{\lambda}(\z),\dot{\eta}(\dist(\x_1,\z))\rangle \cos \theta\\
\quad&=\quad& \lambda \dist(\x_1,\z) \cos \theta\,.
\end{eqnarray*}
Thus, if $\lambda>0$, we get
\begin{eqnarray}\label{eqnarray:inncurv}
\langle D_1^{\lambda}(\z),\dot{\gamma}(\dist(\x_1,\x_2))\rangle&\quad<\quad&\lambda \dist(\x_1,\x_2) \ \text{if } K>0\,,\nonumber\\
\langle D_1^{\lambda}(\z),\dot{\gamma}(\dist(\x_1,\x_2))\rangle&\quad>\quad&\lambda \dist(\x_1,\x_2) \ \text{if } K<0\,.
\end{eqnarray}
and the inequalities are reversed if $\lambda<0$.

From \eqref{eq:drifttwopts}
\begin{equation}\label{eq:dilcomp}
\langle D_2^{\lambda}(\z),\dot{\gamma}(\dist(\x_1,\x_2))\rangle\quad=\quad
\langle D_1^{\lambda}(\z),\dot{\gamma}(\dist(\x_1,\x_2))\rangle + \langle (\mathcal{K}_1-\mathcal{K}_2)(\z),\dot{\gamma}(\dist(\x_1,\x_2))\rangle\,.
\end{equation}
Lemma \ref{lem:Killing} implies that the inner product of a Killing vectorfield with the velocity vector of a geodesic is conserved along the geodesic, yielding 
$$
\langle (\mathcal{K}_1-\mathcal{K}_2)(\z),\dot{\gamma}(\dist(\x_1,\x_2))\rangle\quad=\quad
\langle (\mathcal{K}_1-\mathcal{K}_2)(\x_2),\dot{\gamma}(0)\rangle\,.
$$
From \eqref{eq:dilation} it follows that $D_2^{\lambda}(\x_2)=0$ and also
$$
\langle D_1^{\lambda}(\x_2), \dot{\gamma}(0)\rangle\quad=\quad
\lambda\dist(\x_1,\x_2) \,\langle \dot{\sigma}(0),\dot{\gamma}(0)\rangle
\quad=\quad0\,.
$$
Combining this with \eqref{eq:drifttwopts},
$$
\langle (\mathcal{K}_1-\mathcal{K}_2)(\x_2),\dot{\gamma}(0)\rangle\quad=\quad0\,.
$$

Thus, (\ref{eq:dilcomp}) gives us 
$$
\langle D_2^{\lambda}(\z),\dot{\gamma}(\dist(\x_1,\x_2))\rangle\quad=\quad
\langle D_1^{\lambda}(\z),\dot{\gamma}(\dist(\x_1,\x_2))\rangle\,.
$$

By \eqref{eq:dilation}, $\langle D_2^{\lambda}(\z),\dot{\gamma}(\dist(\x_1,\x_2))\rangle=\lambda \dist(\x_1,\x_2)$. 
Together with \eqref{eqnarray:inncurv}, this forces $\lambda = 0$ if the curvature is non-zero, hence proving the lemma.
\end{proof}
\textbf{Note: }When $K>0$, observe that 
$$
\langle \mathbf{b}(\gamma(0)), \dot{\gamma}_0\rangle\quad=\quad\left\langle \mathbf{b}(\gamma(2\pi/\sqrt{K})), \dot{\gamma}_{2\pi/\sqrt{K}}\right\rangle
$$
 yields $\lambda=0$. 
 But the above proof works for both positive and negative curvatures, and is in some sense, the real geometric reason why the dilation part of the vectorfield $\mathbf{b}$ vanishes for non-zero curvature.
 
Finally we can state and prove the main theorem of this section.
\begin{thm}\label{thm:drift}
Suppose that the standing assumptions of diffusion-geodesic completeness and stochastic completeness both hold.
The drift vectorfield $\mathbf{b}$ permits \textbf{MMC} with {\LPC} if and only if both of the following hold:
\begin{itemize}
\item[(i)] The underlying Riemannian manifold $M$ is one of the three model spaces $\mathbb{S}^d$ $(K>0)$, $\Reals^d$ $(K=0)$ or $\mathbb{H}^d$ $(K<0)$,
in the sense that the diffusion must be expressible as Riemannian Brownian motion plus drift vectorfield \(\mathbf{b}\) for such an \(M\).
\item[(ii)] For $K \neq 0$, the drift $\mathbf{b}$ must and can be any Killing vectorfield $\mathcal{K}$ on $M$. 
For $K=0$, the drift $\mathbf{b}$ must and can be described in Euclidean coordinates by $\mathbf{b}(\x)=\lambda \x +T\x + \mathbf{c}$ for any scalar $\lambda$, any skew-symmetric matrix $T$ and any vector $\mathbf{c}$,
where $\x \mapsto \lambda \x$ is a dilation vectorfield about the origin and $\x \mapsto T\x + \mathbf{c}$ is a Killing vectorfield.
\end{itemize}
\end{thm}
\begin{proof}
The classification of the space $M$ is essentially the content of Theorem \ref{thm:Mclass}. 
Lemmas \ref{lem:driftclass} and \ref{lem:lambdazero} show that if {\LPC} holds then the drift vectorfield $\mathbf{b}$ has to be of the form described in the theorem. 
For the case $K=0$, section \ref{sec:Euclidean} shows the existence of a Markovian maximal coupling with any pair of starting points $\x \in \ball(\x_0,r)$ and $\y \in \ball(\y_0,r)$ and fully describes the coupling.

To show existence and to describe the coupling for $K \neq 0$, recall that any Killing vectorfield $\mathcal{K}$ generates a one-parameter subgroup of isometries 
starting from the identity, say $(\Upsilon_t : t \in \Reals)$. 
Let $Z$ denote a Brownian motion on $M$, and consider the law of 
$$
X_t\quad=\quad\Upsilon_t(Z_t)\,.
$$
 Consider the lift $U$ of the Brownian motion $Z$ onto the orthonormal frame bundle $\mathcal{O}(M)$. 
 Recall that the Stratonovich stochastic differential equation for this lifted process is given by
\begin{equation}\label{eq:stratkillU}
\d U_t\quad=\quad \sum_i H_i(U_t) \circ \d W^i_t
\end{equation}
where $W=(W^1,\dots, W^d)$ is a \(d\)-dimensional Euclidean Brownian motion. 
The process $Z$ is recovered from $U$ by $Z=\pi(U)$. 
Recall that the lift of an isometry $F$ on $M$ to $\hat{F}$ on $\mathcal{O}(M)$ is given by (\ref{eq:isolift}). 
Defining the process $V$ on $\mathcal{O}(M)$ by
\begin{equation}\label{eq:liftmap}
V_t\quad=\quad \hat{\Upsilon}_t(U_t)
\end{equation}
the arguments used to derive (\ref{eq:stochastic differential equationV}) also show that the Stratonovich stochastic differential equation for $V$ is given by
\begin{equation}\label{eq:stratkillV}
\d V_t\quad=\quad 
\sum_i H_i(V_t) \circ \d W^i_t + \hat{\mathcal{K}}_t(\hat{\Upsilon}_t^{-1}(V_t))\d t
\end{equation}
where 
$$
\hat{\mathcal{K}}_t(u)\quad=\quad
\left.\frac{\d}{\d s}\right|_{s=t}\hat{\Upsilon}_s(u)
$$
 for $u \in \mathcal{O}(M)$. Note that, for any $\x \in M$,  
$$
\left.\frac{\d}{\d s}\right|_{s=t}\Upsilon_s(\x)
\quad=\quad
\mathcal{K}(\Upsilon_t(\x))\,.
$$
 Using this, and the fact that $\pi(V_t)=X_t$, we see that
\begin{eqnarray*}
\d X_t\quad&=\quad&\sum_i(\d\pi(H_i(V_t))) \circ \d W^i_t + (\d\pi(\hat{\mathcal{K}}_t(\hat{\Upsilon}_t^{-1}(V_t))))\d t\\
\quad&=\quad&\sum_i V_te_i \circ \d W^i_t + \mathcal{K}(X_t) \d t
\end{eqnarray*}
which demonstrates that \(X\) is a Riemannian Brownian motion with drfit vectorfield given by the Killing vectorfield $\mathcal{K}$.

As discussed in \citet[Example 6.1]{Kuwada-2009} and references therein, 
if $M$ is $\mathbb{S}^d$ or $\mathbb{H}^d$ then
there exists a Markovian maximal coupling $(Z, \widetilde{Z})$ of Brownian motions starting from any two distinct points on $M$. 
Consider a diffusion representable as Riemannian Brownian motion with drift given by any Killing vectorfield $\mathcal{K}$ on such a manifold $M$.
Thus Lemma \ref{lem:mapsto} implies that a Markovian maximal coupling for this diffusion exists between any pair of starting points, and can be constructed by 
$$
\left((\Upsilon_t(Z_t), \Upsilon_t(\widetilde{Z}_t)\;:\;t \ge 0\right)
$$
 where $(\Upsilon_t:t \in\Reals)$ is the one-parameter subgroup of isometries starting from the identity which is generated by the Killing vectorfield $\mathcal{K}$. This proves the theorem.
\end{proof}
\begin{cor}\label{cor:mirrorevolve}
Under the hypothesis of part (ii) of Theorem \ref{thm:drift}, let $(\Upsilon_t : t \in \Reals)$ denote the one-parameter subgroup of isometries corresponding to the Killing vectorfield $\mathcal{K}$. Then for $t \ge 0$, the mirror $H_t$ and the corresponding reflection isometries $F_t$ satisfy $H_t=\Upsilon_t(H_0)$ and $F_t=\Upsilon_t \circ F_0 \circ \Upsilon_t^{-1}$.
\end{cor}
\begin{proof}
Let $Z, \widetilde{Z}$ be maximally coupled Brownian motions on $M$. For any $t \ge 0$, by Remark \ref{rem:bzero}, $H_0=H(Z_t, \widetilde{Z}_t)$ almost surely. By Theorem \ref{thm:zerosetman}, $H_t=H(\Upsilon(Z_t), \Upsilon(\widetilde{Z}_t))$ almost surely. From this, $H_t=\Upsilon_t(H_0)$ easily follows. Further, as $F_t$ and $\Upsilon_t \circ F_0 \circ \Upsilon_t^{-1}$ have the same set of fixed points, namely $H_t$, and neither of them is the identity, therefore $F_t=\Upsilon_t \circ F_0 \circ \Upsilon_t^{-1}$ follows from uniqueness of isometry with fixed point set $H_t$.
\end{proof}
In the following theorem, we characterise the class of drifts $\mathbf{b}$ and starting points $\x_0, \y_0$ for which the interface $I(\x_0,\y_0,t)$ does not depend on time $t$.
\begin{thm}\label{thm:mirrorstat}
Suppose that the standing assumptions of diffusion-geodesic completeness and stochastic completeness both hold.
Suppose the drift vectorfield $\mathbf{b}$ permits \textbf{MMC} with {\LPC}. Let $I(\x_0,\y_0,t)$ denote the interface for the \textbf{MMC} $(X,Y)$ of diffusions $X$ and $Y$ starting from $\x_0$ and $\y_0$ respectively. Then $I(\x_0,\y_0,t) = I(\x_0,\y_0,0)$ for all $t \ge 0$ if and only if one of the following holds:
\begin{itemize}
\item[(i)] $K=0$, $\mathbf{b}(\x)=\lambda \x +T\x + \mathbf{c}$ for some scalar $\lambda$, skew-symmetric matrix $T$ and vector $\mathbf{c}$, and $\x_0,\y_0,\lambda, T, \mathbf{c}$ satisfy $T(\x_0-\y_0)=0$ and $(\x_0-\y_0)^\top(\lambda(\x_0+\y_0) + 2\mathbf{c})=0$.
\item[(ii)] $K \neq 0$ and $\mathbf{b}$ is a Killing vectorfield $\mathcal{K}$ on $M$ which satisfies the following: if $\x^*$ is the midpoint of a minimal geodesic joining $\x_0$ and $\y_0$ and $\mathbf{n}$ is the vector normal to the hypersurface $H(\x_0,\y_0)$ at $\x^*$, then $\langle \mathcal{K}(\x^*), \mathbf{n}\rangle = 0$ and $\nabla_{n}\mathcal{K}(\x^*)=0$.
\end{itemize}
\end{thm}
\begin{proof}
When $K=0$, we observe from \eqref{eq:nform} that $\mathbf{n}(t)=\mathbf{n}(0)$ for all $t \ge 0$ if and only if $T(\x_0-\y_0)=0$. Using this in \eqref{eq:lform}, we get for $\lambda \neq 0$,
$$
l(t)=l(0)e^{\lambda t} + \frac{\mathbf{n}(0)^\top\mathbf{c}}{\lambda} (e^{\lambda t}-1)=e^{\lambda t}\left(l(0) + \frac{\mathbf{n}(0)^\top\mathbf{c}}{\lambda}\right) - \frac{\mathbf{n}(0)^\top\mathbf{c}}{\lambda}.
$$
Thus $l(t)=l(0)$ for all $t \ge 0$ if and only if $l(0) + \frac{\mathbf{n}(0)^\top\mathbf{c}}{\lambda}=0$. Substituting $l(0)=\frac{|\x_0|^2-|\y_0|^2}{2|\x_0-\y_0|}$ and $\mathbf{n}(0)=\frac{\x_0-\y_0}{|\x_0-\y_0|}$ in this equation, we get $(\x_0-\y_0)^\top(\lambda(\x_0+\y_0) + 2\mathbf{c})=0$.

When $\lambda=0$, we get $l(t)=l(0) + t (\mathbf{n}(0)^\top\mathbf{c})$. Thus $l(t)=l(0)$ for all $t \ge 0$ if and only if $(\x_0-\y_0)^\top\mathbf{c}=0$.

Now, suppose $K \neq 0$ and $\mathbf{b}$ is the Killing vectorfield $\mathcal{K}$ on $M$. As there is at most one isometry whose fixed point set is $H(\x_0,\y_0)$, we deduce that $I(\x_0,\y_0,t) = I(\x_0,\y_0,0)$ for all $t \ge 0$ if and only if $F_t = F$ for all $t \ge 0$.

Suppose $F_t = F$ for all $t \ge 0$. Then by \eqref{eq:couplingfundeqn}, $\mathcal{K}(\x)= F_*\mathcal{K}(\x)$ for all $\x \in M$. In particular, $\langle \mathcal{K}(\x^*), \mathbf{n} \rangle = \langle F_*\mathcal{K}(\x^*), \mathbf{n} \rangle$. But, as $F$ is an involutive isometry, $\langle F_*\mathcal{K}(\x^*), \mathbf{n} \rangle = \langle \mathcal{K}(\x^*), F_*\mathbf{n} \rangle= \langle \mathcal{K}(\x^*), -\mathbf{n} \rangle$ from which we get $\langle \mathcal{K}(\x^*), \mathbf{n} \rangle=0$. Now, observe that as $\mathcal{K}$ is a Killing vectorfield, therefore by Lemma \ref{lem:Killing}, $\langle \nabla_{\mathbf{n}}\mathcal{K}(\x^*), \mathbf{n} \rangle=0$. If $u \in T_{\x^*}M$ is orthogonal to $\mathbf{n}$, then
$$
\langle \nabla_{\mathbf{n}}\mathcal{K}(\x^*), u \rangle = \langle \nabla_{\mathbf{n}}F_*\mathcal{K}(\x^*), u \rangle=\langle F_*\nabla_{-\mathbf{n}}\mathcal{K}(\x^*), u \rangle= \langle \nabla_{-\mathbf{n}}\mathcal{K}(\x^*), F_*u \rangle = \langle -\nabla_{\mathbf{n}}\mathcal{K}(\x^*), u \rangle
$$
which gives $\langle \nabla_{\mathbf{n}}\mathcal{K}(\x^*), u \rangle=0$. Hence, $\langle \nabla_{\mathbf{n}}\mathcal{K}(\x^*), u \rangle=0$ for all $u \in T_{\x^*}M$, and therefore, $\nabla_{\mathbf{n}}\mathcal{K}(\x^*)=0$.

Conversely, suppose $\langle \mathcal{K}(\x^*), \mathbf{n}\rangle = 0$ and $\nabla_{n}\mathcal{K}(\x^*)=0$ holds. Let $\gamma$ be any geodesic issuing from $\x^*$ and lying in $H(\x_0,\y_0)$ and let $n_t$ denote the parallel transport of $\mathbf{n}$ along $\gamma$. As $\langle \mathcal{K}(\x^*), \mathbf{n}\rangle = 0$ and $\langle\nabla_{\dot{\gamma}(0)}\mathcal{K}(\x^*), \mathbf{n}\rangle = -\langle\nabla_{\mathbf{n}}\mathcal{K}(\x^*), \dot{\gamma}(0)\rangle = 0$, using the representation \eqref{eqnarray:Killjacobi} for $\mathcal{K}$, we see that $\langle\mathcal{K}(\gamma(t), n_t\rangle = 0$ and hence, $\mathcal{K}(\gamma(t)) \in T_{\gamma(t)}H(\x_0,\y_0)$ for all $t \ge 0$. As the submanifold $H(\x_0,\y_0)$ is a geodesic space, we conclude that $\mathcal{K}$ restricted to $H(\x_0,\y_0)$ is a vectorfield tangent to this submanifold. Thus, if $\Upsilon_t$ denotes the flow of isometries generated by $\mathcal{K}$, then for each $\z_0 \in H(\x_0,\y_0)$, $\Upsilon_t(\z_0)$ lies in $H(\x_0,\y_0)$ at least for a short time. As $\Upsilon_t$ is a global flow (because $M$ is complete), a routine compactness argument implies that $\Upsilon_t(\z_0) \in H(\x_0,\y_0)$ for all $t \ge 0$. Thus, by Corollary \ref{cor:mirrorevolve}, $H_t \subseteq H(\x_0,\y_0)$, and hence $F_t = F$, for all $t \ge 0$.
\end{proof}
\section{Conclusion}\label{sec:conclusion}
In this paper we have shown that Markovian maximal couplings of regular elliptic diffusions with smooth coefficients 
(and satisfying diffusion-geodesic completeness and stochastic completeness)
have to be reflection couplings tied to involutive isometries of the corresponding Riemannian structure on state space;
moreover as soon as the existence of a Markovian maximal coupling is stable (in the sense of {\LPC}) then a rigidity result requires the Riemannian structure to be Euclidean, hyperspherical, or hyperbolic, and the space must be simply connected.
In such cases the drift must also be of a very simple form, corresponding to a rotation with possibly (but only in the Euclidean case) a dilation component.

Thus Markovian maximal couplings of elliptic diffusions are rare, and their existence enforces severe geometric constraints.

It is natural to ask whether the assumptions of diffusion-geodesic completeness and stochastic completeness are required. 
It seems likely that they are not required, but (this paper already being long) we save this question for another occasion.

The scarcity of Markovian maximal couplings places a natural premium on questions of efficiency of Markovian coupling, as discussed for example in \citet{BurdzyKendall-2000}, for the case of reflecting Brownian motion in compact regions. 
One could ask, for example, when it is possible to construct Markovian couplings \((X, Y)\) which are optimal in the sense that 
the tail probability of the coupling time \(\Prob{\tau>t}\) is minimized for all \(t\) amongst Markovian couplings if not amongst all possible couplings.
(Note that this notion of optimality differs from the optimality discussed in \cite{Chen-2004}, which is defined relative to a specified Wasserstein metric.) 
Little is known as yet about such couplings, though \cite{Kendall-2013a} exhibits a coupling of two copies of scalar Brownian motion and local time which is Markovian, non-maximal,
but optimal amongst all Markovian couplings.
The question of whether similar geometric rigidity results for existence of such optimal Markovian couplings remains entirely open, and its answer would be of great interest.

We expect that in fact such optimal Markovian couplings are also rare. Further refinements are possible (for example, one could consider the existence of Markovian couplings 
which minimize the Laplace transform \(\Expect{\exp\left(-u\tau\right)}\) for some or all values of \(u>0\)); 
however the probable rarity of such couplings would focus attention on developing the notions of efficiency from \citet{BurdzyKendall-2000} to apply to non-compact regions.
In particular there is a natural question concerning criteria for existence of efficient Markovian couplings, where ``efficient'' here means, the rate of decay 
of \(\Prob{\tau>t}\) with \(t\) for the Markovian coupling is comparable to that of the total variation distance \(\|\mu_{1,t}-\mu_{2,t}\|_{TV}\) between the one-point distributions 
$\mu_{1,t}$ and $\mu_{2,t}$ (the distributions of $X_t$ and $Y_t$ respectively).

Two other natural extensions of these results are:
\begin{enumerate}
 \item extension of the notion of Markovian maximal coupling to the hypoelliptic case (in which case in fact the very existence of Markovian couplings is moot: but see the positive results of \citealp{KendallPrice-2004,Kendall-2007});
 \item examination of the extent to which the ideas of this paper carry over to Markov processes which are not skip-free (and here a natural first step would be to consider the case of couplings of L\'evy processes, 
 though a potentially significant result in the random walk case is to be found in \citealp{Rogers-1999}).
\end{enumerate}
We hope to consider many of these questions in future work.\\\\
\textbf{Acknowledgement: }We wish to thank an anonymous referee whose very careful reading of the manuscript and detailed comments greatly improved the article.
\bibliographystyle{chicago}
\bibliography{rigidity}
\bigskip
\bigskip
\noindent
{\sc Sayan Banerjee}, {\em University of Warwick}, \url{Sayan.Banerjee@warwick.ac.uk} \\
{\sc Wilfrid S. Kendall}, {\em University of Warwick}, \url{W.S.Kendall@warwick.ac.uk}
\end{document}